\documentclass[twoside,a4paper,11pt,limits]{amsart}
\usepackage{graphicx}
\usepackage{color}
\usepackage{latexsym}
\usepackage{amssymb}
\usepackage{amsfonts}
\usepackage{amsmath}
\usepackage[rose,frak,hyperref,operator]{paper_diening}

\allowdisplaybreaks

\newtheorem{theorem}{Theorem}[section]
\newtheorem{lemma}[theorem]{Lemma}
\newtheorem{corollary}[theorem]{Corollary}
\newtheorem{proposition}[theorem]{Proposition}
\newtheorem{remark}[theorem]{Remark}

\newtheorem{example}[theorem]{Example}

\numberwithin{equation}{section}
\newcommand{\meantmp}[2]{#1\langle{#2}#1\rangle}
\newcommand{\mean}[1]{\meantmp{}{#1}}

\newcommand{\Rn}{{\setR^n}}
\newcommand{\RN}{{\setR^N}}
\newcommand{\RNn}{\setR^{N \times n}}

\newcommand{\dt}{\ensuremath{\,{\rm d} t}}
\newcommand{\ds}{\ensuremath{\,{\rm d} s}}

\newcommand{\dx}{\ensuremath{\,{\rm d} x}}
\newcommand{\dy}{\ensuremath{\,{\rm d} y}}
\newcommand{\dz}{\ensuremath{\,{\rm d} z}}

\renewcommand{\setBMO}{\rm{BMO}}

\newcommand{\comment}[1]{\vskip.3cm
\fbox{%
\parbox{0.93\linewidth}{\footnotesize #1}}
\vskip.3cm}

\newcommand{\tuomo}[1]{%
\textcolor[rgb]{1,0,0}{*}\textcolor[rgb]{0,1,0}{*}\textcolor[rgb]{0,0,1}{*}%
\textcolor[rgb]{1,0,0}{*}\textcolor[rgb]{0,1,0}{*}\textcolor[rgb]{0,0,1}{*}
\textcolor[rgb]{0.2,0.6,0.6}{Tuomo: #1}
\textcolor[rgb]{1,0,0}{*}\textcolor[rgb]{0,1,0}{*}\textcolor[rgb]{0,0,1}{*}%
\textcolor[rgb]{1,0,0}{*}\textcolor[rgb]{0,1,0}{*}\textcolor[rgb]{0,0,1}{*}}
\newcommand{\seb}[1]{{\textcolor[rgb]{1.00,0.30,0.00}{ #1}}}

\begin{document}

\title{Pointwise Calder\'on-Zygmund gradient estimates for the $p$-Laplace system}

\begin{abstract}{Pointwise estimates for the gradient of solutions to the
$p$-Laplace  system with right-hand side in divergence form
are established. They enable us to develop a nonlinear counterpart
of the classical Calder\'on-Zygmund theory
for the Laplacian. As a consequence, a flexible, comprehensive
approach to gradient bounds for the $p$-Laplace system for a broad
class of norms is derived. In particular, new gradient estimates are
exhibited, and well-known results in customary function spaces are
easily recovered.
%
%
%
}
\end{abstract}
\author{D.~Breit, A.~Cianchi, L.~Diening, T.~Kuusi and S.~Schwarzacher}
\address{Dominic Breit,
School of Mathematical \& Computer Science, Heriot-Watt University,
Riccarton Edinburgh EH14 4AS UK} \email{d.breit@hw.ac.uk}
\address{Andrea Cianchi, Dipartimento di Matematica e Informatica \lq\lq U.Dini", Universit\`{a} di Firenze, Viale Morgagni 67/A, 50134, Firenze, Italy}
\email{cianchi@unifi.it}
\address{Lars Diening,
Universit\"at Osnabr\"uck, Institut f\"ur Mathematik, Albrechtstr.
28a, 49076 Osnabr\"uck } \email{ldiening@uni-osnabrueck.de}
\address{Tuomo Kuusi, Aalto University,
Department of Mathematics and Systems Analysis, P.O. Box 11100,
FI-00076 Aalto, Finland } \email{tuomo.kuusi@aalto.fi}
\address{Sebastian Schwarzacher, Department of mathematical analysis, Faculty of Mathematics and Physics,  Charles University in Prague,
Sokolovsk\'{a} 83, 186 75 Prague, Czech Republic}
\email{schwarz@karlin.mff.cuni.cz}

%
\thanks{
This research has been partly supported by: Research project of MIUR
(Italian Ministry of Education, University and Research) Prin 2012,
n. 2012TC7588,  ``Elliptic and parabolic partial differential
equations: geometric aspects, related inequalities, and
applications";  GNAMPA of the Italian INdAM (National Institute of
High Mathematics);  Program PRVOUK P47 at the Department of Analysis
at Charles University, Prague.} \maketitle


\section{Introduction and main results}\label{intro}

The present paper deals with  the  $p$-Laplace elliptic system
\begin{align}
\label{eq:sysA}
  -\divergence(|\nabla \bfu|^{p-2} \nabla \bfu) &= -\divergence
  \bfF\, \qquad \hbox{in $\Omega$.}
\end{align}
 Here, $\Omega$ is an open set in $\mathbb R^n$, with $n \geq 2$, the exponent $p \in (1, \infty)$, the
function $\bfF \,:\, \Omega \to \mathbb R^{N \times n}$, with $N
\geq 1$, is assigned, and $\bfu \,:\, \Omega \to \RN$ is the
unknown. The notation $\RNn$ stands for the space of $N \times n$
matrices.
\par
We are concerned with gradient estimates for local weak solutions
$\bfu$ to this system. Our purpose is to establish pointwise bounds
for $\nabla \bfu$, or more precisely,  for $|\nabla
\bfu|^{p-2}\nabla \bfu$, in terms of $\bfF$, which, in a suitable
sense, linearize the problem.
%
 This provides us with a powerful tool for a unified regularity theory of the gradient.
%
%
%
%
%
\par
A brief digression to the linear setting, corresponding to the
choice $p=2$, may help to grasp the spirit and novelty of our
contribution.  In this case, system \eqref{eq:sysA} reduces to
\begin{align} \label{eq:sysAlin}
  -\divergence(\nabla \bfu) &= -\divergence
  \bfF\,,
\end{align}
  namely,
  \begin{align}
\label{eq:sysAlap}
  - \Delta \bfu &= -\divergence
  \bfF\,.
\end{align}
 The classical Calder\'on-Zygmund theory  offers an exhaustive picture for gradient bounds in this framework. In
particular, it implies that the divergence operator can \lq\lq
almost" be
  canceled in \eqref{eq:sysAlin}.
  Indeed, assume, for simplicity, that $\Omega =
       \Rn$.
A standard representation formula, in terms of Riesz transforms,
tells us that, if $\bfu$ is the solution to \eqref{eq:sysAlap} under
suitable assumptions -- for instance $\bfF \in L^2(\mathbb R^n)$ and
$\nabla \bfu \in L^2(\mathbb R^n)$ -- then
%
 \begin{equation}\label{linear} \nabla \bfu = T(\bfF),
 \end{equation}
  where
$T$ is  a  Calder\'on-Zygmund singular integral operator.  The
operator $T$ is known to be bounded in any
       non-borderline
       function space. As a consequence, bounds for the norm of $\nabla \bfu
       $ in any such space
       via the same norm of $\bfF$
       immediately follows from \eqref{linear}.
A standard instance amounts to an estimate in Lebesgue spaces, which
reads
\begin{equation}\label{lapp}\|\nabla \bfu\|_{L^r(\Rn )} \leq C \|\bfF\|_{L^r(\Rn
  )}\end{equation}
  for every  $r \in (1, \infty)$. Here, and in what follows, $C$
  denotes a constant whose dependence will be specified whenever
  needed. Observe that, by contrast,  \eqref{lapp} fails
    if either
    $r=1$, or $r=\infty$.  Replacements for \eqref{lapp} in these
       endpoint cases are  known.
 For example,   the space of functions of
      bounded mean oscillation, denoted by $\setBMO(\Rn )$, is a well known substitute for   $L^\infty (\Rn)$. Actually, one has that
   $$\|\nabla \bfu\|_{\setBMO (\Rn )} \leq C \|\bfF\|_{\setBMO (\Rn )}.
  $$
 Also,   classically
   $$\|\nabla \bfu\|_{C^{\alpha}(\Rn)} \leq C \|\bfF\|_{C^\alpha (\Rn)},
  $$
  for $\alpha \in (0,1)$, where  $\|\cdot \|_{C^{\alpha}(\Rn )}$
  denotes the H\"older seminorm.

  \smallskip
\par
Let us now turn to the nonlinear case, corresponding to $p \neq 2$.
The beginning of a systematic study of the
    so-called
    nonlinear Calder\'on-Zygmund theory, associated with
    \eqref{eq:sysA},
 can be traced back to  \cite{Iwa83}.  In particular, in that paper it is  shown that, if
  $N=1$, $p \geq 2$ and $ r  \in [p,\infty)$, then
\begin{equation}\label{Iw1} \|\nabla \bfu\|_{L^r(\Rn )} \leq C \|\bfF\|_{L^{\frac
  r{p-1}}(\Rn )}^{\frac 1{p-1}},
  \end{equation}
or, equivalently,
    \begin{equation}\label{Iw2} \||\nabla \bfu|^{p-1}\|_{L^q(\Rn )} \leq C
      \|\bfF\|_{L^q(\Rn )}
      \end{equation}
      for $q \in [p',\infty)$. Estimate \eqref{Iw2}
        was extended to every $N\geq 1$
       and $p>1$ in
\cite{DiBMan93}, a contribution which is also devoted to the
  the inequality
   $$ \|\nabla \bfu\|_{\setBMO(\Rn )} \leq C \|\bfF\|_{\setBMO(\Rn )}^{\frac 1{p-1}},$$
  but only for $p \geq 2$. On the other hand,  the recent paper
  \cite{DieKapSch11}
%
%
      contains the inequality
   \begin{equation}\label{BMOrn} \||\nabla \bfu|^{p-2}\nabla \bfu\|_{\setBMO(\Rn )} \leq C
   \|\bfF\|_{\setBMO(\Rn )}
   \end{equation}
for every $N\geq 1$ and $p>1$.  In particular, this suggests that
  gradient bounds for solutions to  \eqref{eq:sysA} are suitably formulated in terms of the nonlinear
  expression $|\nabla \bfu|^{p-2}\nabla \bfu$.
\par\noindent As far as H\"older  regularity
 of the gradient of solutions to \eqref{eq:sysA} is concerned,   the scalar case ($N=1$)  was settled in
 \cite{Ura68}. The same result for systems ($N \geq 1$),
for $p \geq 2$,  goes back
 to the paper \cite{Uhl77}, in the homogeneous case when $\bfF =0$. Systems involving differential operators depending
  only on the length of the gradient are hence usually called with Uhlenbeck structure.   The
contribution \cite{Uhl77} was  extended to the situation when
$1<p<2$ in \cite{AceF89} and \cite{ChenDiBe}. In particular, the
latter paper includes the case of non-vanishing smooth $\bfF$. As is
well known, regularity of solutions to  nonlinear elliptic systems
is  a critical issue, and exhibits special features compared to the
case of a single equation. This has been demonstrated via several
counterexamples, including those from \cite{IwaMan89, Nec75, Nec96,
SVYa00, SVYa02}.
\par
The study of pointwise elliptic gradient regularity has received an
impulse from the papers \cite{Mingione2011} and \cite{DuzMing2011},
where Havin-Maz'ya-Wolff nonlinear potentials have been shown to
yield precise estimates for the gradient of local solutions to
nonlinear $p$-Laplacian type equations, but  with right-hand side in
non-divergence form. Enhancements  and extensions of these results,
involving classical Riesz potentials instead of nonlinear
potentials, are the object of a series of papers starting from
\cite{KuuMin13, KuuMin12univ}. Special pointwise bounds for the
gradient, which also hold for solutions to systems, can be found in
\cite{DuzMin10a, KuuMin14}. These   estimates for the gradient of
solutions were preceded by parallel results for the solutions
themselves obtained in  \cite{KilpMaly}. Estimates in rearrangement
invariant form for the gradient of solutions to boundary value
problems are established in \cite{ACMM, AFT, CiaMaz14rearr}.
\\
Besides the papers mentioned above, which are
 focused on pointwise results, in recent years gradient
regularity has been the object of a number of contributions on
elliptic equations and systems with different peculiarities. For
instance, results on elliptic problems involving differential
operators affected by weak regularity properties can be found in
\cite{Du02, Du04, Du11, KinZho99, KinZho01}. The papers
\cite{BreStrVer11, CiaFus, DieSV09, Fu11, Li94, Li05, Li14,
MarPap06} are concerned with operators governed by general growth
conditions. Recent results dealing with global gradient estimates
for boundary value problems, under minimal regularity assumptions on
the boundary of the ground domain, are the object of \cite{AdiPhuc,
BanLew14, CiaMaz11, CiaMaz14syst, Phuc14, Lew08, Lew12}.
%

\par
Our main results show that, as in the linear case
 $\nabla \bfu$ and $\bfF$ are linked
     through the
   (singular integral) linear operator $T$ appearing in \eqref{linear}, which acts \lq\lq almost diagonally"
   between function  spaces, likewise $|\nabla \bfu|^{p-2}\nabla \bfu$  and
   $\bfF$  are related via a sublinear operator in the nonlinear
   setting.
%
%
   The operator which now comes into play
   is the sharp maximal operator, and has to be applied  both to $|\nabla \bfu|^{p-2}\nabla \bfu$  and
   $\bfF$. Recall that the
sharp maximal  operator $M^ \sharp$ is defined  as
   \begin{equation}\label{Msharp} M^\sharp \bff (x) = \sup _{B\ni x} \dashint_{B}|\bff - \mean{\bff}_B|dy \qquad \hbox{for $x \in \Rn$,}
   \end{equation}
for  any locally integrable function $\bff : \Rn \to \mathbb R^m$.
Here, $m \in \mathbb N$, $B$ denotes a ball in $\Rn$,  $|B|$ stands
for its Lebesgue measure, $\dashint_{B} = \frac 1{|B|} \int $, and
$\mean{\bff}_B = \dashint_{B} \bff (y) dy$.
More generally, if $q\geq 1$,
  the operator $M^ {\sharp ,q}$  is given by
       \begin{equation}\label{Msharps}
       M^{\sharp ,q}\bff (x) = \sup _{B \ni x} \bigg(\dashint_{B}|\bff -\mean{\bff}_B|^qdy\bigg)^{\frac 1q} \qquad \hbox{for $x \in \Rn$,}.
       \end{equation}
       Hence, $M^{\sharp ,1}= M^\sharp$.  These operators are known to
       be bounded in customary, non-borderline, function spaces endowed with a norm which is  locally stronger  than $L^q(\mathbb R^n)$.    For instance, $M^{\sharp ,q}$ is bounded in  $L^r(\Rn)$  for every $r \in (q, \infty
   ]$. The operator $M^{\sharp ,q}$, as well as other operators to
   be considered below, will also be applied to matrix-valued
   functions, with a completely analogous definition. In fact,  matrices will be identified with vectors,  with an
   appropriate number of components, whenever they are
   elements of the target space of functions.

\par
 We denote by
$W^{1,p}(\Omega)$, $W^{1,p}_{\rm \loc}(\Omega)$ and
$W^{1,p}_0(\Omega)$ the usual Sobolev spaces of  weakly
differentiable functions. Moreover, we shall make use of the
homogeneous Sobolev space
\begin{equation*}
V^{1,p}(\Omega )= \{\bfu : \hbox{$\bfu$ is a weakly differentiable
function in $\Omega$,  and $|\nabla \bfu| \in L^p(\Omega)$}\}.
\end{equation*}
Here, and in similar occurrences below, we do not indicate the
target space in the notation of function spaces. What are the
elements of the target space in question (real numbers, vectors,
matrices) will be clear from the context. Usually, real-valued
functions will be denoted in standard-face, and vector-valued or
matrix-valued functions in  bold-face.
\par
A basic version of our pointwise estimates is stated in the
following theorem, where, in particular, we deal with the case when
$\Omega = \Rn$ in \eqref{eq:sysA}.

\begin{theorem}
\label{thm:main} {\bf [Basic estimate in $\Rn$]} Let $n \geq 2$, $N
\geq 1$ and $p \in (1, \infty)$. Assume that
$\bfF \in L^{p'}_{\rm loc}(\Rn)$.
Let $\bfu \in V^{1,p}(\Rn)$ be a local weak solution to system
\eqref{eq:sysA},  with  $\Omega = \Rn$. Then there exists a constant
$c=c(n, N, p)$ such that
\begin{equation}\label{main1}
 M^\sharp\big(\abs{\nabla \bfu}^{p-2}\nabla \bfu\big)(x)\leq cM^{\sharp, {p'}} \bfF (x) \quad \hbox{for a.e. $x\in \mathbb
 R^n$.}
\end{equation}
\end{theorem}

\medskip
\par

 Theorem  \ref{thm:main} enables one to transfer the problem of bounds
       for so called Banach function norms of $\nabla \bfu$ in terms of $\bfF$
              to boundedness properties of  $M^{\sharp, {p'}}$, and  reverse boundedness properties of
       $M^\sharp$ between function spaces endowed with norms of this kind.
 In particular, the  results
        for Lebesgue  norms recalled above can  easily be recovered via classical
       properties of the sharp maximal operator. More interestingly, new estimates also
       follow from Theorem \ref{thm:main}, including gradient bounds in Lorentz and Orlicz norms. These can be derived
       as special instances of a general approach, developed in Section \ref{rearr}
       below,  for gradient regularity
in norms depending only on its size.
%
%
%
%
%

\begin{remark}\label{min}
{\rm It will be clear from our proof of Theorem \ref{thm:main} that
the operator $ M^\sharp$ can be replaced with
$M^{\sharp,\min\set{p',2}}$ on the left-hand side of inequality
\eqref{main1}. Thus,
 the slightly stronger
inequality
\begin{equation}\label{main1bis}
 M^{\sharp, {\min\set{p',2}}}\big(\abs{\nabla \bfu}^{p-2}\nabla \bfu\big)(x)\leq c M^{\sharp, {p'}} \bfF (x)
\end{equation}
actually holds for  a.e. $x\in \mathbb
 R^n$. However, in all our applications, \eqref{main1} and
\eqref{main1bis} turn out to lead to the same conclusions.}
\end{remark}

\medskip
\par
Although quite general, inequality \eqref{main1} can still be
enhanced to a form which is also well suited  for gradient bounds in
norms possibly depending on oscillations. The resulting inequality
can be given a local form, which applies to solutions to  system
\eqref{eq:sysA} in any open set $\Omega$. A localized and weighted
sharp maximal  operator comes into play, which is defined as
follows. Let $ q \in [1, \infty)$, and, given  $R>0$, let $\omega :
(0, R) \to (0, \infty)$ be a  function.
Define, for $\bff \in L^q_{\rm loc}(\Omega)$,
%
\begin{align}\label{localweight}
  M^{\sharp,q}_{\omega, R } \bff(x)=\sup_{
\begin{tiny}
 \begin{array}{c}{
    B_r\ni x} \\
r< R
 \end{array}
  \end{tiny}
}\frac{1}{\omega(r)}
  \Big(\dashint_{B_r}\abs{\bff-\mean{\bff}_{B_r}}^qdy\Big)^\frac{1}{q}\quad \hbox{for $x \in
  B_R$,}
\end{align}
for every $x \in \Omega$ such that ${\rm dist}(x, \Rn \setminus
\Omega)>R$. Here, $B_r$ denotes any ball of radius $r>0$. When
needed, we shall use the notation $B_r(x)$ for a ball of radius $r$,
centered at the point $x \in \Rn$. The simplified notation
$$ M^{\sharp}_{\omega, R }=  M^{\sharp,1}_{\omega, R }$$
will be employed for $q=1$. If $\Omega = \Rn$, then the right-hand
side of \eqref{localweight} is well defined also for $R=\infty$. In
this case, we set $$ M^{\sharp,q}_{\omega } = M^{\sharp,q}_{\omega,
\infty}.$$
 In view of our purposes,  the additional property
 that the function $\omega (r)r^{-\beta}$ be almost decreasing
in $(0, R)$, for a suitable $\beta >0$, will be needed. This amounts
to requiring  that
\begin{equation} \label{eq:omega condition}
\omega(r) \leq c_\omega \rho^{-\beta} \omega(r\rho) \qquad \hbox{for
$r \in (0,R)$ and $\rho \in (0,1)$,}
\end{equation}
for some  constant $c_\omega$.
%
%
%
%
%
%

\begin{theorem}
\label{thm:main2} {\bf [General oscillation estimate in domains]}
Let $n \geq 2$, $N \geq 1$, $R>0$ and $p \in (1, \infty)$. Let
$\Omega $ be an open set in $\Rn$ and let
 $\bfF \in
L^{p'}_{\rm loc}(\Omega)$. Let $\bfu \in W^{1,p}_{\rm loc}(\Omega)$
be a local weak solution to system \eqref{eq:sysA}.
Then there exists a constant $\beta = \beta(n,p,N)>0$ such that, if
$\omega : (0, R) \to (0, \infty)$ is any function with the property
that
 $\omega (r)r^{-\beta}$ is almost decreasing in $(0, R)$
in the sense of \eqref{eq:omega condition}, then
\begin{multline}\label{E:main2}
 M^{\sharp}_{\omega,R}\big(\abs{\nabla \bfu}^{p-2}\nabla \bfu\big)(x)
 \\ \leq cM^{\sharp,p'}_{\omega,R} \bfF (x)+ \frac{c}{\omega(R)}\bigg(\dashint_{B_{2R}}
 \abs{\abs{\nabla \bfu}^{p-2}\nabla \bfu-\mean{\abs{\nabla \bfu}^{p-2}\nabla
 \bfu}_{B_{2R}}}^{p'}\, dy\bigg)^{\frac 1{p'}}\quad \hbox{for
 a.e.
 $x \in B_R$,}
\end{multline}
for some  constant $c=c(n, N, p,c_\omega)$, and for every concentric
balls $B_R\subset B_{2R} \subset \Omega$. Here, $c_\omega$ denotes
the constant appearing in \eqref{eq:omega condition}. In particular,
the conclusion holds for any $\beta \in \big(0, \min \{1,
\frac{2\alpha}{p'}\}\big)$, where $\alpha =\alpha (n,N,p)$ is the
H\"older exponent appearing in a gradient estimate for the solutions
to the $p$-harmonic system (see Theorem \ref{thm:decay}, Section
\ref{sec:decay} below).
\end{theorem}
%

\begin{remark}\label{remrn}
{\rm As in \eqref{main1}, the operator $ M^\sharp_{\omega,R}$ can be
replaced with $M^{\sharp, {\min\set{p',2}}}_{\omega,R}$ on the
left-hand side of \eqref{E:main2}.}
\end{remark}

\begin{remark}\label{remrn bis}
{\rm In particular, if $\Omega = \Rn$, and assumption
\eqref{eq:omega condition} holds with $R=\infty$, then inequality
\eqref{E:main2} implies that
\begin{equation}\label{E:main2rn}
M^{\sharp}_{\omega}\big(\abs{\nabla \bfu}^{p-2}\nabla \bfu\big)(x)
 \\ \leq cM^{\sharp,p'}_{\omega} \bfF (x) \quad \hbox{for a.e. $x \in
 \Rn$,}
 \end{equation}
provided that $\bfu \in V^{1,p}(\Rn)$.}
%
%
%
%
%
%
\end{remark}

\par
As a consequence of Theorem \ref{thm:main2},  the H\"older
 regularity of the gradient of solutions to the $p$-Laplace system is easily recovered.
 However, more general regularity properties can  be deduced.
Inequalities between semi-norms of $\nabla \bfu$ and $\bfF$ in
(generalized) Campanato spaces, associated with a function $\omega$
as above, stem from inequality \eqref{E:main2}. In particular, they
tell us that information on the modulus of continuity of $\nabla
\bfu$ in terms of that  of $\bfF$ can still be derived, if  the
 modulus of continuity $\omega$ of $\bfF$, although not of power type, yet
satisfies the Dini type condition
$$\int _0 \frac {\omega (r)}r\, dr < \infty.$$
Such a condition   is  sharp, even in the simplest linear case when
$p=2$, as shown by Example \ref{sharpness}, Section \ref{oscill}.
These consequences of Theorem \ref{thm:main2} are presented in
Section \ref{oscill}, which also includes a discussion of $\setBMO$
and $\setVMO$ gradient regularity.

\medskip
\par
Pointwise gradient bounds for solutions to system \eqref{eq:sysA},
which are maximal operator free,  also follow from the methods of
this paper. Interestingly, they involve an unconventional
Havin-Maz'ya-Wolff type nonlinear potential of the right-hand side
of \eqref{eq:sysA}, defined in terms of its integral oscillations on
balls.

\begin{theorem}
\label{thm:main3} {\bf [Potential type estimates]} Let $n \geq 2$,
$N \geq 1$ and $p \in (1, \infty)$. Let $\Omega $ be an open set in
$\Rn$, and let $\bfF \in L^{p'}_{\rm loc}(\Omega)$. Let $\bfu \in
W^{1,p}_{\rm loc}(\Omega)$ be a local weak solution to system
\eqref{eq:sysA}. Then there exists a constant $c=c(n, N, p)$ such
that
\begin{align}\label{E:main3}
 \abs{\nabla \bfu(x)}^{p-1}
&\leq c\int_0^R\bigg( \dashint_{B_r(x)}
\bigg(\frac{\abs{\bfF-\mean{\bfF}_{B_r(x)}}}{r} \bigg)^{p'} \dy
\bigg)^\frac{1}{p'}\frac{dr}{r} + c\dashint_{B_{R}(x)}\abs{\nabla
\bfu}^{p-1}\dy
\end{align}
  for a.e. $x \in \Omega$, and every $R>0$ such that
$B_{R}(x)\subset {\Omega}$. Moreover, a point $x \in \Omega$
 is a Lebesgue point of
$\abs{\nabla \bfu}^{p-2} \nabla \bfu$ whenever the right-hand side
of \eqref{E:main3} is finite for some $R>0$.
\end{theorem}

From Theorem \ref{thm:main3} one immediately infers, for instance,
that $|\nabla \bfu|$ is locally bounded in $\Omega$, provided that
$\bfF$ has a modulus of continuity $\omega$ satisfying the Dini type
condition displayed above. Example \ref{sharpness} again
demonstrates the sharpness of the relevant condition with this
regard.

\begin{remark}\label{remrnwolff}
{\rm If $\Omega = \mathbb R^n$, and a solution $\bfu$ to system
\eqref{eq:sysA} belongs to $V^{1,p}(\mathbb R^n)$, then letting $R$
tend to infinity in inequality \eqref{E:main3} tells us that
\begin{equation}\label{may100}
 \abs{\nabla \bfu(x)}^{p-1}
\leq c\int_0^\infty \bigg(\dashint_{B_r(x)}
\bigg(\frac{\abs{\bfF-\mean{\bfF}_{B_r(x)}}}{r} \bigg)^{p'} \dy
\bigg)^\frac{1}{p'}\frac{dr}{r}\quad \hbox{for a.e. $x \in \mathbb
R^n$.}
\end{equation}
}
\end{remark}

Let us conclude this section with a brief outline of the methods of
proofs. Our approach relies upon precise decay estimates on balls
for suitable nonlinear expressions of the gradient of solutions to
system \eqref{eq:sysA}. These estimates are obtained through
comparisons with the gradient of solutions to simpler systems, whose
behavior is quite well known: the $p$-harmonic system, namely
\eqref{eq:sysA} with $\bfF =0$, and a linear system which
approximates \eqref{eq:sysA} up to the second order. One major
novelty in our technique amounts to a different use of such
auxiliary systems, depending  on whether $p \in (1, 2)$ or $p\in [2,
\infty)$, and on whether the integral, on the relevant balls, of an
appropriate function of the gradient is \lq\lq small" or \lq\lq
large", compared with the integral oscillation of the function in
question on the same balls. Merging the resulting comparison
estimates  requires a fine tuning of various parameters which come
into play. Most of the intermediate steps, that eventually lead to
our final results, call for the replacement of the $p$-power
function with a smoothed (still convex) function near $0$, called
\lq\lq shifted $p$-power function" in what follows, which again
depends on a parameter. Of course, a crucial feature of the relevant
intermediate steps is that the involved constants
 are independent of this parameter.

\section{Decay estimates}\label{sec:decay}

The present section is devoted to  decay estimates for the
oscillation on balls of the gradient of local weak solutions to
system \eqref{eq:sysA}. A function $\bfu \in W^{1,p}_{\loc}(\Omega)$
is called a local  weak solution to \eqref{eq:sysA} if
\begin{equation}\label{weaksol}
\int _{\Omega '} |\nabla \bfu|^{p-2}\nabla \bfu \cdot \nabla \bfphi
\, dx = \int _{\Omega '} \bfF \cdot \nabla \bfphi\,dx
\end{equation}
for every function $\bfphi \in W^{1,p}_0(\Omega ' )$, and every open
set $\Omega ' \subset \subset \Omega$. Here, the dot $\lq\lq \cdot
"$ stands for scalar product.
\\
The relevant estimates  constitute the core of our proofs, and
 involve the function $\bfA : \setR^{N \times
n} \to \setR^{N \times n}$ given by
$$\hbox{$\bfA(\bfP) = \abs{\bfP}^{p-2} \bfP$ \qquad
 for $\bfP \in \setR^{N \times n}$}$$ and $\bfV : \setR^{N \times n} \to
\setR^{N \times n}$ given by
$$\hbox{$\bfV(\bfP) = \abs{\bfP}^{\frac{p-2}{2}} \bfP$
 \qquad for $\bfP
\in \setR^{N \times n}$}.$$

Our final goal here is  the following result.

\begin{proposition} \label{prop:decay conclusion}
Let $p\in (1, \infty)$, and let $\bfu$ be a local weak solution to
\eqref{eq:sysA}.
Assume that  $\delta \in (0,1)$. Then there exist constants $\theta =
\theta(n,p,N,\delta)\in (0, 1)$ and
$c_\delta=c_\delta(n,p,N,\delta)>0$ such that
\begin{align}
\label{eq:decay conclusion}
 \bigg(\dashint_{\theta B}& \abs{\bfA(\nabla \bfu)-\mean{\bfA(\nabla \bfu)}_{\theta
 B}}^{\min\set{2,p'}}\dx\bigg)^\frac{1}{\min{\set{2,p'}}}
 \\ \nonumber & \qquad \leq \delta\bigg( \dashint_{B}\abs{\bfA(\nabla \bfu)-\mean{\bfA(\nabla
\bfu)}_{{B}}}^{\min\set{2,p'}}\dx\bigg)^\frac1{{\min\set{2,p'}}}
 + c_\delta \bigg(\dashint_{B}\abs{\bfF -\bfF
 _0}^{p'}\dx\bigg)^{\frac 1{p'}}
\end{align}
for every ball $B \subset \Omega$. In particular, inequality
\eqref{eq:decay conclusion} holds with
\begin{equation}\label{august300}
\theta = c\delta ^{\varsigma}
\end{equation}
 for small $\delta$, where  $\varsigma$ is any
number larger than $\max\{1, \frac {p'}{2\alpha}\}$, $\alpha =\alpha
(n,N,p)$ is the H\"older exponent appearing in a gradient estimate
for the solutions to the $p$-harmonic system (see Theorem
\ref{thm:decay} below), and $c=c(n,N,p,\varsigma)$.
\end{proposition}

The proof of Proposition \ref{prop:decay conclusion} is
accomplished through several steps, to which
the following subsections are devoted.

\subsection{Preliminary estimates}\label{sec:preliminary}

Several inequalities will be conveniently formulated in terms of a
 \lq\lq shifted" $p$-power function  $\varphi_{p,a} : [0,
\infty) \to [0, \infty)$, introduced in~\cite{DieE08}, and defined
for $a\geq0$ and $p\in(1,\infty)$ as
$$\varphi_{p,a}(t)=(a+t)^{p-2}t^2 \qquad \hbox{for $t \geq 0$.}$$
Clearly, $\varphi_{p,0}(t)=t^p $. The function $\varphi_{p,a}$ is
nonnegative and convex,  and vanishes at $0$;  it is hence  a Young
function. Consequently, for every $\delta >0$, there exists a
constant $c=c(\delta, p)$ such that
 \begin{align}
  \label{eq:youngb}
    t s &\leq \delta\, \varphi_{p,a}(t) + c \,
    \varphi_{p',a^{p-1}}(s) \qquad \hbox{for  $t, s \geq 0$.}
 \end{align}
 Basic algebraic relations among the functions $\bfA$, $\bfV$, and
$\varphi_{p,a}$ are summarized hereafter. They are the content of
\cite[Lemmas~3, 21, and 26]{DieE08} and \cite[Appendix]{DieK08}.
Throughout, we denote by $c$ or $C$  a generic constant, which may
change from line to line, which depends  on specified quantities.
Moreover, given two nonnegative functions $f$ and $g$, we write
$f\preceq g$ to denote that there exists a positive constant
 $c$ such
that $f \le c g$. The notation $f\approx g$ means that $f\preceq g
\preceq f$.
\\
 Let $p\in(1,\infty)$. Then
    \begin{align}
      \label{eq:hammera}
      \big({\bfA}(\bfP) - {\bfA}(\bfQ)\big) \cdot \big(\bfP-\bfQ
      \big)
      & \approx \bigabs{ \bfV(\bfP) - \bfV(\bfQ)}^2 \approx (\abs{\bfQ} + \abs{\bfP})^{p-2}\abs{\bfQ-\bfP}^2
\\ \nonumber
    &
    \approx\varphi_{p,\abs{\bfQ}}\big(\abs{\bfP-\bfQ}\big)
    \approx\varphi_{p',\abs{\bfQ}^{p-1}}\big(\abs{\bfA(\bfP)-\bfA(\bfQ)}\big),
\end{align}
    for  $\bfP, \bfQ \in \setR^{N \times n}$, up to equivalence constants depending only on $n, N, p$.
        Moreover,
      \begin{align}\label{eq:hammerd}
      \bfA(\bfQ) \cdot \bfQ = \abs{\bfV(\bfQ)}^2 \approx \varphi_p(\abs{\bfQ}),
  \end{align}
and
\begin{align}\label{eq:hammerd-bis}
           \abs{\bfA(\bfP)-\bfA{(\bfQ)}}&\approx\big(\varphi_{p,\abs{\bfQ}}\big)'\big(\abs{\bfP-\bfQ}\big) \approx (\abs{\bfQ} + \abs{\bfP})^{p-2}\abs{\bfQ-\bfP} ,
     \end{align}
   for  $\bfP, \bfQ \in \setR^{N \times n}$, up to equivalence constants depending only on $n, N, p$.

\par\noindent
For every $\gamma \in (0, 1]$, the   \lq\lq shift change" formula
     \begin{align}
    \label{eq:shiftchb}
    \varphi_{p',\abs{\bfP}^{p-1}}(t) \leq c\, \gamma^{1-\max\set{p,2}}
    \varphi_{p',\abs{\bfQ}^{p-1}}(t) + \gamma \abs{\bfV (\bfP) - \bfV (\bfQ)}^2.
  \end{align}
holds for some constant $c=c(n, N, p)$, and  for every  $\bfP, \bfQ
\in \setR^{N \times n}$.

Let us now recall  some inequalities, in integral form, for merely
measurable functions,  to be repeatedly used in our proofs. In what
follows, we denote by $m$ any number in $\mathbb N$. To begin with,
it is classical, and easily verified, that
\begin{equation}\label{mean2}
\|\bff - \mean{\bff}_E\|_{L^2(E)} = \min _{\bfc\in \mathbb
R^m}\|\bff - \bfc\|_{L^2(E)},
\end{equation}
for any measurable set $E$ in $\Rn$, and and every function $\bff :
E \to \mathbb R^m$ such that $\bff \in L^2(E)$. If $q \in [1,
\infty]$, then
\begin{equation}\label{meanq}
\|\bff - \mean{\bff}_E\|_{L^q(E)} \leq 2\min _{\bfc\in \mathbb
R^m}\|\bff - \bfc\|_{L^q(E)},
\end{equation}
for every $f \in L^q(E)$.

\par
The following result is less standard, and concerns the equivalence
of certain integral averages on balls.
\begin{lemma}\label{lem:trick} \cite[Lemma 6.2]{DieKapSch11}
  Let $p\in(1,\infty)$, and let  $B$ be a ball in $\Rn$.
  Given any function $\bfg : \Omega \to \mathbb R^m$ such that $\bfg\in L^p(B)$, denote by $\bfg_\bfA$  the vector in $\mathbb R^m$
  obeying
  $\bfA(\bfg_\bfA) = \mean{\bfA(\bfg)}_B$. Then
  \begin{align}\label{trick1}
    \dashint_B \abs{\bfV(\bfg) \!-\! \mean{\bfV(\bfg)}_B}^2\,dx &\approx
    \dashint_B \abs{\bfV(\bfg) \!-\! \bfV(\mean{\bfg}_B)}^2\,dx \approx \dashint_B
    \abs{\bfV(\bfg) \!-\! \bfV(\bfg_\bfA)}^2\,dx,
  \end{align}
 up to equivalence  constants independent of $B$ and $\bfg$.
\end{lemma}

The next lemma encodes self-improving properties of reverse H\"older
inequalities for shifted functions. In what follows, given a ball
$B$ in $\Rn$ and a positive number $\theta$, we denote by $\theta B$
the ball, with the same center as $B$, whose radius is $\theta$
times the radius of $B$.

\begin{lemma}
  \label{cor:invjensen}\cite[Corollary 3.4]{DieKapSch11}
  Let $\Omega$ be an open subset of $\Rn$. Let  $p\in(1,\infty)$, $a \in [0, \infty)$, and let $\bfg , \bfh : \Omega \to \mathbb R^m$ be such that
  $\bfg\in L^p_{\loc}(\Omega)$ and $\bfh\in L^1_{\loc}(\Omega)$.
  Assume that
  there exist constants  $\sigma \in (0,1)$ and $c_0>0$ such that
  \begin{align*}
    \dashint_B \varphi_{p,a}(\abs{\bfg})\,dx \leq c_0\, \bigg(\dashint_{2B}
    \varphi_{p,a}(\abs{\bfg})^{\sigma} \,dx\bigg)^\frac{1}{\sigma} + \dashint_{2B}
    \abs{\bfh}\,dx,
  \end{align*}
  for every ball $B$ such that $2B \subset \Omega$. Then there exists a constant $c_1=c_1(c_0, p, n,\sigma)$  such that
  \begin{align*}
    \dashint_B \varphi_{p,a}(\abs{\bfg})\,dx \leq c_1\,\varphi_{p,a} \bigg( \dashint_{2B}
    \abs{\bfg}\,dx \bigg) + c_1\,
    \dashint_{2B} \abs{\bfh}\,dx
  \end{align*}
  for every ball $B$ such that $2B \subset \Omega$.
\end{lemma}

We begin our discussion of system \eqref{eq:sysA} by recalling a
reverse H\"older type inequality for the excess functional
$\dashint_B \abs{\bfV(\nabla \bfu) - \bfV(\bfP)}^2\dx$.
\begin{lemma}\label{lem:VPrev-bis}\cite[Lemma 3.2]{DieKapSch11}
  Let $p \in (1,\infty)$, and let $\bfu$ be a local weak solution to \eqref{eq:sysA}.  Then there exist constants
$\sigma = \sigma (n, N, p)
  \in (0,1)$  and $c=c(n, N, p)>0$ such that
  \begin{align*}
    \begin{aligned}
      \dashint_B \abs{\bfV(\nabla \bfu) - \bfV(\bfP)}^2\dx &\leq c\,
      \bigg(\dashint_{2B} \abs{\bfV(\nabla \bfu) - \bfV(\bfP)}^{2\sigma}\dx
      \bigg)^{ \frac 1\sigma}
      \\
      &\quad + c \dashint_{2B} \varphi_{p',\abs{\bfP}^{p-1}}(\abs{\bfF -\bfF_0})
      \dx
    \end{aligned}
  \end{align*}
for every $\bfP,\bfF_0\in \setR^{N \times n}$,  and every
  ball $B$ such that $2B\subset\Omega$.
\end{lemma}

Lemmas \ref{cor:invjensen} and \ref{lem:VPrev-bis} enable us to
transfer information from the excess functional involving $\bfV$, to
an excess functional involving $\bfA$.

\begin{corollary}
  \label{cor:VPL1}
  Let $p \in (1,\infty)$, and  let $\bfu$ be a local weak solution to \eqref{eq:sysA}. Then there exists a constant
 $c=c(n, N, p)$
  such that
    \begin{align}\label{may1}
    \dashint_B \abs{\bfV(\nabla \bfu) - \bfV(\bfP)}^2\dx &\leq c\,
    \varphi_{p',\abs{\bfP}^{p-1}} \bigg(\dashint_{2B} \abs{\bfA(\nabla \bfu) -
      \bfA(\bfP)}\dx \bigg)
    \\ \nonumber
    &\quad + c\,\dashint_{2B} \varphi_{p',\abs{\bfP}^{p-1}}(\abs{\bfF -\bfF_0})
      \dx
  \end{align}
for every $\bfP,\bfF_0\in \setR^{N \times n}$,  and every
  ball $B$ such that $2B\subset\Omega$.
\end{corollary}
\begin{proof}
  By~\eqref{eq:hammera}, we have
  \begin{align*}
    \abs{\bfV(\nabla \bfu) - \bfV(\bfP)}^2 &\approx
    \varphi_{p',\abs{\bfP}^{p-1}}(\abs{\bfA(\nabla \bfu) - \bfA(\bfP)}).
  \end{align*}
Combining this relation with Lemmas~\ref{lem:VPrev-bis}
and~\ref{cor:invjensen} yields \eqref{may1}.
%
%
  \end{proof}

Crucial  use will be made in what follows of  decay estimates
 for $p$-harmonic maps, namely for solutions to \eqref{eq:sysA} with
 $\bfF=0$.
%
%
In particular, the  unique solution  $\bfv \in W^{1,p}(B)$ to the
Dirichlet problem
\begin{equation}
  \label{eq:hareq}
  \begin{cases}
    -\divergence \bfA(\nabla \bfv) =0  & \text{in} \, B,
    \\
    \bfv =\bfu & \text{on}\,\partial B
  \end{cases}
\end{equation}
will come into play. Here, $B$ is a ball in $\Omega$, and $\bfu$ is
a local weak solution to system \eqref{eq:sysA}. As usual, the
boundary condition in \eqref{eq:hareq} has to be understood in the
sense that $\bfu-\bfv \in W^{1,p}_0(B)$.
\\ The relevant decay estimates are collected  in the following
statement.

\begin{theorem}{\bf [Decay estimate for $p$-harmonic maps]}
  \label{thm:decay}
  Let $\Omega \subset \Rn$ be an open set and let $p\in(1,\infty)$. Assume that $\bfv $ is a $p$-harmonic function in~$\Omega$. Then there
  exist positive constants
  $\alpha= \alpha (n, N, p)$ and $c=c(n, N, p)$  such that
\begin{align}\label{pharmonic1}
    \sup_{z,y\in\theta B}\abs{\bfV(\nabla \bfv)(z)-\bfV(\nabla \bfv)(y)}^2 \le c\, \theta^{2\alpha} \dashint_{B}|\bfV(\nabla \bfv)-
    \mean{\bfV(\nabla \bfv)}_B|^2 \,dx
  \end{align}
for every  $\theta \in (0,\frac12]$, and every ball $B \subset
\Omega$. Moreover, for every $\kappa <\min\set{1, \tfrac
{2\alpha}{p'}}$, there exists a constant $c=c(n, N, p, \kappa)$ such
that
 \begin{align}\label{pharmonic2}
    \sup_{z,y\in\theta B}\abs{\bfA(\nabla \bfv)(z)-\bfA(\nabla \bfv)(y)} \le c\, \theta^\kappa \dashint_{B}|\bfA(\nabla \bfv)-
    \mean{\bfA(\nabla \bfv)}_B| \,dx
  \end{align}
 for every  $\theta \in (0,\frac12]$, and every ball $B \subset \Omega$.
\end{theorem}
A version of \eqref{pharmonic1} with a left-hand-side in integral
form is a special case of \cite[Theorem~6.4]{DieSV09}; an integral
form of \eqref{pharmonic2} can be found in \cite[Remark
5.6]{DieKapSch11}. The present version follows from Campanato's
characterization of H\"older spaces \cite[Lemma A.2]{Sch13}.
%

\smallskip
\par
A preliminary relation between the decay of a solution to system
\eqref{eq:sysA}, and that of the corresponding  solution to the
Dirichlet problem~\eqref{eq:hareq} is established in the following
lemma.

\begin{lemma} \label{lemma:comp1}
Let $p \in (1, \infty)$. Let $\bfu$ be a local weak solution to
system \eqref{eq:sysA}, and let $\bfv$ be the solution to the
Dirichlet problem~\eqref{eq:hareq}. Then, for every $\beta >0$,
there exists a constant $c_\beta  = c(n,N,p,\beta)$ such that
\begin{align}
\label{comp:1} \dashint_B \abs{\bfV(\nabla\bfu) - \bfV(\nabla
\bfv)}^2\dx &\leq \,\beta\,
    \varphi_{p',\abs{\mean{\bfA(\nabla\bfu)}_{2B}}}\bigg(\dashint_{2B}
    \abs{\bfA(\nabla\bfu) - \mean{\bfA(\nabla\bfu)}_{2B}}\dx\bigg)
    \\ \nonumber
    &\quad+ c_\beta \dashint_{2B} \varphi_{p',\abs{\mean{\bfA(\nabla\bfu)}_{2B}}}(\abs{\bfF -\bfF_0}) \dx
\end{align}
for every $\bfF_0 \in \mathbb R^{N \times n}$, and every ball $B$
such that $2B \subset \Omega$.
\end{lemma}

\begin{proof}
 Choosing
$\bfu-\bfv$ as a test function in \eqref{weaksol} and making use of
equation ~\eqref{eq:hammera} yield
\begin{align*}
\dashint_B \abs{\bfV(\nabla\bfu) - \bfV(\nabla \bfv)}^2\dx &\leq
\,c\,\dashint_B
\big(\bfA(\nabla\bfu)-\bfA(\nabla\bfv)\big)\cdot\nabla(\bfu-\bfv)\dx\\&=\,c\,\dashint_B
\big(\bfF-\bfF_0\big)\cdot\nabla\big(\bfu-\bfv\big)\dx.
\end{align*}
Hence, via an application of inequality \eqref{eq:youngb}, with
$a=|\nabla\bfu|$, and equation ~\eqref{eq:hammera} again, one
obtains that
\begin{align}\label{nov1}
\dashint_B \abs{\bfV(\nabla\bfu) - \bfV(\nabla \bfv)}^2\dx &\leq
\,\delta
\dashint_B\varphi_{p,|\nabla\bfu|}(|\nabla(\bfu-\bfv)|)\dx+c_\delta
\dashint_B\varphi_{p',|\nabla\bfu|^{p-1}}(|\bfF-\bfF_0|)\dx\\
\nonumber
 &\leq \,c\delta \dashint_B\abs{\bfV(\nabla\bfu) - \bfV(\nabla \bfv)}^2\dx+c_\delta
 \dashint_B\varphi_{p',|\nabla\bfu|^{p-1}}(|\bfF-\bfF_0|)\dx
\end{align}
 for every $\delta>0$. Let  $\bfQ \in \mathbb R^{N \times n}$ be such that
\begin{equation}\label{Q}
\bfA(\bfQ)= \mean{\bfA(\nabla\bfu)}_{2B}.
\end{equation} Owing to~\eqref{eq:shiftchb}, applied with  $\bfP=\nabla\bfu$ and  $\bfQ$ as in \eqref{Q}, we
deduce from \eqref{nov1} that
\begin{align}\label{Q'}
\dashint_B \abs{\bfV(\nabla\bfu) - \bfV(\nabla \bfv)}^2\dx &\leq \,
c\dashint_B\varphi_{p',|\nabla\bfu|^{p-1}}(|\bfF-\bfF_0|)\dx\\
\nonumber &\leq \gamma
\dashint_B|\bfV(\nabla\bfu)-\bfV(\bfQ)|^2\dx+\, c_\gamma
\dashint_B\varphi_{p',\abs{\bfA(\bfQ)}}(|\bfF-\bfF_0|)\dx,
\end{align}
for every $\gamma >0$, and a corresponding suitable constant
$c_\gamma$. On the other hand, by Corollary~\ref{cor:VPL1} and the
equality $\abs{A(Q)}=\abs{Q}^{p-1}$,
    \begin{align}
    \label{comp:001}
    \dashint_B \abs{\bfV(\nabla \bfu) - \bfV(\bfQ)}^2\dx &\leq c\,
    \varphi_{p',\abs{\bfA(\bfQ)}} \bigg(\dashint_{2B} \abs{\bfA(\nabla \bfu) -
      \bfA(\bfQ)}\dx \bigg)
    \\ \nonumber
    &\quad + c\,\dashint_{2B} \varphi_{p',\abs{\bfA(\bfQ)}}(\abs{\bfF -\bfF_0}) \dx.
        \end{align}
Inequality \eqref{comp:1} follows from \eqref{Q}, \eqref{Q'} and
\eqref{comp:001}, via a suitable choice of $\gamma$.
%
%
\end{proof}

We are ready to prove a first decay estimate for $\bfu$.

\begin{lemma}\label{lem:new}
  Let $p \in (1, \infty)$, and  let $\bfu$ be a local weak solution to
  \eqref{eq:sysA}. Let $\alpha$ be the exponent appearing in Theorem
  \ref{thm:decay}.
  Assume that
 $\theta \in (0, \tfrac 12)$. Then there exist  constants $c=c(n, N, p)$
and $c_\theta = c_\theta (n, N, p, \theta)$ such that
\begin{align}
\label{eq:key2}
 \dashint_{\theta B} \abs{\bfV(\nabla\bfu) - \mean{\bfV(\nabla \bfu)}_{\theta B}}^2\dx
\leq & c\, \theta^{2\alpha}
      \varphi_{p',\abs{\mean{\bfA(\nabla\bfu)}_{2B}}}\bigg(\dashint_{2B} \abs{\bfA(\nabla \bfu) -
        \mean{\bfA(\nabla\bfu)}_{2B}}\dx\bigg)
       \\ \nonumber &
      + c_\theta \dashint_{2B} \varphi_{p',\abs{\mean{\bfA(\nabla\bfu)}_{2B}}}(\abs{\bfF -\bfF_0})\dx
\end{align}
for every ball $B$ such that $2B \subset \Omega$.
\end{lemma}

\begin{proof}
From inequality \eqref{pharmonic1} and property \eqref{mean2} we
deduce that
  \begin{align}\label{nov2}
          \dashint_{\theta B}& \abs{\bfV(\nabla \bfu) - \mean{\bfV(\nabla
          \bfu)}_{\theta B}}^2\dx
      \\ \nonumber
      &\leq c\, \dashint_{\theta B} \abs{\bfV(\nabla \bfv) - \mean{\bfV(\nabla
          \bfv)}_{\theta B}}^2\dx + c\,\dashint_{\theta B} \abs{\bfV(\nabla \bfu) -
        \bfV(\nabla \bfv)}^2\dx
      \\ \nonumber
      &\leq c\, \theta^{2\alpha} \dashint_{B} \abs{\bfV(\nabla \bfv) -
        \mean{\bfV(\nabla \bfv)}_{B}}^2\dx + c\theta^{-n} \,\dashint_{B}
      \abs{\bfV(\nabla \bfu) - \bfV(\nabla \bfv)}^2\dx
      \\ \nonumber
      &\leq c\, \theta^{2\alpha} \dashint_{B} \abs{\bfV(\nabla \bfu) -
        \mean{\bfV(\nabla \bfu)}_{B}}^2\dx + c\,\theta^{-n} \,\dashint_{B}
      \abs{\bfV(\nabla \bfu) - \bfV(\nabla \bfv)}^2\dx.
      \\ \nonumber
      &\leq c\, \theta^{2\alpha} \dashint_{B} \abs{\bfV(\nabla \bfu) -
        \bfV(\bfQ)}^2\dx + c\,\theta^{-n} \,\dashint_{B}
      \abs{\bfV(\nabla \bfu) - \bfV(\nabla \bfv)}^2\dx,
      \end{align}
  where $\bfQ$ is defined in~\eqref{Q}. Hence, inequality \eqref{eq:key2}
  follows,
  via Corollary~\ref{cor:VPL1} (applied with $\bfP = \bfQ$), and  Lemma~\ref{lemma:comp1} (applied with $\beta = \theta^{n+2\alpha}$).
   %
%
%
\end{proof}

From here on, our decay estimate on a given ball $B \subset \Omega$
takes a different form depending on whether the quantity  $
\dashint_{2B}\abs{\nabla \bfu}^p\dx$ is \lq\lq small" or \lq\lq
large" compared to $\dashint_{2B}|\bfV(\nabla
\bfu)-\mean{\bfV(\nabla \bfu)}_{2B}|^2\dx$. We shall refer to the
former situation as the \lq\lq degenerate case", and to the latter
as the \lq\lq non-degenerate case".

\subsection{The degenerate case}\label{sec:comparison}
Throughout this subsection, we assume that $\bfu$ is a local weak
solution to system \eqref{eq:sysA} such that
\begin{align}
\label{eq:deg} \dashint_{2B}\abs{\bfV(\nabla \bfu)}^2\dx\leq
\frac1\varepsilon\dashint_{2B}|\bfV(\nabla \bfu)-\mean{\bfV(\nabla
\bfu)}_{2B}|^2\dx
\end{align}
for some fixed
 ball  $B$ such that $4B \subset \Omega$, and  some fixed number $\varepsilon >0$.
\\
We begin with the following lemma.

\begin{lemma}
  \label{lem:RHdeg}
  Let $p \in (1,\infty)$.  Let $\bfu$ be a local weak solution to \eqref{eq:sysA} satisfying \eqref{eq:deg} for some ball $B$ and some $\varepsilon >0$. Let
$\bfP \in \mathbb R^{N \times n}$ be such that
  $ \bfA(\bfP)=\mean{\bfA(\nabla \bfu)}_{2B}$. Then there exists a
  constant $c=c(n, N, p)$ such that
%
    \begin{align}\label{may2}
   & \dashint_{2B} \abs{\bfV(\nabla \bfu) - \bfV(\bfP)}^2\dx +  \varphi_{p',\abs{\bfA(\bfP)}} \bigg(\dashint_{2B} \abs{\bfA(\nabla \bfu) - \bfA(\bfP)}\dx \bigg)
   + \dashint_{2B} \varphi_{p',\abs{\bfA(\bfP)}}(\abs{\bfF -\bfF_0})\dx\\
   \nonumber
    &\quad\leq c\varepsilon^{1-\max\set{2,p}}
     \bigg(\dashint_{4B} \abs{\bfA(\nabla \bfu) -
      \bfA(\bfP_0)}\dx \bigg)^{p'}
 + c\varepsilon^{1-\max\set{2,p}}\dashint_{4B} \abs{\bfF -\bfF_0}^{p'}
      \dx
  \end{align}
for every $\bfF_0,\bfP_0 \in \setR^{N \times n}$.
\end{lemma}
\begin{proof} Set
\begin{equation}\label{m}
m = \max\set{2,p}.
\end{equation}
 By   Corollary \ref{cor:VPL1}, applied with $B$ replaced by
$2B$, and  by inequality \eqref{eq:shiftchb},  with $\bfQ=0$ and
$\gamma = \varepsilon c'$, where $c'$ is a positive constant to be
chosen later,
one has that
    \begin{align}\label{june1}
    &\dashint_{2B} \abs{\bfV(\nabla \bfu) -  \bfV(\bfP)}^2\dx +  \varphi_{p',\abs{\bfA(\bfP)}} \bigg(\dashint_{2B} \abs{\bfA(\nabla \bfu) - \bfA(\bfP)}\dx \bigg)
   + \dashint_{2B} \varphi_{p',\abs{\bfA(\bfP)}}(\abs{\bfF -\bfF_0})\dx
   \\ \nonumber
    &\quad \leq c\,
    \varphi_{p',\abs{\bfP}^{p-1}} \bigg(\dashint_{4B} \abs{\bfA(\nabla \bfu) -
      \bfA(\bfP)}\dx \bigg)
 + c\,\dashint_{4B} \varphi_{p',\abs{\bfP}^{p-1}}(\abs{\bfF
 -\bfF_0})\dx
       \\ \nonumber
      &\quad \leq c\varepsilon^{1-m}\,\bigg(\dashint_{4B} \abs{\bfA(\nabla \bfu) -
      \bfA(\bfP)}\dx \bigg)^{p'} + cc' \varepsilon
      \abs{\bfV(\bfP)}^2\\ \nonumber &\quad
+ c\varepsilon^{1-m}\,\dashint_{4B} \abs{\bfF -\bfF_0}^{p'} \dx +c
c' \varepsilon \abs{\bfV(\bfP)}^2.
      \dx
  \end{align}
  Since, by \eqref{eq:deg} and  \eqref{mean2},
  \begin{align}
\label{eq:jensen} \abs{\bfV(\bfP)}^2\leq
\dashint_{2B}\abs{\bfV(\nabla \bfu)}^2\, dx \leq
\frac1\varepsilon\dashint_{2B}|\bfV(\nabla \bfu)-\mean{\bfV(\nabla
\bfu)}_{2B}|^2\dx\,\leq
\,\frac{1}{\varepsilon}\dashint_{2B}|\bfV(\nabla
\bfu)-\bfV(\bfP)|^2\dx,
  \end{align}
  choosing $c' = \tfrac 1{4c}$ in \eqref{june1} ensures that
  \begin{align*}
    &\dashint_{2B} \abs{\bfV(\nabla \bfu) -  \bfV(\bfP)}^2\dx  +  \varphi_{p',\abs{\bfA(\bfP)}} \bigg(\dashint_{2B} \abs{\bfA(\nabla \bfu) - \bfA(\bfP)}\dx \bigg)
   + \dashint_{2B} \varphi_{p',\abs{\bfA(\bfP)}}(\abs{\bfF -\bfF_0})\dx\\
    &\quad \leq c\varepsilon^{1-m}\,\bigg(\dashint_{4B} \abs{\bfA(\nabla \bfu) -
      \bfA(\bfP)}\dx \bigg)^{p'}+c\varepsilon^{1-m}\,\dashint_{4B} \abs{\bfF -\bfF_0}^{p'}\dx\\
      &\quad \leq c\varepsilon^{1-m}\,\bigg(\dashint_{4B} \abs{\bfA(\nabla \bfu) -
      \bfA(\bfP_0)}\dx \bigg)^{p'}+c\varepsilon^{1-m}\abs{\mean{\bfA(\nabla \bfu)}_{2B}-\bfA(\bfP_0)}^{p'}
      \\
    &\quad
      +c \varepsilon^{1-m}\,\dashint_{4B} \abs{\bfF -\bfF_0}^{p'}\dx \\
       &\quad \leq c \varepsilon^{1-m}\,\bigg(\dashint_{4B} \abs{\bfA(\nabla \bfu) -
      \bfA(\bfP_0)}\dx \bigg)^{p'}+c\varepsilon^{1-m}\,\dashint_{4B} \abs{\bfF
      -\bfF_0}^{p'}\dx.
  \end{align*}
  This establishes inequality \eqref{may2}.
\end{proof}

 The next lemma provides us with an estimate for a distance between the gradient of
a solution  $\bfu$ to
 \eqref{eq:sysA}, and the gradient  of the
 solution  $\bfv$ to the associated Dirichlet problem \eqref{eq:hareq} in a form suitable for our purposes.

\begin{lemma}
  \label{lem:choice}
Let $p \in (1,\infty)$. Let $\bfu$ be a local weak solution to
\eqref{eq:sysA} satisfying \eqref{eq:deg} for some
 ball  $B$ and  some $\varepsilon >0$. Let $\bfv$ be the solution to the Dirichlet problem \eqref{eq:hareq}.
 Then for every $\delta >0$, there exists a constant $c=c(n, N, p,
 \varepsilon, \delta)$ such that
%
  \begin{align}\label{nov9}
\dashint_B\abs{\bfA(\nabla \bfu)-\bfA(\nabla \bfv)}^{p'}\dx\leq \,\delta\,
    \bigg(\dashint_{4B}
    \abs{\bfA(\nabla\bfu)) - \mean{\bfA(\nabla\bfu)}_{4B}}\,dx\bigg)^{p'} + c \dashint_{4B}\abs{\bfF
    -\bfF_0}^{p'}\dx
  \end{align}
   for every $\bfF_0 \in \setR^{N \times n}$.
  \end{lemma}\noindent
\begin{proof}
Fix $\gamma >0$, and define $m$ as
in \eqref{m}. Then the following chain holds:
\begin{align}\label{july200}
\dashint_B & \abs{\bfA(\nabla\bfu) - \bfA(\nabla \bfv)}^{p'}\dx \leq
c \gamma ^{1-m}\dashint_B \varphi _{p', |\nabla
\bfu|^{p-1}}\big(\bfA(\nabla\bfu) - \bfA(\nabla \bfv)\big) \dx +
\gamma \dashint_B \abs{\bfV(\nabla \bfu)}^2\dx \\
\nonumber & \leq c\gamma ^{1-m} \dashint_{B} \abs{\bfV(\nabla\bfu) - \bfV(\nabla
\bfv)}^2\dx + \gamma \dashint_B \abs{\bfV(\nabla \bfu)}^2\dx
\\
\nonumber & \leq c \gamma ^{1-m} \dashint_{B} \abs{\bfV(\nabla\bfu)
- \bfV(\nabla \bfv)}^2\dx + \frac {\gamma}\varepsilon \dashint_{2B}
\abs{\bfV(\nabla\bfu) - \mean{\bfV(\nabla \bfu)}_{2B}}^2\dx\,,
\end{align}
where the first inequality follows from \eqref{eq:shiftchb} applied
with  $\bfP =0$, $\bfQ=\nabla \bfu$, $t=\bfA(\nabla\bfu) -
\bfA(\nabla \bfv)$, the second inequality from \eqref{eq:hammera}
applied with $\bfP = \nabla \bfv$, $\bfQ = \nabla \bfu$,  and the
last  one from \eqref{eq:deg}.
\\
Let
$\bfP$ be such that $ \bfA(\bfP)=\mean{\bfA(\nabla \bfu)}_{2B}$.
Given $\beta >0$, by Lemma \ref{lemma:comp1} and inequality
\eqref{mean2} one has that
\begin{align}\label{july201}
c \gamma ^{1-m} &\dashint_{B} \abs{\bfV(\nabla\bfu) - \bfV(\nabla
\bfv)}^2\dx + \frac {\gamma}\varepsilon \dashint_{2B}
\abs{\bfV(\nabla\bfu) - \mean{\bfV(\nabla \bfu)}_{2B}}^2\dx
\\ \nonumber \leq &
 c \gamma
^{1-m} \beta\,
     \varphi_{p',\abs{\bfA(\bfP)}}\bigg(\dashint_{2B}
    \abs{\bfA(\nabla\bfu) - \bfA(\bfP)}\dx\bigg)
    +  c\gamma^{1-m} c_\beta \dashint_{2B} \varphi_{p',\abs{\bfA(\bfP)}}(\abs{\bfF -\bfF_0}) \dx
\\ \nonumber & \quad + \frac
{c\gamma}\varepsilon \dashint_{2B} \abs{\bfV(\nabla\bfu) -
\bfA(\bfP)}^2\dx\,.
\end{align}
On the other hand, inequality \eqref{eq:shiftchb}, applied with
$\bfP =\bfP$, $\bfQ=0$ and $t=\abs{\bfF -\bfF_0}$, and inequality
\eqref{eq:jensen} tell us that, if $\tau
>0$,  then
\begin{align}\label{july202}
\dashint_{2B} & \varphi_{p',\abs{\bfA(\bfP)}}(\abs{\bfF -\bfF_0})
\dx \leq c \tau ^{1-m}\dashint_{2B}  \abs{\bfF -\bfF_0}^{p'}\dx +
\tau \abs{\bfV (\bfP)}^2 \dx
\\ \nonumber & \leq  c \tau ^{1-m}\dashint_{2B}  \abs{\bfF
-\bfF_0}^{p'}\dx + \frac{\tau}{\varepsilon} \dashint_{2B}
\abs{\bfV(\nabla \bfu) - \bfV(\bfP)}^2\dx.
\end{align}
Combining inequalities \eqref{july200}--\eqref{july202}, and then
making use of Lemma \ref{lem:RHdeg}, yield
\begin{align}\label{july203}
\dashint_B & \abs{\bfA(\nabla\bfu) - \bfA(\nabla \bfv)}^{p'}\dx \leq
c \gamma ^{1-m} \beta\,
     \varphi_{p',\abs{\bfA(\bfP)}}\bigg(\dashint_{2B}
    \abs{\bfA(\nabla\bfu) - \bfA(\bfP)}\dx\bigg)
    \\ \nonumber &  \frac c\varepsilon\left( {\gamma} + \gamma
    ^{1-m} c_\beta {\tau}\right)\dashint_{2B}
\abs{\bfV(\nabla \bfu) - \bfV(\bfP)}^2\dx + c \gamma ^{1-m} c_\beta
\tau^{1-m}\dashint_{2B} \abs{\bfF -\bfF_0}^{p'}\dx
\\ \nonumber & \leq
c \gamma ^{1-m} \beta  \varepsilon ^{1-m}
     \bigg(\dashint_{4B} \abs{\bfA(\nabla \bfu) -
       \mean{\bfA(\nabla \bfu)}_{4B}}\dx \bigg)^{p'}
+  c \gamma ^{1-m} \beta  \varepsilon ^{1-m}
  \dashint_{4B} \abs{\bfF -\bfF_0}^{p'} \dx
 \\ \nonumber & +
\frac c\varepsilon \left(\gamma + \gamma ^{1-m} c_\beta
\tau\right)\varepsilon ^{1-m}
     \bigg(\dashint_{4B} \abs{\bfA(\nabla \bfu) -
       \mean{\bfA(\nabla \bfu)}_{4B}}\dx \bigg)^{p'}
\\ \nonumber & + \frac c\varepsilon \left(\gamma + \gamma ^{1-m} c_\beta
\tau\right)\varepsilon ^{1-m} \dashint_{2B} \abs{\bfF
-\bfF_0}^{p'}\dx
+ c \gamma
^{1-m} c_\beta \tau^{1-m}\dashint_{2B} \abs{\bfF -\bfF_0}^{p'}\dx.
\end{align}
Choosing first $\gamma$, then $\beta$ and finally $\tau$
sufficiently small in \eqref{july203} yields \eqref{nov9}.
\end{proof}

\begin{proposition}
\label{pro:2} Let $p \in (1,\infty)$. Let $\bfu$ be a local weak
solution to \eqref{eq:sysA} satisfying \eqref{eq:deg} for some
 ball  $B$ and  some  $\varepsilon >0$, and let $\bfv$ be the solution to the Dirichlet problem
 \eqref{eq:hareq}. Let $\alpha$ be the exponent appearing in Theorem
 \ref{thm:decay}.
Then, for every $\theta \in (0,1)$ and every $\kappa <\min\set{1,
\tfrac {2\alpha}{p'}}$, there exist constants $c=c(n, N, p,
\kappa)$, and $c' = c'(n, N, p, \varepsilon, \theta, \kappa)$ such
that
\begin{align}
\label{nov10}  \dashint_{\theta B} & \abs{\bfA(\nabla
\bfu)-\mean{\bfA(\nabla \bfu)}_{\theta B}}^{p'}\dx
\\ \nonumber & \qquad  \leq c \theta^{\kappa p'} \bigg(\dashint_{4B}\abs{\bfA(\nabla \bfu)-\mean{\bfA(\nabla
\bfu)}_{{4B}}}\dx\bigg)^{p'} +c'  \dashint_{4B} \abs{\bfF
-\bfF_0}^{p'} \dx
\end{align}
for every $\bfF_0\in \RNn$.
\end{proposition}
\begin{proof}
 By property \eqref{meanq},
 \begin{align}\label{nov11}
  \dashint_{\theta B}\abs{\bfA(\nabla \bfu)-\mean{\bfA(\nabla \bfu)}_{\theta B}}^{p'}\dx
  \leq &
    c\dashint_{\theta B}\abs{\bfA(\nabla \bfv)-\mean{\bfA(\nabla \bfv)}_{\theta
    B}}^{p'} \dx \\ \nonumber
& + c \theta^{-n} \dashint_B\abs{\bfA(\nabla \bfu)-\bfA(\nabla
\bfv)}^{p'}\dx =  I_1+I_2.
 \end{align}
We exploit inequality \eqref{pharmonic2},   property \eqref{meanq},
and Lemma~\ref{lem:choice} to obtain the following chain of
inequalities:
\begin{align}\label{nov12}
I_1 & \leq c \theta ^{\kappa p'}
\bigg(\dashint_{B}\abs{\bfA(\nabla \bfv)-\mean{\bfA(\nabla
\bfv)}_{B}}\dx\bigg)^{p'}
\\ \nonumber & \leq \widetilde c\theta ^{\kappa p'}
\dashint_{B}\abs{\bfA(\nabla \bfv)- \bfA(\nabla \bfu)}^{p'}\dx +
\widetilde  c\theta ^{2\alpha} \bigg(\dashint_{B}\abs{\bfA(\nabla
\bfu)-\mean{\bfA(\nabla \bfu)}_{B}}\dx\bigg)^{p'}
\\ \nonumber & \leq
\widetilde c\theta ^{\kappa p'}
\bigg(\dashint_{4B}\abs{\bfA(\nabla \bfu)-\mean{\bfA(\nabla
\bfu)}_{4B}}\dx\bigg)^{p'} + c' \dashint_{4B}\abs{\bfF
-\bfF_0}^{p'}\dx,
\end{align}
where $\widetilde c = \widetilde c(n,p,N,\kappa)$ and $c' =
c'(n,p,N,\varepsilon,\kappa)$.
 One can further make use of Lemma
\ref{lem:choice}
to infer that
\begin{align}\label{nov13}
I_2 \leq c \delta \theta^{-n}\bigg( \dashint_{4B}\abs{\bfA(\nabla
\bfu)-\mean{\bfA(\nabla \bfu)}_{{4B}}}\dx\bigg)^{p'} +
\,c_{\delta} \theta^{-n} \dashint_{4B}\abs{\bfF -\bfF_0}^{p'}\dx
\end{align}
for any $\delta>0$, and for some constant $c_\delta = c_\delta (n,
N, p, \varepsilon, \delta)$. On choosing
$\delta = \theta^{n+\kappa p'}$,
inequalities \eqref{nov11}--\eqref{nov13} yield the result.
%
%
%
\end{proof}

\subsection{The non-degenerate case}\label{sec:comparison-nondeg}

In the present subsection, we assume that $\bfu$ is a local weak
solution to \eqref{eq:sysA} such that
\begin{align}
\label{eq:non-deg}
 \dashint_{2B}|\bfV(\nabla \bfu)-\mean{\bfV(\nabla \bfu)}_{ 2B}|^2\dx < \varepsilon\dashint_{2B}\abs{\bfV(\nabla \bfu)}^2\dx.
\end{align}
for some fixed reference
 ball  $B$ such that $4B \subset \Omega$, and some number $\varepsilon \in (0,\tfrac 14)$.
 Since in the degenerate case appropriate estimates are possible for any $\varepsilon$, we can
 consider it to be a free parameter at this stage. In particular,
 our choice of $\varepsilon$ in \eqref{eq:non-deg} will depend on
 the parameter $\theta$ in Proposition~\ref{prop:decay conclusion}.

Let us start by introducing a few notations to be used in what
follows. We shall have to deal with balls whose centers differ from
the reference ball $B$ in \eqref{eq:non-deg}. We call $\sigma B(z)$
the ball, centered at the point $z$, whose radius is $\sigma$ times
the radius of  $B$.
%
 Given $z \in \Rn$ and $\sigma >0$,
we denote by $\bfA_{\sigma,z},\bfV_{\sigma,z},\bfU_{\sigma,z}$ the
matrices from $\mathbb R^{N \times n}$  satisfying
\begin{equation} \label{eq:means}
\bfA(\bfA_{\sigma,z}) = \mean{\bfA(\nabla \bfu)}_{\sigma B(z)},
\quad
\bfV(\bfV_{\sigma,z}) = \mean{\bfV(\nabla \bfu)}_{\sigma B(z)},
\quad
\bfU_{\sigma,z} = \mean{\nabla \bfu}_{\sigma B(z)},
\end{equation}
respectively.  Whenever $z$ is the center of the reference ball $B$,
it will be  omitted, and we just write
\begin{equation} \label{eq:means1}
\bfA(\bfA_{\sigma}) = \mean{\bfA(\nabla \bfu)}_{\sigma B}, \quad
\bfV(\bfV_{\sigma}) = \mean{\bfV(\nabla \bfu)}_{\sigma B}, \quad
\bfU_{\sigma} = \mean{\nabla \bfu}_{\sigma B}.
\end{equation}

 The next lemma tells us  that, under  condition~\eqref{eq:non-deg},  the expressions defined above are equivalent, up to multiplicative constants
  proportional to $\varepsilon$.
\begin{lemma} \label{lemma:means close}
Let $p \in (1, \infty)$. Let $\bfu$ be a local weak solution
to~\eqref{eq:sysA} satisfying \eqref{eq:non-deg} for some ball $B$
and some $\varepsilon \in(0, \tfrac 14)$.
 Then there exists a constant $c=c(n, N, p)$  such that, if
\begin{equation} \label{eq:epsilon cond 1}
  \varepsilon c  \leq \sigma^{5n}\,,
\end{equation}
then
\begin{equation} \label{eq:means similar}
\max\set{\abs{\bfU_{\theta,z}},\abs{\bfA_{\theta,z}}} \leq 2\abs{\bfV_2}\leq 4\min\set{\abs{\bfU_{\theta,z}},\abs{\bfA_{\theta,z}}},
\end{equation}
and
\begin{equation} \label{eq:means similar bis}
\max\set{\abs{\bfU_{2}},\abs{\bfA_{2}}} \leq
2\abs{\bfV_{\theta,z}}\leq 4\min\set{\abs{\bfU_{2}},\abs{\bfA_{2}}},
\end{equation}
 for every $\theta \in [\sigma, 2]$ and $z \in B$ satisfying
$\theta B(z) \subset 2B$. Moreover, $\bfU_2 \neq 0$, and
\begin{equation} \label{eq:means close}
\abs{\bfA_{\theta,z} - \bfU_2} +  \abs{\bfV_{\theta,z} - \bfU_2} +
\abs{\bfU_{\theta,z} - \bfU_2} \leq \tfrac 18 \sigma^{2n}
\abs{\bfU_2}
\end{equation}
 for every $\theta \in [\sigma,2]$.
%
\end{lemma}
\begin{proof}
By condition \eqref{eq:non-deg},
\begin{align*}
\abs{\bfV(\bfV_2)}^2 \leq & \dashint_{2B} \abs{ \bfV(\nabla \bfu)
}^2\dx \\ \leq & 2 \dashint_{2B} \abs{ \bfV(\nabla \bfu)-
\bfV(\bfV_2)}^2 \dx + 2 \abs{\bfV(\bfV_2)}^2 < 2 \varepsilon
\dashint_{2B} \abs{ \bfV(\nabla \bfu) }^2 \dx +  2
\abs{\bfV(\bfV_2)}^2.
\end{align*}
Hence,
\begin{equation} \label{eq:means001}
\abs{\bfV(\bfV_2)}^2 \leq \dashint_{2B} \abs{ \bfV(\nabla \bfu) }^2\dx < 4 \abs{\bfV(\bfV_2)}^2,
\end{equation}
since we are assuming that $\varepsilon \leq \tfrac 14$. Hence, in
particular, $\abs{\bfV(\bfV_2)}>0$.
\\ Next, denote by $\bfT_{\theta, z}$ any of the quantities $\bfA_{\theta,z}$, $\bfV_{\theta,z}$, $\bfU_{\theta,z}$ for $\theta
\in [\sigma ,2]$. Then the following chain holds:
%
\begin{align}\label{june20}
\abs{\bfV(\bfV_2) - \bfV(\bfT_{\theta, z})}^2
 \leq &
 2 \dashint_{\theta B} \abs{\bfV(\nabla \bfu) - \bfV(\bfT_{\theta, z})}^2\dx  + 2\dashint_{\theta B} \abs{\bfV(\nabla \bfu) - \bfV(\bfV_2)}^2\dx
\\  \nonumber \leq &
c  \dashint_{\theta B} \abs{\bfV(\nabla \bfu) - \bfV(\bfV_\theta)}^2\dx + c \theta^{-n}
\dashint_{2B} \abs{\bfV(\nabla \bfu) - \bfV(\bfV_2)}^2\dx
\\ \nonumber  \leq &
c \theta^{-n} \dashint_{2B} \abs{\bfV(\nabla \bfu) - \bfV(\bfV_2)}^2\dx
\\  \nonumber \leq &
 c \varepsilon \theta^{-n} \dashint_{2B} \abs{\bfV(\nabla \bfu)
}^2\dx \leq \hat{c} \varepsilon  \theta^{-n} \abs{\bfV(\bfV_2)}^2\,,
\end{align}
where the second inequality is a consequence of
Lemma~\ref{lem:trick}, the third of \eqref{meanq}, the fourth of
\eqref{eq:non-deg}, and the last one of \eqref{eq:means001}. Hence,
via  the triangle inequality, and
\eqref{eq:epsilon cond 1} with sufficiently large $c$, we obtain
that
 \[
 \abs{\bfT_{\theta, z}}^p=\abs{\bfV(\bfT_{\theta, z})}^2 \leq \big(1+\tfrac{1}{2}\big)\abs{\bfV(\bfV_2)}^2
 \]
 and
 \[
 \abs{\bfV(\bfV_2)}^2\leq \abs{\bfT_{\theta, z}}^p+ \frac{1}{2}\abs{\bfV(\bfV_2)}^2.
 \]
Therefore
\begin{equation}\label{june21} \abs{\bfT_{\theta, z}}\leq
2\abs{\bfV_2}\leq 4\abs{\bfT_{\theta, z}}.
\end{equation}
  Equation \eqref{eq:means similar} is thus established.
\\ As far as \eqref{eq:means close} is concerned, observe that, by  \eqref{eq:hammera} and
\eqref{june20},
\begin{align}\label{june22}
\abs{\bfV_2}^{p-2} \abs{\bfV_2 - \bfT_{\theta, z}}^2 \leq c
\abs{\bfV(\bfV_2) - \bfV(\bfT_{\theta, z})}^2 \leq \epsilon
c'\sigma^{-n} \abs{\bfV_2}^p.
\end{align}
 Coupling equations \eqref{june22} and
\eqref{eq:means similar} tells us that
\begin{align*}
\abs{\bfV_2 - \bfT_{\theta, z}} \leq \sqrt{\epsilon c \sigma^{-n}}
\abs{\bfV_2} \leq 2\sqrt{\epsilon c\sigma^{-n}}\abs{\bfU_2}\,.
\end{align*}
Hence,
\begin{align}\label{june23}
\abs{\bfV_2 - \bfT_{\theta, z}} \leq 2\sqrt{\epsilon c'
\sigma^{-n}}\abs{\bfU_2}\,,
\end{align}
\begin{align}\label{june24}
 \abs{\bfV_2 - \bfU_2} \leq
2\sqrt{\epsilon c' \sigma^{-n}}\abs{\bfU_2}\,.
\end{align}
Inequality \eqref{eq:means close} follows from \eqref{june23} and
\eqref{june24}, via the triangle inequality,  on choosing $c$ in
\eqref{eq:epsilon cond 1} larger than a suitable absolute constant
times $\max \{c', \hat {c}\}$.
\\ Finally, equation \eqref{eq:means similar bis} can be derived
from \eqref{eq:means similar} and \eqref{eq:means close}.
\end{proof}


\begin{lemma} \label{lem:F big}
Let $p \in (1, \infty)$. Assume that $\bfu$ is a local weak solution
to~\eqref{eq:sysA} satisfying \eqref{eq:non-deg} for some ball $B$
and some $\varepsilon \in (0, \tfrac 14)$. Assume, in addition, that
there exists a constant $\overline c$ such that
%
\begin{equation} \label{eq:F big}
\abs{\mean{\bfA(\nabla \bfu)}_{2B} }^{p'} \leq \frac{\overline
c^{p'}}{\varepsilon} \dashint_{2B} \abs{\bfF-\bfF_0}^{p'}\dx
\end{equation}
for every $\bfF _0\in \mathbb R^{N \times n}$.
 Then there exists a constant $c=c(n, N, p)$ such that
\begin{equation} \label{eq:F big concl}
 \dashint_{\sigma B} \abs{\bfA(\nabla \bfu) - \mean{\bfA(\nabla \bfu)}_{\sigma B}}^{p'}\dx \leq c \frac{\overline c^{p'}}{\varepsilon \sigma^{n}} \dashint_{2B} \abs{\bfF-\bfF_0}^{p'}\dx
\end{equation}
for every  $\sigma$ satisfying \eqref{eq:epsilon cond 1}.
\end{lemma}

\begin{proof}
One has that
\begin{align*}
 \dashint_{\sigma B} \abs{\bfA(\nabla \bfu) - \bfA(\bfA_\sigma)}^{p'}\dx
 \leq  & c \sigma^{-n} \dashint_{2B} \abs{\bfA(\nabla \bfu) - \bfA(\bfA_2)}^{p'}\dx
   \leq
c \sigma^{-n}\bigg( \dashint_{2B} \abs{\bfV(\nabla \bfu)}^2 \dx
+  \abs{\bfA_2}^{p} \bigg)
\\ \leq  &
c \sigma^{-n}\big(\abs{\bfV(\bfV_2)}^2+ \abs{\bfA_2}^{p} \big)
\end{align*}
where the first inequality is due to \eqref{meanq}, the third to
\eqref{eq:means001} and the last one to Lemma \ref{lemma:means close}. Hence, the result follows via~\eqref{eq:F big}.
\end{proof}

In view of \eqref{eq:means001} and \eqref{eq:means similar},
assumption \eqref{eq:non-deg}
implies that
\begin{align}
\label{eq:non-deg'}
 \dashint_{2B}|\bfV(\nabla \bfu)-\mean{\bfV(\nabla \bfu)}_{ 2B}|^2\dx < c \varepsilon \abs{\mean{\bfA(\nabla \bfu)}_{2B}}^{p'}
\end{align}
for some constant $c=c(n,N,p)$,
provided that~\eqref{eq:epsilon cond 1} holds. Moreover, Lemma
\ref{lem:F big} enables us to argue under the additional condition
that
\begin{equation} \label{eq:F small}
\overline c^{p'} \dashint_{2B} \abs{\bfF-\bfF_0}^{p'}\dx \leq
\varepsilon \abs{\mean{\bfA(\nabla \bfu)}_{2B}}^{p'}
\end{equation}
for some  given constant $\overline c$. Equations
\eqref{eq:non-deg'} and \eqref{eq:F small} amount to requiring that
both $\bfV(\nabla \bfu)$ and $\bfF$ are small compared to the
averages of $\bfA(\nabla \bfu)$. Next, given $\sigma
>0$, we choose
\begin{equation}\label{csigma}
\overline c =c_\sigma
\end{equation}
 in \eqref{eq:F small}, where $c_\sigma$ is the
constant appearing in Lemma~\ref{lem:new} for  $\theta = \sigma$.
Also, we introduce the notation:
 \begin{equation} \label{eq:E def}
E =   \dashint_{2B} \abs{\bfA(\nabla \bfu) -
\mean{\bfA(\nabla\bfu)}_{2B}}\dx +  \bigg(\sigma^{-2\alpha} c_\sigma
\dashint_{2B} \abs{\bfF -\bfF_0}^{p'}\dx\bigg)^{1/p'}.
\end{equation}

\smallskip
\par

Our key decay estimate in the non-degenerate case, for $p \in (1,
2)$, is contained in the following proposition.

\begin{proposition} \label{lemma:non-deg decay 1 p<2}
Assume that $p \in (1, 2)$. Let $\bfu$ be a local weak solution to
\eqref{eq:sysA} satisfying \eqref{eq:non-deg} and~\eqref{eq:F small}
for some ball $B$, some $\varepsilon \in (0, \tfrac 14)$, and
$\overline c =c_\sigma$ for some $\sigma$ as in \eqref{eq:epsilon
cond 1}. Here, $c_\sigma$ denotes the constant appearing in
Lemma~\ref{lem:new}, with $\theta = \sigma$. Then there exist
constants $c=c(n, N, p)$ and $\alpha = \alpha(n,p,N)\in(0,1)$ such
that
 \begin{multline}
 \label{eq:key2 non-deg p<2}
 \dashint_{\theta B} \abs{\bfA(\nabla\bfu) - \mean{\bfA(\nabla \bfu)}_{\theta B}}^{2}\dx
  \\ \leq c \,  \theta^{\frac{4\alpha}{p'}} \bigg[\bigg(\dashint_{2B}
\abs{\bfA(\nabla \bfu) - \mean{\bfA(\nabla\bfu)}_{2B}}\dx\bigg)^{2}
+ \bigg(\sigma^{-2\alpha} c_\sigma \dashint_{2B} \abs{\bfF
-\bfF_0}^{p'}\dx\bigg)^{\frac 2{p'}}\bigg]
 \end{multline}
 for every $\theta \in [\sigma,\frac{1}{4})$ and $z \in B$.
\end{proposition}

\begin{proof}
One has that
%
%
\begin{align}\label{june12}
& \dashint_{2B} \abs{\bfA(\nabla \bfu) -
\mean{\bfA(\nabla\bfu)}_{2B}}\dx \leq \bigg(\dashint_{2B}
\abs{\bfA(\nabla \bfu) -
\mean{\bfA(\nabla\bfu)}_{2B}}^{p'}\dx\bigg)^\frac{1}{p'}
\\ \nonumber & \qquad \leq\bigg(\dashint_{2B} \varphi_{p',\abs{\mean{\bfA(\nabla\bfu)}_{2B}}}\Big(\abs{\bfA(\nabla \bfu) -
\mean{\bfA(\nabla\bfu)}_{2B}}\Big)\dx\bigg)^\frac{1}{p'}\\
\nonumber  & \qquad
 \leq c\bigg(\dashint_{2B} \abs{\bfV(\nabla \bfu) - \mean{\bfV(\nabla\bfu)}_{2B}}^2\dx\bigg)^\frac{1}{p'}
\leq c \varepsilon^\frac{1}{p'} \abs{ \mean{\bfA(\nabla \bfu)}_{2B}
},
\end{align}
where the second inequality holds since $t^{p'} \leq \varphi_{p', a}
(t)$ when $p'\geq 2$,  $a >0$, $t \geq 0$, the third inequality
holds by \eqref{eq:hammera} and Lemma~\ref{lem:trick}, and the last
inequality is a consequence of \eqref{eq:non-deg'}.
%
%
Coupling \eqref{june12} with ~\eqref{eq:F small},
and taking into account the fact
that  $\varepsilon <1$, imply that
\begin{equation} \label{eq:E small}
E\leq c\varepsilon^\frac1{p'}\abs{ \mean{\bfA(\nabla \bfu)}_{2B} },
\end{equation}
where $E$ is defined by \eqref{eq:E def}. Now, observe that
\begin{align}\label{17july}
\varphi_{p',\abs{\bfA(\bfA_2)}}(t) \leq c t^{p'} + c
\abs{\bfA(\bfA_2)}^{p'-2}t^2\quad\text{for all} \quad t\geq0
\end{align}
inasmuch as $p'\geq2$. Therefore, owing to \eqref{eq:E small},
\begin{align}\label{june15}
 & \dashint_{2B} \varphi_{p',\abs{\bfA(\bfA_2)}}(\abs{\bfF -\bfF_0})\dx \leq c \abs{\bfA(\bfA_2)}^{p'-2}
  \dashint_{2B} \abs{\bfF -\bfF_0}^2 \dx +  c\dashint_{2B} \abs{\bfF -\bfF_0}^{p'}
  \dx\\ \nonumber
&\quad \leq c \abs{\bfA(\bfA_2)}^{p'-2}  \bigg(\dashint_{2B}
\abs{\bfF -\bfF_0}^{p'}\bigg)^\frac{2}{p'}+ c\sigma^{2\alpha}E^{p'}
\leq c \theta^{\frac{4\alpha}{p'}}
\abs{\bfA(\bfA_2)}^{p'-2} E^2 .
 \end{align}
An application of  Lemma~\ref{lem:new}, with $B$ replaced by
$\frac12 B(z)$, tells us that, if  $\theta\in [\sigma,\frac14)$,
then
\begin{align}\label{june14}
 \dashint_{\theta B(z)} \abs{\bfV(\nabla\bfu) - \bfV(\bfU_{\theta,z})}^2\dx
\leq & c\, \theta^{2\alpha}
\varphi_{p',\abs{\bfA(\bfA_{1,z}}}\bigg(\dashint_{B(z)}
\abs{\bfA(\nabla \bfu) -
        \mean{\bfA(\nabla\bfu)}_{B(z)}}\dx\bigg)
\\ \nonumber
      &\quad+ c_\sigma \dashint_{B(z)} \varphi_{p',\abs{\bfA(\bfA_{1.z})}}(\abs{\bfF -\bfF_0})\dx
\end{align}
for any $z \in B$. Since $B(z)\subset 2B$, by Lemma~\ref{lemma:means
close} one has that $\abs{\bfA_{1,z}}\leq2\abs{\bfA_2}\leq
4\abs{\bfA_{1,z}}$. Hence, by \eqref{june14},
\begin{align}
\label{eq:key2 non-deg 1}
 \dashint_{\theta B(z)} \abs{\bfV(\nabla\bfu) -  \bfV(\bfU_{\theta,z})}^2\dx
\leq & c\, \theta^{2\alpha}
\varphi_{p',\abs{\bfA(\bfA_{2})}}\bigg(\dashint_{2B}
\abs{\bfA(\nabla \bfu) -
        \mean{\bfA(\nabla\bfu)}_{2B}}\dx\bigg)
\\ \nonumber
      &\quad+ c_\sigma \dashint_{2B} \varphi_{p',\abs{\bfA(\bfA_{2})}}(\abs{\bfF -\bfF_0})\dx.
\end{align}
Starting from \eqref{eq:key2 non-deg 1}, and making use of
 \eqref{mean2},  \eqref{17july} and \eqref{june15} tell
us that
 \begin{align*}
 \dashint_{\theta B(z)} \abs{\bfV(\nabla\bfu) - \mean{\bfV(\nabla \bfu)}_{\theta B(z)}}^2\dx
\leq & c\,
\theta^{\frac{4\alpha}{p'}}
 \abs{\bfA(\bfA_2)}^{p'-2} E^2\,.
\end{align*}
Thus,
 \begin{align}\label{june16}
 \dashint_{\theta B} \abs{\bfA(\nabla\bfu) -
 \bfA(\bfA_\theta)}^{2}\dx \leq &  \dashint_{\theta B} \abs{\bfA(\nabla\bfu) - \bfA(\bfV_\theta)}^{2}\dx
\\ \nonumber \leq &
 c  \dashint_{\theta B} (\abs{\nabla\bfu} + \abs{\bfV_\theta})^{2(p-2)}  \abs{\nabla\bfu - \bfV_\theta}^2\dx
 \\ \nonumber \leq &
 c \abs{\bfV_\theta}^{p-2}  \dashint_{\theta B}  (\abs{\nabla\bfu} + \abs{\bfV_\theta})^{(p-2)}\abs{\nabla\bfu - \bfV_\theta}^2\dx
 \\ \nonumber \leq &
 c \abs{\bfV_\theta}^{p-2}  \dashint_{\theta B} \abs{\bfV(\nabla\bfu) - \mean{\bfV(\nabla \bfu)}_{\theta B}}^2\dx
 \\ \nonumber \leq &
  c\,  \theta^{\frac{4\alpha}{p'}}
 \abs{\bfA_2}^{p-2}  \abs{\bfA(\bfA_2)}^{p'-2} E^2
   \\ \nonumber \leq &
  c\,  \theta^{\frac{4\alpha}{p'}}
 E^2 ,
 \end{align}
where the first inequality holds owing to \eqref{meanq}, the second
to \eqref{eq:hammerd-bis}, the third to the fact that $p-2<0$, the
fourth to \eqref{eq:hammera}, the fifth to Lemma \ref{lemma:means
close}, and the last one to the fact that $p-2 + (p'-2)(p-1)=0$.
%
%
This establishes \eqref{eq:key2 non-deg p<2}.
  \end{proof}

 The remaining
part of the present section is devoted to a counterpart of
inequality \eqref{eq:key2 non-deg p<2} in the case when $p \in [2,
\infty)$, which requires some further steps.
\\ A first decay conclusion in the spirit of \eqref{eq:key2 non-deg p<2}, but with $\bfA(\nabla\bfu)$ replaced by
$\bfV(\nabla\bfu)$, reads as follows.

\begin{lemma} \label{lemma:non-deg decay 1 p>2}
Assume that $p \in [2, \infty)$. Let $\bfu$ be a local weak solution
to \eqref{eq:sysA} satisfying \eqref{eq:non-deg} and~\eqref{eq:F
small}  for some ball $B$, some $\varepsilon \in (0, \tfrac 14)$,
and $\overline c =c_\sigma$ for some $\sigma$ as in
\eqref{eq:epsilon cond 1}. Here, $c_\sigma$ denotes the constant
appearing in Lemma~\ref{lem:new}, with $\theta = \sigma$. Then there
exist constants $c=c(n, N, p)$ and $\alpha = \alpha(n,p,N)\in(0,1)$
such that
%
 \begin{multline}
\label{eq:key2 non-deg p>2}
 \dashint_{\theta B(z)} \abs{\bfV(\nabla\bfu) - \mean{\bfV(\nabla \bfu)}_{\theta B(z)}}^2\dx
\\ \leq  c \, \theta^{2\alpha} \bigg[\bigg(\dashint_{2B}
\abs{\bfA(\nabla \bfu) - \mean{\bfA(\nabla\bfu)}_{2B}}\dx\bigg)^{p'}
+ \sigma^{-2\alpha} c_\sigma \dashint_{2B} \abs{\bfF
-\bfF_0}^{p'}\dx\bigg]
\end{multline}
for every $\theta \in [\sigma,\frac14)$ and  $z \in B$.
\end{lemma}
\begin{proof}
An inspection of the proof of inequality \eqref{eq:key2 non-deg 1}
reveals that it holds, in fact, also for $p \in [2, \infty)$. On the
other hand, for these values of $p$,
  we
have that $\varphi_{p',\abs{\bfA(\bfA_2)}}(t) \leq t^{p'}$ for
$t\geq 0$. Inequality \eqref{eq:key2 non-deg p>2} thus follows from
\eqref{eq:key2 non-deg 1}.
\end{proof}

The key idea which enables us to turn the decay estimate for
$\bfV(\nabla\bfu)$ contained in Lemma \ref{lemma:non-deg decay 1
p>2} into the desired estimate for $\bfA(\nabla\bfu)$ is to exploit
a linearization argument. Specifically,
 denote by $D\bfA :
\mathbb{R}^{N \times n} \to \mathbb R^{N \times n} \times \mathbb
R^{N \times n}$ the differential of the map $\bfA$, and  let $\bfz$
be the solution to the linear Dirichlet problem
\begin{align}\label{eq:linear}
   \begin{cases}
    \divergence (D\bfA(\bfA_{2})\nabla \bfz )=0 & \text{in}\quad B,
    \\
    \bfz =\bfu & \text{on}\quad\partial B ,
    \end{cases}
    \end{align}
    where $\bfA_{2}$ is the (constant) matrix defined as in
    \eqref{eq:means1}. We shall compare the gradient of the solution
    $\bfu$ to the original system \eqref{eq:sysA} with the gradient
    of $\bfz$.
\\ To begin with, recall that, since $\bfz$ solves a linear uniformly elliptic system with constant
coefficients, the standard linear theory provides us with the decay
estimate
\begin{equation} \label{eq:z decay}
\sup_{x,x' \in \theta B} \abs{\nabla \bfz(x) - \nabla\bfz(x')} \leq
\theta c \dashint_{B} \abs{\nabla \bfz - \mean{\nabla \bfz}_{B}} \dy
\end{equation}
for any $\theta \in (0,1)$, for some constant $c = c(n,p,N)$.
\\ Next, observe that, on defining the function $\bfH : \mathbb R^{N \times
n} \times \mathbb R^{N \times n} \to \mathbb R^{N \times n}$ as
\begin{equation} \label{eq:H}
\bfH(\bfP,\bfQ) = \bfA(\bfP) - \bfA(\bfQ) - D\bfA(\bfQ) (\bfP-\bfQ)
\quad \hbox{for $(\bfP, \bfQ) \in \mathbb R^{N \times n} \times
\mathbb R^{N \times n}$,}
\end{equation}
 system \eqref{eq:sysA} can be rewritten  as
\begin{align}
 \divergence(D\bfA(\bfA_{2}) \nabla (\bfu - \bfz))  =  - \divergence\left(\bfH(\nabla \bfu,\bfA_{2})\right)+  \divergence\left(\bfF-\bfF_0
 \right)\,.
\end{align}
The following classical result (see e.g. \cite[Lemma~2]{DolM95})
applies to systems of the form \eqref{eq:H}.
\begin{theorem}
 \label{thm:CZ} {\bf [\Calderon --Zygmund]}
Let $B$ be a ball in $\Rn$. Let $\bfB : \Rn \to \mathbb R^{n\times
n}$ be such that  $\lambda \abs{\bfxi}^2\leq \skp{\bfB
\bfxi}{\bfxi}\leq \Lambda\abs{\bfxi}^2$ for some constants $\Lambda
\geq \lambda >0$, and for every $\bfxi \in \Rn$. Assume that   $r
\in (1, \infty)$, and let $\bfG \in L^r(B , \mathbb R^{N \times
n})$. Let $\bfw$ be the unique solution to the system
\begin{align*}
\begin{cases}
 -\divergence (\bfB\nabla\bfw)=-\divergence  \bfG & \text{ in } B\\
\bfw =0   & \text{ on }\partial B.
\end{cases}
\end{align*}
Then, there exists a constant $c=c(n,r,\tfrac{\Lambda}{\lambda})$
such that
\[
 \lambda\norm{\nabla\bfw}_{L^{r}(B)}\leq c\norm{\bfG}_{L^{r}(B)}.
\]
\end{theorem}
Since the matrix $\bfB$, given by $\bfB =  D\bfA(\bfA_{2})$,
satisfies the assumptions  of Theorem~\ref{thm:CZ} with $\lambda =
\lambda (n, N, p)$ and $\Lambda = \Lambda (n, N, p)$, an application
of this theorem with $\bfw = \bfu - \bfz$ and $r=p'$ yields the
following result.

\begin{lemma} \label{august1}
Assume that $p \in (1, \infty)$, and let $\bfu$ be a local weak
solution to \eqref{eq:sysA}. Let $B$ be a ball such that $B \subset
\Omega$, and let $\bfz$ be the solution to the Dirichlet problem
\eqref{eq:linear}. Then there exists a constant  $c = c(n,N,p)$ such
that
\begin{align} \label{eq:CZ res}
 \abs{\bfA_{2}}^{(p-2)p'} \dashint_{B} \abs{\nabla \bfu -
 \nabla \bfz }^{p'} \dx
  \leq
c \dashint_{B} \abs{\bfF-\bfF_0}^{p'} \dx + c \dashint_{B}
\abs{\bfH(\nabla \bfu,\bfA_{2}) }^{p'} \dx .
\end{align}
\end{lemma}

The error term containing $\bfH$ in~\eqref{eq:CZ res} will be
estimated with the aid of an  algebraic inequality which is the
object of the next lemma.

\begin{lemma} \label{lemma:H}
Let $p \in [2, \infty)$, and let $\bfH$ be the function defined by
\eqref{eq:H}. Then, there exists a constant
$c = c(n,N,p)$ such that
\begin{align}
\label{eq:H est} \bfH(\bfP,\bfQ) \leq & c \delta
\bigg(\frac{\abs{\bfP - \bfQ} }{\delta}\bigg)^{p-1}
\chi_{\set{\abs{\bfP - \bfQ} \geq \delta \abs{\bfQ} }} + c \delta
\abs{\bfA(\bfP) - \bfA(\bfQ)}\chi_{\set{\abs{\bfP - \bfQ} \leq
\delta \abs{\bfQ} }},
\end{align}
for every $\bfP,\bfQ \in \mathbb R^{N \times n}$, and   $\delta \in
(0,1/2]$. Here, $\chi _G$ stands for the characteristic function of
a set $G$.
 %
%
\end{lemma}

\begin{proof}
The function $\bfH$ can be represented as
\begin{align*}
\bfH(\bfP,\bfQ) = & \int_0^1 \left[  D\bfA(\bfQ + t(\bfP-\bfQ)) -D\bfA(\bfQ) \right]\dt\, (\bfP-\bfQ)
\\ = & \abs{\bfQ}^{p-1}
 \int_0^1 \left[  D\bfA(\widetilde \bfQ + t \widetilde \bfR) -D\bfA(\widetilde \bfQ) \right]\dt \, \widetilde \bfR,
\end{align*}
where we have introduced the notations
\[
\widetilde \bfQ = \frac{\bfQ}{\abs{\bfQ}}\,, \qquad \widetilde \bfP
= \frac{\bfP}{\abs{\bfQ}} \,, \qquad \widetilde \bfR = \widetilde
\bfP - \widetilde \bfQ\,.
\]
Computations show  that
\begin{equation} \label{eq:D^2A}
\abs{D^2\bfA(\bfT)} \leq c  \abs{\bfT}^{p-3}
\end{equation}
for some constant $c = c(n,N,p)$ and for every $\bfT \in \mathbb
R^{N \times n} \setminus \{0\}$. Owing to \eqref{eq:D^2A},
\[
\left|  \int_0^1 \left[  D\bfA(\widetilde \bfQ + t \widetilde \bfR)
-D\bfA(\widetilde \bfQ) \right]\dt \right| \leq c \abs{\widetilde
\bfR} \int_0^1\int_0^1 \abs{\widetilde \bfQ + st \widetilde
\bfR}^{p-3} \ds \dt
\] for some constant $c = c(n,N,p)$. Now, if
$\abs{\widetilde \bfR} \leq \delta$, and hence, in particular,
$\abs{\widetilde \bfR} \leq \frac 12$, then
\[
 \int_0^1\int_0^1 \abs{\widetilde \bfQ + s t \widetilde
\bfR}^{p-3} \ds \dt  \leq c.
\]
Since our present assumption on $\abs{\widetilde \bfR}$ is
equivalent to $\abs{\bfP- \bfQ} \leq \delta \abs{\bfQ}$, and hence,
in particular, it implies that $\abs{\bfP} + \abs{\bfQ} \leq 3
\abs{\bfQ}$, we have that
\begin{equation}\label{august2}
\abs{\bfH(\bfP,\bfQ)} \leq c \delta \abs{\bfQ}^{p-2} \abs{\bfP -
\bfQ} \leq c' \delta \abs{\bfA(\bfP) - \bfA(\bfQ)},
\end{equation}
whence  \eqref{eq:H est} follows. Note that, in the derivation of
\eqref{august2},  we have also made use of \eqref{eq:hammerd-bis}.
\\
 On the other hand, if
$\abs{\widetilde \bfR} > \delta$,
then
\[
 \int_0^1\int_0^1 \abs{\widetilde \bfQ + s t  \widetilde
\bfR}^{p-3} \ds \dt  \leq c (1 + \abs{\widetilde \bfR})^{p-3}.
\]
Hence,
\[
\abs{\bfH(\bfP,\bfQ)} \leq c ( \abs{\bfQ} +
\abs{\bfP-\bfQ})^{p-3}\abs{\bfP-\bfQ}^2\,,
\]
and this inequality can be shown to imply \eqref{eq:H est} also in
this case.
\end{proof}

%


Our last preparatory step is  contained in the following lemma.

 \begin{lemma} \label{lemma:non-deg large energy}
Let $p \in [2, \infty)$, and let $\bfu$ be a local weak solution to
\eqref{eq:sysA} satisfying  ~\eqref{eq:non-deg} for some ball $B$,
and some $\varepsilon \in (0, \tfrac 14)$. Assume that $\sigma \in
(0,\tfrac 14)$ is such that \eqref{eq:epsilon cond 1} holds. Let
$\alpha$ be the exponent appearing in the statement of Theorem
\ref{thm:decay}.
 Then there exist constants $c =
c(n,p,N)$ and $\alpha = \alpha(n,p,N)$ such that
\begin{multline} \label{eq:non-deg large energy}
\int_{B \cap \{\abs{\nabla \bfu - \bfA_2} \geq \sigma^{2n}
\abs{\bfA_2} \}} \abs{\bfV(\nabla \bfu) - \bfV(\bfA_2)}^2  \dx
\\ \leq c |B| \sigma^{2\alpha}  \bigg[\bigg(\dashint_{2B}
\abs{\bfA(\nabla \bfu) - \mean{\bfA(\nabla\bfu)}_{2B}}\dx\bigg)^{p'}
+ \sigma^{-2\alpha} c_\sigma \dashint_{2B} \abs{\bfF
-\bfF_0}^{p'}\dx\bigg],
\end{multline}
where $\bfA_2$ is given by \eqref{eq:means1}.
\end{lemma}
\begin{proof} By Besicovich covering theorem,
 there exists a countable covering of the set  $B\cap\{\abs{\nabla \bfu - \bfA_2} \geq \sigma^{2n} \abs{\bfA_2}\}$ by balls $\sigma B(z)$
whose number of overlaps is uniformly bounded by a constant
depending only on $n$. On this set we have, by
Lemma~\ref{lemma:means close},
\begin{align*}
\sigma^{2n} \abs{\bfA_2}  \leq & \abs{\nabla \bfu - \bfA_2} \leq   \abs{\nabla \bfu - \bfV_{\sigma,z}} + \abs{\bfV_{\sigma,z}-\bfU_{2}}+\abs{\bfU_{2}- \bfA_2}
 \leq
 \abs{\nabla \bfu - \bfV_{\sigma,z}} + \frac{\sigma^{2n}}{4}\abs{\bfA_2}
\end{align*}
and hence $ \abs{\nabla \bfu - \bfA_2} \leq c \abs{\nabla \bfu -
\bfV_{\sigma,z}}$. On the other hand, by Lemma~\ref{lemma:means
close},
  $\abs{\bfA_2}\leq 2\abs{\bfV_{\sigma,z}}$.
Thus, owing to ~\eqref{eq:hammera} and~\eqref{eq:key2 non-deg p>2},
\begin{align*} &
\int_{\sigma B(z) \cap \{\abs{\nabla \bfu - \bfA_2} \geq \sigma^{2n}
\abs{\bfA_2} \}} \abs{\bfV(\nabla \bfu) - \bfV(\bfA_2)}^2 \dx
\\ & \qquad \leq  c
\int_{\sigma B(z) \cap \{\abs{\nabla \bfu - \bfA_2} \geq \sigma^{2n}
\abs{\bfA_2} \}} (\abs{\nabla \bfu} + \abs{\bfA_2})^{p-2}
\abs{\nabla \bfu - \bfA_2}^2 \dx
\\ & \qquad \leq  c
\int_{\sigma B(z) \cap \{\abs{\nabla \bfu - \bfA_2} \geq \sigma^{2n}
\abs{\bfA_2} \}} (\abs{\nabla \bfu} + \abs{\bfV_{\sigma,z}})^{p-2}
\abs{\nabla \bfu - \bfV_{\sigma,z}}^2 \dx
\\ & \qquad \leq  c \int_{\sigma B(z)} \abs{\bfV(\nabla\bfu) - \mean{\bfV(\nabla \bfu)}_{\sigma B(z)}}^2\dx \leq c
 |\sigma B(z)| \sigma^{2\alpha}
 E^{p'},
\end{align*}
where $E$ is defined as in   \eqref{eq:E def}. Hence, inequality
\eqref{eq:non-deg large energy} follows.
\end{proof}

We are now ready to prove the decay estimate in the case when
$p\geq2$.

\begin{proposition} \label{prop:non-deg p>2}
Let $p \in [2, \infty)$, and let $\bfu$ be a local weak solution to
\eqref{eq:sysA}.
Let $\alpha$ be the exponent appearing in the statement of Theorem
\ref{thm:decay}.
 Let $\theta \in (0,1)$. Then there exist
constants $c=c(n,p,N)$, $\varepsilon _\theta=
\varepsilon(n,p,N,\theta) \in (0, \tfrac 14)$, and $c_\theta =
c_\theta(n,p,N,\theta)$ such that if $\bfu$ satisfies
\eqref{eq:non-deg} with $\varepsilon = \varepsilon _\theta$, then
\begin{multline}
\label{eq:non-deg p>2 concluded} \bigg( \dashint_{\theta
B}\abs{\bfA(\nabla \bfu)-\mean{\bfA(\nabla \bfu)}_{\theta
B}}^{p'}\dx\bigg)^\frac1{p'}
\\  \leq c \theta^{\min\set{\frac {2\alpha}{p'},1}} \bigg( \dashint_{2B}\abs{\bfA(\nabla \bfu)-\mean{\bfA(\nabla \bfu)}_{{2B}}}^{p'}\dx \bigg)^{\frac1{p'}}
+c_\theta \bigg(\dashint_{2B} \abs{\bfF - \bfF _0}^{p'} \dx
\bigg)^\frac1{p'}
\end{multline}
for every $\bfF_0 \in \mathbb R^{N \times n}$.
\end{proposition}

\begin{proof}
From \eqref{meanq}, \eqref{eq:hammerd-bis} and \eqref{eq:hammera}
one can infer that
  \begin{align}\label{june32}
 \dashint_{\theta B} & \abs{\bfA(\nabla \bfu) - \bfA(\bfA_\theta)}^{p'}\dx
 \leq  c \dashint_{\theta B} \abs{\bfA(\nabla \bfu) - \bfA(\bfU_\theta)}^{p'}\dx
\\ \nonumber
& \leq c \dashint_{\theta B}\Big(\abs{\nabla
\bfu-\bfU_{\theta}}+\abs{\bfU_{\theta}}\Big)^{p'(p-2)}\abs{\nabla
\bfu-\bfU_{\theta}}^{p'}
\\ \nonumber &  \leq
c  \dashint_{\theta B}   \abs{\bfV(\nabla \bfu)  -
\bfV(\bfU_\theta)}^{2} \dx + c \abs{\bfU_\theta}^{p'(p-2)}
\dashint_{\theta B  }  \abs{\nabla \bfu- \bfU_\theta}^{p'} \dx.
\end{align}
By Lemma~ \ref{lemma:non-deg decay 1 p>2},
\begin{equation}\label{june33}
\dashint_{\theta B}   \abs{\bfV(\nabla \bfu)  -
\bfV(\bfU_\theta)}^{2} \dx  \leq c \theta^{2\alpha} E^{p'},
\end{equation}
where $E$ is defined by \eqref{eq:E def}.
 Next, via a repeated use of inequality \eqref{meanq}, the
triangle inequality,  Lemma \ref{lemma:means close}, and inequality
\eqref{eq:z decay}  one obtains the following chain:
\begin{align}\label{june34}
 & \abs{\bfU_\theta}^{p'(p-2)}  \dashint_{\theta B  }  \abs{\nabla \bfu- \bfU_\theta}^{p'} \dx
   \leq\,c\,
 \abs{\bfU_\theta}^{p'(p-2)} \dashint_{\theta B  }  \abs{\nabla \bfu- \mean{\nabla \bfz}_{\theta B} }^{p'} \dx
 \\ \nonumber &  \leq
c \theta^{p'} \abs{\bfA_2}^{p'(p-2)} \bigg(\dashint_{ B  }
\abs{\nabla \bfz- \mean{\nabla \bfz}_{B}} \dx\bigg)^{p'}  + c
\theta^{-n} \abs{\bfA_2}^{p'(p-2)} \dashint_{ B  }  \abs{\nabla
\bfz- \nabla \bfu}^{p'} \dx
 \\ \nonumber &  \leq
c \theta^{p'} \abs{\bfA_2}^{p'(p-2)} \bigg(\dashint_{ B  }
\abs{\nabla \bfu- \mean{\nabla \bfu}_{B}} \dx\bigg)^{p'}
  + c(\theta^{-n} + \theta^{p'})  \abs{\bfA_2}^{p'(p-2)}
\dashint_{B  }  \abs{\nabla \bfz- \nabla \bfu}^{p'} \dx
  \\ \nonumber &  = I_1 + I_2.
\end{align}
Owing to~\eqref{meanq} and \eqref{eq:hammerd-bis},
\begin{align}\label{june35}
I_1 \leq & c \theta^{p'} \bigg(\dashint_{  B  }  \abs{\bfA_{2}
}^{p-2}  \abs{\nabla \bfu- \bfA_{2} } \dx\bigg)^{p'} \leq c
\theta^{p'} \bigg(\dashint_{ B  } \abs{\bfA(\nabla \bfu)-
\bfA(\bfA_{2}) } \dx\bigg)^{p'} \leq c \theta^{p'} E^{p'}.
\end{align}
In order to estimate $I_2$, observe that, thanks  to Lemma
\ref{lemma:H},
\begin{align*}
 \int_{B} \abs{\bfH(\nabla \bfu,\bfA_{2}) }^{p'} \dx
\leq & c \delta^{-(p-2)p'}  \int_{B \cap \{\abs{\nabla \bfu -
\bfA_2} \geq \delta \abs{\bfA_2}\}}  \abs{\nabla \bfu - \bfA_2}^p
\dx
\\  & + c \delta^{p'} \int_{2B} \abs{\bfA(\nabla \bfu) -\bfA(\bfA_2)}^{p'} \dx
\end{align*}
for every  $\delta \in (0,\frac12)$. This inequality, applied with
 $\delta =  \sigma^{2\alpha/p}$, and Lemma \ref{lemma:non-deg large
 energy} ensure that
\begin{align} \label{eq:H bound 1}
 \dashint_{B} \abs{\bfH(\nabla \bfu,\bfA_{2}) }^{p'} \dx
\leq &  c \sigma^{\frac{2\alpha}{p-1}} E^{p'}   + c
\sigma^{\frac{2\alpha}{p-1}}\dashint_{ 2B}   \abs{\bfA(\nabla \bfu)-
\bfA(\bfA_{2}) }^{p'} \dx.
\end{align}
Here, we have also exploited the fact that, by \eqref{eq:hammera},
$|\bfP - \bfQ|^p \leq c |\bfV (\bfP) - \bfV (\bfQ)|^2$ for $\bfP,
\bfQ \in \mathbb R^{N \times n}$, since $p \geq 2$.
\\
By ~\eqref{eq:CZ res} and~\eqref{eq:H bound 1},
\begin{align}\label{june30}
\abs{\bfA_2}^{p'(p-2)} \dashint_{B}  \abs{\nabla \bfz- \nabla \bfu}^{p'} \dx
\leq & c \dashint_{B} \abs{\bfF-\bfF_0}^{p'} \dx + c \dashint_{B} \abs{\bfH(\nabla \bfu,\bfA_{2}) }^{p'} \dx.
\\ \nonumber \leq & c_\sigma \dashint_{B} \abs{\bfF-\bfF_0}^{p'} \dx + c \sigma^{\frac{2\alpha}{p-1}} \dashint_{ 2B  } \abs{\bfA(\nabla \bfu)- \bfA(\bfA_{2}) }^{p'} \dx.
\end{align}
Choosing
$
\sigma = \theta^{{2\alpha}{p}}
$
in \eqref{june30} yields
\begin{equation}\label{june31}
I_2 \leq  c \theta^{p'} \dashint_{ 2B  } \abs{\bfA(\nabla \bfu)-
\bfA(\bfA_{2}) }^{p'} \dx + c_\theta \dashint_{B}
\abs{\bfF-\bfF_0}^{p'} \dx\,,
\end{equation}
where $c = c(n,p,N)$ and $c_\theta =  c(n,p,N,\theta)$.
We next fix $\varepsilon _\theta = \varepsilon(n,p,N,\theta)$ in
such a way that~\eqref{eq:epsilon cond 1} is satisfied. Combining
inequalities \eqref{june32}--\eqref{june35} and \eqref{june31}
yields \eqref{eq:non-deg p>2 concluded}.
%
%
\end{proof}

\subsection{Proof of Proposition~\ref{prop:decay conclusion}, concluded}\label{sec:comparison-nondeg}

The  decay estimates established above enable us to accomplish the
proof of Proposition~\ref{prop:decay conclusion}.

\begin{proof}[Proof of Proposition~\ref{prop:decay conclusion}]
We apply Proposition~\ref{pro:2},  Lemma~\ref{lem:F big}, and either
Proposition \ref{lemma:non-deg decay 1 p<2} or Proposition
\ref{prop:non-deg p>2}, according to whether $p \in (1,2)$ or $p \in
[2, \infty)$, with $4B$ or $2B$ replaced by $B$ in the integrals  on
the right-hand sides of the relevant decay estimates. Moreover,  we
denote by $c_1$, $c_2$ and $c_3$ the constants multiplying the
integrals involving the expression $\bfA(\nabla \bfu)$ on the
right-hand sides of  inequalities \eqref{nov10}, \eqref{eq:key2
non-deg p<2}, and \eqref{eq:non-deg p>2 concluded}, respectively.
\\ Fix any $\kappa \in \big(0, \min\set{1, \tfrac {2\alpha}{p'}}\big)$ and any $\delta >0$. In the case when $p \in (1, 2)$, we choose $\theta $ so small that
\begin{align}
\label{delta1}
 c_1 \theta^{\kappa} + c_2 \,
\theta^{\frac{2\alpha}{p'}}
\leq \delta.
\end{align}
 If, instead, $p \in [2, \infty)$, we choose
  $\theta $ so small that
\begin{align}
\label{delta2}
 c_1 \theta^{\kappa} + c_3 \theta ^{\min \{1, \frac {2\alpha}{p'}\}}\leq \delta\,,
\end{align}
%
%
and let $\varepsilon _\theta = \varepsilon(n,p,N,\delta)$  be fixed
as in Proposition~\ref{prop:non-deg p>2}.
In particular, inequalities \eqref{delta1} and \eqref{delta2} hold
if
$\theta = \big(c_1+\max\set{c_2,c_3}\big)^{-\frac 1\kappa}\delta
^{\frac 1\kappa}$
for small $\delta$. Inequality \eqref{eq:decay conclusion} follows,
modulo an application of
 H\"older's inequality.
\end{proof}

\section{Proofs of the main results}

This section is devoted to the proofs of Theorems \ref{thm:main},
\ref{thm:main2} and \ref{thm:main3}.

\begin{proof}[Proof of Theorem \ref{thm:main}]
Owing to Proposition~\ref{prop:decay conclusion},  given any $\delta
\in (0,1)$, there exist constants $\theta \in (0, 1)$ and $c>0$,
both depending only on $n, N, p, \delta$, such that
\begin{align}
\label{eq:decay conclusion appl}
 \bigg(\dashint_{\theta B} & \abs{\bfA(\nabla \bfu)-\mean{\bfA(\nabla \bfu)}_{\theta
 B}}^{\min\set{2,p'}}\dx\bigg)^\frac{1}{\min{\set{2,p'}}}
 \\ \nonumber & \qquad \leq \delta\bigg( \dashint_{B}\abs{\bfA(\nabla \bfu)-\mean{\bfA(\nabla
\bfu)}_{{B}}}^{\min\set{2,p'}}\dx\bigg)^\frac1{{\min\set{2,p'}}}
 + c\bigg(\dashint_{B}\abs{\bfF -\bfF _0}^{p'}\dx\bigg)^{\frac 1{p'}}.
\end{align}
 for every ball $B \subset \Rn$.  Given any $x \in \Rn$, let $B$ be
any ball such that $x \in \theta B$. Owing to the definition of
sharp maximal function \eqref{Msharps}, and to the arbitrariness of
$\bfF _0$, we deduce from \eqref{eq:decay conclusion appl} that
\begin{align}\label{nov51}
M^{\sharp,{\min\set{p',2}}}(\bfA (\nabla \bfu))(x) \leq   c \delta
M^{\sharp, {\min\set{p',2}}}(\bfA (\nabla \bfu))(x) +  c M^{\sharp,
{p'}} (\bfF) (x).
\end{align}
Note that $M^{\sharp,{\min\set{p',2}}}(\bfA (\nabla \bfu))(x)<
\infty$ for a.e. $x \in \Rn$, since we are assuming that $\bfu \in
V^{1,p}(\Rn)$.
 With the choice $\delta = (2 c)^{-1}$,  inequality \eqref{nov51} thus implies that
$$
 M^{\sharp, {\min\set{p',2}}}\big(\bfA(\nabla \bfu) \big)(x)\leq c M^{\sharp, {p'}}(\bfF)(x) \quad \hbox{for  $x\in \mathbb
 R^n$} \quad \hbox{for a.e. $x \in \Rn$.}
$$
Hence \eqref{main1} follows, inasmuch as $M^{\sharp }(\bfA(\nabla
\bfu)) \leq M^{\sharp, {\min\set{p',2}}}(\bfA(\nabla \bfu))$.
\end{proof}

\begin{proof}[Proof of Theorem~\ref{thm:main2}]
Let $\beta \in \big(0,  \min\set{1,\frac {2\alpha}{p'}}\big)$ be
such that condition \eqref{eq:omega condition} is fulfilled by
$\omega$. Fix $\gamma$ and  $\kappa$ such that
\begin{equation}\label{august301}
\beta < \gamma < \kappa <  \min\set{1,\tfrac {2\alpha}{p'}}.
\end{equation}
Given $\delta$, let $\theta$ be the number obeying
\begin{equation}\label{august302}
\overline c \theta ^\kappa = \delta\,,
\end{equation}
for some constant $\overline c=\overline c(n,N,p, \kappa)>1$, and
such that inequality \eqref{eq:decay conclusion} holds. This is
possible owing to \eqref{august300}. Thus, there exists
$\theta_0\in(0,1)$ (depending on $\kappa$ and $\gamma$, as well as
on $n, N, p$) such that $\overline c \theta^{\kappa-\gamma}\leq1$
for $\theta\leq\theta_0$, and hence
\begin{align}\label{26082}
\delta\leq \theta^\gamma
\end{align}
for every $\delta\leq\delta_0$, where $\delta_0=\overline c \theta
_0^\kappa$. Now choose $\delta = \delta_0$ in
Proposition~\ref{prop:decay conclusion}. Thus, given  any ball $B_r
\subset \Omega$, an iteration of \eqref{eq:decay conclusion} tells
us that
\begin{multline}
\label{eq:decay conclusion appl 1}
 \bigg(\dashint_{B_{\theta^k r}}  \abs{\bfA(\nabla \bfu)-\mean{\bfA(\nabla \bfu)}_{
 B_{\theta^k r}}}^{\min\set{2,p'}}\dx\bigg)^\frac{1}{\min{\set{2,p'}}}
 \\  \leq \delta_0^k \bigg( \dashint_{B_r}\abs{\bfA(\nabla \bfu)-\mean{\bfA(\nabla
\bfu)}_{{B_r}}}^{\min\set{2,p'}}\dx\bigg)^\frac1{{\min\set{2,p'}}}
 + c \theta^{-\frac{nk}{p'}}\bigg(\dashint_{B_r}\abs{\bfF -\bfF _0}^{p'}\dx\bigg)^\frac1{p'}
\end{multline}
for $k \in \mathbb N$.
 On setting
$\rho=\theta^k$, and making use of  \eqref{26082},
 inequality \eqref{eq:decay conclusion appl 1} tells us that
\begin{multline}
\label{eq:decay conclusion appl 2}
 \bigg(\dashint_{ B_{\rho r}}   \abs{\bfA(\nabla \bfu)-\mean{\bfA(\nabla \bfu)}_{
B_{\rho r}}}^{\min\set{2,p'}}\dx\bigg)^\frac{1}{\min{\set{2,p'}}}
 \\   \leq \,\rho^{\gamma} \bigg( \dashint_{B_r}\abs{\bfA(\nabla \bfu)-\mean{\bfA(\nabla
\bfu)}_{{B_r}}}^{\min\set{2,p'}}\dx\bigg)^\frac1{{\min\set{2,p'}}}
 + c \rho^{-\frac{n}{p'}}\bigg(\dashint_{B_r}\abs{\bfF -\bfF _0}^{p'}\dx\bigg)^\frac1{p'}.
\end{multline}
Dividing through by $\omega(\rho r)$ in \eqref{eq:decay conclusion
appl 2} yields
\begin{align}\label{june40}
 \frac{1}{\omega(\rho r)}&\bigg(\dashint_{ B_{\rho r}} \abs{\bfA(\nabla \bfu)-\mean{\bfA(\nabla \bfu)}_{ B_{\rho r}}}^{\min\set{2,p'}}\dx\bigg)^\frac{1}{\min{\set{2,p'}}}
 \\ \nonumber  &  \leq \frac{\rho^{\gamma}\omega(r)}{\omega(\rho r)}  \frac{1}{\omega(r)} \bigg( \dashint_{B_r}\abs{\bfA(\nabla \bfu)-\mean{\bfA(\nabla
\bfu)}_{{B_r}}}^{\min\set{2,p'}}\dx\bigg)^\frac1{{\min\set{2,p'}}}
 \\ \nonumber & \qquad + c \frac{\omega(r) \rho^{-\frac{n}{p'}}}{\omega(\rho r)}   \frac{1}{\omega(r)}
 \bigg(\dashint_{B_r}\abs{\bfF -\bfF _0}^{p'}\dx\bigg)^\frac1{p'}.
\end{align}
Let $k= k (n,N,p,\omega_\beta, \gamma, \beta)$ be the smallest
integer such that $c_\omega \theta^{(\gamma -\beta )k } \leq \frac
12$. Hence,
\begin{align}\label{august4} \frac{\omega(r)}{\omega(\rho r)} \leq c_\omega \rho ^{-\beta}
\end{align}
and
\begin{align}\label{june41} \frac{\rho^{\gamma }\omega(r)}{\omega(\rho r)} \leq \frac 12.
\end{align}
%
As a consequence of \eqref{june40} and \eqref{august4}
 we obtain that
\begin{multline}\label{june42}
  \frac{1}{\omega(\rho r)}\bigg(\dashint_{ B_{\rho r}} \abs{\bfA(\nabla \bfu)-\mean{\bfA(\nabla \bfu)}_{
 B_{\rho r}}}^{\min\set{2,p'}}\dx\bigg)^\frac{1}{\min{\set{2,p'}}}
 \\  \leq  \frac{1}{2\omega(r)} \bigg( \dashint_{B_r}\abs{\bfA(\nabla \bfu)-\mean{\bfA(\nabla
\bfu)}_{{B_r}}}^{\min\set{2,p'}}\dx\bigg)^\frac1{{\min\set{2,p'}}}
 +   \frac{c}{\omega(r)} \bigg(\dashint_{B_r}\abs{\bfF -\bfF
 _0}^{p'}\dx\bigg)^\frac1{p'}.
\end{multline}
Now, let $B$ be as in the
statement, and let $x \in B$. Inequality \eqref{june42}, applied to
any ball $B_r$ such that $r < R$ and $B_{\rho r} \ni x$, tells us
that
\begin{align}\label{june43}
 \sup_{r < \rho R }\frac{1}{\omega(r)} & \bigg(\dashint_{B_r} \abs{\bfA(\nabla \bfu)-\mean{\bfA(\nabla \bfu)}_{B_r}}^{\min\set{2,p'}}\dx\bigg)^\frac{1}{\min{\set{2,p'}}}
 \\ \nonumber &  \leq   \frac 12 \sup_{r <  R }\frac{1}{\omega(r)}
 \bigg(\dashint_{B_r} \abs{\bfA(\nabla \bfu)-\mean{\bfA(\nabla \bfu)}_{B_r}}^{\min\set{2,p'}}\dx\bigg)^\frac{1}{\min{\set{2,p'}}}
 \\ \nonumber & \quad
+  \sup_{r <  R } \frac{c}{\omega(r)} \bigg(\dashint_{B_r}\abs{\bfF
-\bfF
 _0}^{p'}\dx\bigg)^\frac1{p'}.
\end{align}
%
%
%
%
%
%
%
On the other hand, owing to \eqref{meanq} and \eqref{eq:omega
condition},
\begin{align}\label{june44}
&  \sup_{\rho R<r <  R}\frac{1}{\omega(r)} \bigg(\dashint_{B_r}
 \abs{\bfA(\nabla \bfu)-\mean{\bfA(\nabla \bfu)}_{B_r}}^{\min\set{2,p'}}\dx\bigg)^\frac{1}{\min{\set{2,p'}}}
 \\ \nonumber & \quad  \leq c \frac{\omega(2R)}{\omega(r)} \rho^{-\frac{n}{ {\min\set{2,p'}} }}
 \frac{1}{\omega(2R)} \bigg( \dashint_{2B}\abs{\bfA(\nabla \bfu)-\mean{\bfA(\nabla
\bfu)}_{{2B}}}^{\min\set{2,p'}}\dx\bigg)^\frac1{{\min\set{2,p'}}}.
 \\ \nonumber & \quad \leq c' \rho^{-\beta-\frac{n}{ {\min\set{2,p'}} }}
 \frac{1}{\omega(2R)} \bigg( \dashint_{2B}\abs{\bfA(\nabla \bfu)-\mean{\bfA(\nabla
\bfu)}_{{2B}}}^{\min\set{2,p'}}\dx\bigg)^\frac1{{\min\set{2,p'}}}.
\end{align}
Coupling \eqref{june43} with
\eqref{june44} tells us that
\begin{align}\label{june45}
 \sup_{r <  R }\frac{1}{\omega(r)} & \bigg(\dashint_{B_r} \abs{\bfA(\nabla \bfu)-\mean{\bfA(\nabla \bfu)}_{B_r}}^{\min\set{2,p'}}\dx\bigg)^\frac{1}{\min{\set{2,p'}}}
 \\ \nonumber &  \leq \frac 12  \sup_{r < R }\frac{1}{\omega(r)}
 \bigg(\dashint_{B_r} \abs{\bfA(\nabla \bfu)-\mean{\bfA(\nabla \bfu)}_{B_r}}^{\min\set{2,p'}}\dx\bigg)^\frac{1}{\min{\set{2,p'}}}
 \\ \nonumber & \quad
+  \sup_{r <  R } \frac{c }{\omega(r)} \bigg(\dashint_{B_r}\abs{\bfF
-\bfF
 _0}^{p'}\dx\bigg)^\frac1{p'}
\\ \nonumber & \quad +
 \frac{c }{\omega(2R)} \bigg( \dashint_{2B}\abs{\bfA(\nabla \bfu)-\mean{\bfA(\nabla
\bfu)}_{{2B}}}^{\min\set{2,p'}}\dx\bigg)^\frac1{{\min\set{2,p'}}}.
 \end{align}
Inequality \eqref{E:main2} follows from \eqref{june45}, via the very
definition of localized  weighted sharp maximal operator
\eqref{localweight}.
\end{proof}
\begin{proof}[Proof of Theorem~\ref{thm:main3}]
Without loss of generality,  we may assume that
\begin{equation}\label{june47}
\int_0^R \bigg( \dashint_{B_\varrho(x) }
 \bigg(\frac{\abs{ \bfF-\mean{\bfF}_{B_\varrho(x)} } }{\rho} \bigg)^{p'} \dy \bigg)^\frac{1}{p'}  \, d\varrho < \infty.
 \end{equation}
Fix $\delta =\frac12$, and let  $\theta = \theta(n,N,p)$ be the
corresponding value provided by  Proposition~\ref{prop:decay
conclusion}.  Given any $k \in \setN$, one can show, via a telescope
sum argument, that
\begin{align} \label{eq:Pf1.6 001}
\Bigabs{\dashint_{ B_{\theta^{k}R}(x)}\bfA(\nabla \bfu)\dy -
\dashint_{B_R(x)}\bfA(\nabla \bfu)\dy }
\leq \theta^{-n} \sum_{i=0}^{k-1}\dashint_{  B_{\theta^i R}(x)}
\abs{\bfA(\nabla \bfu)-\mean{\bfA(\nabla \bfu)}_{B_{\theta^i
R}(x)}}\dy\,.
\end{align}
Inequality \eqref{eq:decay conclusion} implies that
\begin{align*}
\sum_{i=1}^{k}& \bigg(\dashint_{B_{\theta^i R}(x)}\abs{\bfA(\nabla
\bfu) -\mean{\bfA(\nabla \bfu)}_{B_{\theta^i R}(x) }}
^{\min\set{p',2}}\dy\bigg)^\frac1{\min\set{p',2}}
\\ \nonumber
& \leq \frac12 \sum_{i=0}^{k-1} \bigg(\dashint_{B_{\theta^i
R}(x)}\abs{\bfA(\nabla \bfu)-\mean{\bfA(\nabla \bfu)}_{B_{\theta^i
R}(x) }}^{\min\set{p',2}}\dy\bigg)^\frac1{\min\set{p',2}}
\\ \nonumber
&\quad + c
\sum_{i=0}^{k-1}\bigg(\dashint_{B_{\theta^i R}(x)}
\abs{\bfF-\mean{\bfF}_{B_{\theta^i
R}(x)}}^{p'}\dy\bigg)^\frac{1}{p'}.
\end{align*}
Hence,
\begin{align}\label{june46}
\sum_{i=0}^{k}  &\bigg(\dashint_{B_{\theta^i R}(x)}\abs{\bfA(\nabla
\bfu) -\mean{\bfA(\nabla \bfu)}_{B_{\theta^i R}(x)
}}^{\min\set{p',2}}\dy \bigg)^\frac1{\min\set{p',2}}
\\ \nonumber
&\leq c \bigg(\dashint_{B_R(x)}\abs{\bfA(\nabla
\bfu)-\mean{\bfA(\nabla \bfu)}_{ B_R(x)
}}^{\min\set{p',2}}\dy\bigg)^\frac1{\min\set{p',2}}
\\
&\quad + c \sum_{i=0}^{k-1}\bigg(\dashint_{B_{\theta^i R}(x)}
\abs{\bfF-\mean{\bfF}_{B_{\theta^i
R}(x)}}^{p'}\dy\bigg)^\frac{1}{p'}.
\end{align}
Note that
\[
\sum_{i=0}^{k}\bigg(\dashint_{B_{\theta^i R}(x)}
\abs{\bfF-\mean{\bfF}_{B_{\theta^i
R}(x)}}^{p'}\dy\bigg)^\frac{1}{p'} \leq c \int_0^R \bigg(
\dashint_{B_\varrho(x) }
 \bigg(\frac{\abs{ \bfF-\mean{\bfF}_{B_\varrho(x)} } }{\rho} \bigg)^{p'} \dy \bigg)^\frac{1}{p'}  \, d\varrho.
\]
Therefore, letting $k$ go  to $\infty$ in \eqref{june46} tells us
that
\begin{align}\label{eq:Pf1.6 002}
 \sum_{i=0}^{\infty} & \bigg(\dashint_{B_{\theta^i
R}(x)}\abs{\bfA(\nabla \bfu)-\mean{\bfA(\nabla \bfu)}_{B_{\theta^i
R}(x) }}^{\min\set{p',2}}\dy \bigg)^\frac1{\min\set{p',2}}
\\ \nonumber
& \leq c  \bigg(\dashint_{B_R(x)}\abs{\bfA(\nabla
\bfu)-\mean{\bfA(\nabla \bfu)}_{B_R(x)
}}^{\min\set{p',2}}\dy\bigg)^\frac1{\min\set{p',2}}
\\ \nonumber
&\quad + c \int_0^R \bigg( \dashint_{B_\varrho(x) }
 \bigg(\frac{\abs{ \bfF-\mean{\bfF}_{B_\varrho(x)} } }{\rho} \bigg)^{p'} \dy \bigg)^\frac{1}{p'}  \, d\varrho.
\end{align}
Owing to \eqref{june47}, the series in \eqref{eq:Pf1.6 002} is
convergent. Thus, in particular,
\begin{align} \label{eq:Pf1.6 003}
 \lim_{i\to \infty} \bigg(\dashint_{B_{\theta ^i R(x)}}\abs{\bfA(\nabla
\bfu)-\mean{\bfA(\nabla \bfu)}_{B_{\theta ^i R}(x)
}}^{\min\set{p',2}}\dy\bigg)^\frac1{\min\set{p',2}} =0.
\end{align}
Now fix any $0<s<r\leq \tfrac 1\theta R$. Then there exists $k\in
\mathbb N$ such that $\theta^{1-k}s< r \leq \theta^{-k}s$. Set $r_s=
\theta^{-k}s$ and observe that there exists $h\in\mathbb N$ such
that $\theta^hR< r_s\leq \theta^{h-1}R$. In particular, $\theta
^{h+k} R < s \leq \theta ^{h+k-1} R$ and $\theta ^{h+1}R < r \leq
\theta ^{h-1}R$. Therefore
\begin{multline}\label{august200}
\dashint_{B_r(x)}\abs{\bfA(\nabla \bfu)-\mean{\bfA(\nabla
\bfu)}_{B_s(x) }}\dy \leq \frac c{\theta
^n}\dashint_{B_{r_s}(x)}\abs{\bfA(\nabla \bfu)-\mean{\bfA(\nabla
\bfu)}_{B_s(x) }}\dy
\\  \leq
\frac c{\theta ^n}\dashint_{B_{r_s}(x)}\abs{\bfA(\nabla
\bfu)-\mean{\bfA(\nabla \bfu)}_{B_{r_s}(x) }}\dy + \frac c{\theta
^n} \abs{\mean{\bfA(\nabla \bfu)}_{B_{r_s}(x)}-\mean{\bfA(\nabla
\bfu)}_{B_{s}(x)}}
\\\leq
\frac c{\theta ^{2n}}\dashint_{B_{\theta
^{h-1}R}(x)}\abs{\bfA(\nabla \bfu)-\mean{\bfA(\nabla
\bfu)}_{B_{\theta ^{h-1}R}(x) }}\dy + \frac c{\theta ^n}
\abs{\mean{\bfA(\nabla \bfu)}_{B_{r_s}(x)}-\mean{\bfA(\nabla
\bfu)}_{B_{\theta ^k r_s}(x)}}.
\end{multline}
%
%
%
%
%
The last addend on the right-hand side of \eqref{august200} can be
estimated via
 \eqref{eq:Pf1.6 001} and~\eqref{eq:Pf1.6 002}, with $R$ replaced by $r_s$. As a consequence, we
 obtain that
\begin{align}\label{august201}
\dashint_{B_r(x)} & \abs{\bfA(\nabla \bfu)-\mean{\bfA(\nabla
\bfu)}_{B_s(x) }}\dy \leq c \int_0^{r_s} \bigg(
\dashint_{B_\varrho(x) }
 \bigg(\frac{\abs{ \bfF-\mean{\bfF}_{B_\varrho(x)} } }{\rho} \bigg)^{p'} \dy \bigg)^\frac{1}{p'}  \, d\varrho
 \\ \nonumber
&\quad +
 \bigg(\dashint_{B_{\theta^{h-1}R}(x)}\abs{\bfA(\nabla \bfu)-\mean{\bfA(\nabla \bfu)}_{B_{\theta^{h-1}R}(x) }}^{\min\set{p',2}}\dy\bigg)^\frac1{\min\set{p',2}}
\\ \nonumber & \quad
+ \bigg(\dashint_{B_{r_s}(x)}\abs{\bfA(\nabla
\bfu)-\mean{\bfA(\nabla \bfu)}_{B_{r_s}(x)
}}^{\min\set{p',2}}\dy\bigg)^\frac1{\min\set{p',2}}
\\ \nonumber &
 \leq c \int_0^{r_s} \bigg( \dashint_{B_\varrho(x) }
 \bigg(\frac{\abs{ \bfF-\mean{\bfF}_{B_\varrho(x)} } }{\rho} \bigg)^{p'} \dy \bigg)^\frac{1}{p'}  \, d\varrho .
\\ \nonumber
 &\quad  +
 \bigg(\dashint_{B_{\theta^{h-1}R}(x)}\abs{\bfA(\nabla \bfu)-\mean{\bfA(\nabla \bfu)}_{B_{\theta^{h-1}R}(x) }}^{\min\set{p',2}}\dy\bigg)^\frac1{\min\set{p',2}}
\\ \nonumber & \quad
+ \bigg(\frac c{\theta
^n}\dashint_{B_{\theta^{h-1}R}(x)}\abs{\bfA(\nabla
\bfu)-\mean{\bfA(\nabla \bfu)}_{B_{\theta^{h-1}R}(x)
}}^{\min\set{p',2}}\dy\bigg)^\frac1{\min\set{p',2}}.
\end{align}
Since $\theta ^{h-1}R \leq \frac r\theta$ and $r_s \leq \frac
r\theta$, equations \eqref{august201} and \eqref{eq:Pf1.6 003}
ensure that, given any $\epsilon>0$ there exists  $r_\epsilon>0$
such that, if $0<s<r<r_\epsilon$, then
\begin{equation}\label{june50}
  \Bigabs{\dashint_{B_s(x)}\bfA(\nabla \bfu)\dy-\dashint_{B_r(x)}\bfA(\nabla \bfu)\dy}\leq
     \dashint_{B_r(x)}\abs{\bfA(\nabla \bfu)-\mean{\bfA(\nabla \bfu)}_{B_s(x) }}\dy<\epsilon.
\end{equation}
This shows that the function  $r \mapsto
\dashint_{B_r(x)}\bfA(\nabla \bfu)\dy$ satisfies the Cauchy
property, whence there exists its limit as $r\to 0^+$, and is
finite. On denoting by $\bfA(\nabla \bfu (x))$ such limit, we infer
from \eqref{june50} that
%
%
%
%
%
\begin{equation}\label{august203}
\lim_{r\to 0}\dashint_{B_r(x)}\abs{\bfA(\nabla \bfu(y))-\bfA(\nabla \bfu(x))}dy=0,
\end{equation}
and hence $x$ is a Lebesgue point of $\bfA(\nabla \bfu)$.
  \\ On making use of  inequalities ~\eqref{eq:Pf1.6 001}
and~\eqref{eq:Pf1.6 002}, and of H\"older's inequality, and passing
to the limit as  $k\to\infty$, one can show that
\begin{multline} \label{eq:Pf1.6 004}
\Bigabs{\bfA(\nabla \bfu(x))  - \dashint_{B_R(x)}\bfA(\nabla
\bfu)\dy }
 \leq c \int_0^R \bigg( \dashint_{B_\varrho(x) }
 \bigg(\frac{\abs{ \bfF-\mean{\bfF}_{B_\varrho(x)} } }{\rho} \bigg)^{p'} \dy \bigg)^\frac{1}{p'}  \, d\varrho\\
 +
c  \bigg(\dashint_{B_R(x)}\abs{\bfA(\nabla \bfu)-\mean{\bfA(\nabla
\bfu)}_{B_R(x) }}^{\min\set{p',2}}\dy\bigg)^\frac1{\min\set{p',2}}
\end{multline}
for every $x \in \Omega$ such that
$B_R(x) \subset \Omega$ and \eqref{august203} holds. An application
of inequality \eqref{eq:Pf1.6 004} with $R$ replaced with $\tfrac
R2$, and of Corollary~\ref{cor:VPL1} with $\bfP=0$ and
$B=B_\frac{R}{2}(x)$, tell us that
\begin{align}\label{june48}
& \abs{\bfA(\nabla \bfu(x))}\\ \nonumber
 &\leq c\bigg(\dashint_{B_\frac{R}{2}(x)}\abs{\bfA(\nabla\bfu)}^{\min\set{p',2}}\dy\bigg)^\frac1{\min\set{p',2}}+c \int_0^R\bigg( \dashint_{B_\varrho(x) }
 \bigg(\frac{\abs{ \bfF-\mean{\bfF}_{B_\varrho(x)} } }{\rho} \bigg)^{p'} \dy \bigg)^\frac{1}{p'}  \,
 d\varrho
 \\ \nonumber
& \leq c\dashint_{B_R(x)}\abs{\bfA(\nabla\bfu)}\dy+c \int_0^R\bigg(
\dashint_{B_\varrho(x) }
 \bigg(\frac{\abs{ \bfF-\mean{\bfF}_{B_\varrho(x)} } }{\rho} \bigg)^{p'} \dy \bigg)^\frac{1}{p'}  \, d\varrho
\\ \nonumber & \quad
+c \bigg(\dashint_{B_R(x) }
 \abs{ \bfF-\mean{\bfF}_{B_R(x)} }^{p'} \dy \bigg)^\frac{1}{p'}
\end{align}
for a.e. $x \in \Omega$ such that $B_R(x) \subset \Omega$. Since
$\abs{\bfA(\nabla \bfu(x))} = |\nabla \bfu (x)|^{p-1}$, and the last
term in \eqref{june48} can be estimated (up to a multiplicative
constant) by the last but one, inequality \eqref{E:main3} follows.
\end{proof}

\section{Estimates in norms depending on the size of functions}\label{rearr}

Here, we are concerned with gradient estimates for solutions to
\eqref{eq:sysA} involving rearrangement-invariant norms. Loosely
speaking, a  rearrangement-invariant norm is a norm on the space of
measurable  functions which only depends on their \lq\lq size", or,
more  precisely, on the measure of their level sets. In particular,
a \lq\lq reduction principle" is established via Theorem
\ref{thm:main}, which turns the problem of bounds for this kind of
norms of the gradient in terms of norms of the same kind of the
right-hand side  into a couple of one-dimensional Hardy type
 inequalities. This general principle is then specialized to various
 customary classes of norms.
 \\ For ease of presentation, here we limit our discussion to local
 solutions to \eqref{eq:sysA} when $\Omega = \Rn$. Analogous results
 hold, however, in arbitrary open sets $\Omega$.
\par
Let $f : \Omega \to \mathbb R$ be a measurable function. We denote
by $f^* :[ 0, \infty)\to [0, \infty ]$ the decreasing rearrangement
of $f$, defined as
 $$ f^* (s) = \inf\{ t\geq 0 : |\{x \in \Omega :|f(x)|
>t\}|\leq s\} \qquad {\rm {for}}\, \, s \geq 0.$$
Moreover, we set $$f^{**}(s) = \frac 1s\int _0^sf^*(r)dr \quad
\hbox{for $s>0$.}$$
 In other words, $f^*$ is the (unique) non
increasing, right-continuous function in $[0, \infty )$
equimeasurable with $f$, and $f^{**}$ is a maximal function of
$f^*$. Observe that
\begin{equation}\label{ff}
f^*(s) \leq f^{**}(s) \quad \hbox{for $s >0$.}
\end{equation}

\par
We say that $\| \cdot \|_{X(0, \infty)}$ is a
rearrangement-invariant functional defined on the set of real-valued
measurable functions on $(0, \infty)$  if it takes values into $[0,
\infty ]$ and satisfies the following properties:
\par\noindent
(i) $\| \varphi \|_{X(0, \infty)}\leq \| \psi \|_{X(0, \infty)}$
\quad whenever $0 \leq \varphi \leq \psi$ a.e. in $(0, \infty)$,
\par\noindent (ii) $\| \varphi \|_{X(0, \infty)}= \| \psi \|_{X(0, \infty)}$ \quad
whenever $\varphi^* = \psi^*$.
\par
Let $m \in \mathbb N$. We say that $\| \cdot \|_{X(\Omega )}$ is a
rearrangement-invariant functional on the set of measurable
functions from $\Omega$ into $\mathbb R^m$ if there exists a
rearrangement-invariant functional $\| \cdot \|_{ X(0, \infty)}$ on
$(0, \infty)$ such that
\begin{equation}\label{ri}
\| \bff \|_{X(\Omega )} = \| |\bff|^* \|_{X(0, \infty)}
\end{equation}
for every such function $\bff$.
\par
A rearrangement-invariant functional $\|\cdot\|_{X(0,\infty)}$ is
called a rearrangement-invariant function norm, if, for every
$\varphi$, $\psi$ and $\{\varphi_j\}_{j\in\mathbb N}$, and every
$\lambda \geq 0$, the following additional properties hold:
\begin{itemize}
\item[(P1)]\qquad $\|\varphi\|_{X(0,\infty)}=0$ if and only if $\varphi=0$;
$\|\lambda \varphi\|_{X(0,\infty)}=
\lambda\|\varphi\|_{X(0,\infty)}$;
\par\noindent \qquad $\|\varphi+\psi\|_{X(0,\infty)}\leq \|\varphi\|_{X(0,\infty)}+
\|\psi\|_{X(0,\infty)}$;
\item[(P2)]\qquad $  \varphi_j \nearrow \varphi$ a.e.\ implies
$\|\varphi_j\|_{X(0,\infty)} \nearrow \|\varphi\|_{X(0,\infty)}$;
\item[(P3)]\qquad $\|\chi _E\|_{X(0,\infty)}<\infty$ if
$E\subset (0, \infty)$ with $|E|<\infty$;
\item[(P4)]\qquad
$\int_E \varphi(s)\,ds \le C \|\varphi\|_{X(0,\infty)}$ if $E\subset
(0, \infty)$ with $|E|<\infty$, for some constant $C$ independent of
$\varphi$.
\end{itemize}

If $\| \cdot \|_{X(\Omega )}$ is a rearrangement-invariant
functional built upon a rearrangement-invariant function norm
$\|\cdot\|_{X(0,\infty)}$, we denote
 by  $X(\Omega)$ the  collection of all  measurable functions   $\bff : \Omega \to \mathbb R^m$ such
that $\| \bff \|_{X(\Omega )} < \infty$. The functional $\| \cdot
\|_{X(\Omega )}$ defines
 a norm on $X(\Omega)$, and the latter  is a Banach space
endowed with this norm, which is called a~rearrangement-invariant
space. The function norm $\|\cdot\|_{X(0,\infty)}$ is called a
representation norm of  $\| \cdot \|_{X(\Omega )}$. In particular,
$X(0,\infty)$ is  a rearrangement-invariant space itself.
%
%
%
%
%
%
%
%

 \par
 Given   a measurable subset $G$ of $\Omega$,  we denote by $\chi _{G}$  the characteristic function of
 $G$, and    define  $$\| \bff  \|_{X(G)} = \| \bff \chi _{G}  \|_{X(\Omega)}$$   for any measurable function $\bff : \Omega \to \mathbb R^m$.
\par\noindent
Moreover, we denote by $X_{\rm loc}(\Omega )$ the space of
measurable functions $\bff$  such that $\| \bff  \|_{X(G)}< \infty$
for every compact set $G \subset \Omega$.
\\
Given a rearrangement-invariant functional $\|\cdot \|_{X(\Omega)}$
and any number $q \in (0, \infty)$, the functional
 $\|\bff \|_{X^q(\Omega)}$, defined as
 \begin{equation}\label{Xq}
\|\bff \|_{X^q(\Omega)} = \||\bff|^q \|_{X(\Omega)}^{\frac 1q}
\end{equation}
for any measurable function $\bff : \Omega \to \mathbb R^m$, is also
a rearrangement-invariant functional. Moreover, if  $\|\cdot
\|_{X(\Omega)}$ is a rearrangement-invariant norm, and $q \geq 1$,
then $\|\cdot \|_{X^q(\Omega)}$ is a rearrangement-invariant norm as
well~\cite{MP}.

\begin{theorem}\label{reduction}{\bf [Reduction to one-dimensional
inequalities]} Let $\|\cdot \|_{X(\mathbb R^n)}$ and $\|\cdot
\|_{Y(\mathbb R^n)}$ be rearrangement-invariant functionals. Assume
that there exists a constant $c$ such that
\begin{equation}\label{reduction0}
\bigg\|\frac 1s\int _0^s \varphi (r)dr\bigg\|_{X^{\frac 1{p'}} (0,
\infty)} \leq c \|c\varphi \|_{ X^{\frac 1{p'}}(0, \infty)}
\end{equation}
and
\begin{equation}\label{reduction1}
\bigg\|\int _s^\infty \varphi (r)\frac{dr}{r}\bigg\|_{ Y (0,
\infty)} \leq c \|c\varphi \|_{ X(0, \infty)}
\end{equation}
for every nonnegative function $\varphi \in
 X(0, \infty)$. Then
there exists a constant $c'=c'(n, N, p, c)$ such that
\begin{equation}\label{reduction2}
\||\nabla \bfu|^{p-1}\|_{Y (\mathbb R^n)} \leq c' \|c'
\bfF\|_{X(\mathbb R^n)}
\end{equation}
for every  local weak solution $\bfu\in V^{1,p}(\setR^n)$ to  system
\eqref{eq:sysA} with $\Omega = \Rn$.
\end{theorem}

\bigskip
\par\noindent
\begin{remark}\label{remjuly}{\rm
Let us briefly comment on assumptions \eqref{reduction0} and
\eqref{reduction1} in Theorem \ref{reduction}. The first one amounts
to requiring that the functional $\|\cdot \|_{X(\mathbb R^n)}$ is
stronger, in a qualified sense, than  $\|\cdot \|_{L^{p'}(\mathbb
R^n)}$. As recalled in Section \ref{intro}, this is a borderline
norm for estimates for the gradient to  weak solutions to system
\eqref{eq:sysA}. Assumption \eqref{reduction1} concerns, instead,
the opposite endpoint in the scale of admissible gradient estimates,
which fail if the norm $\|\cdot \|_{X(\mathbb R^n)}$ is too strong,
namely if the latter is \lq\lq too close" to $\|\cdot
\|_{L^\infty(\mathbb R^n)}$. In particular, if this is the case,
\eqref{reduction1} tells us  how much the $\|\cdot \|_{Y(\mathbb
R^n)}$ norm of $|\nabla \bfu|^{p-1}$ has to be weaker than the
$\|\cdot \|_{X(\mathbb R^n)}$ norm of $|\bfF|$, for inequality
\eqref{reduction2} to hold. }
\end{remark}

\medskip
\par\noindent

The proof of Theorem \ref{reduction} requires the following
proposition.

\medskip
\par\noindent

\begin{proposition}\label{hardy}
Let $\|\cdot \|_{X(\mathbb R^n)}$ and $\|\cdot \|_{Y(\mathbb R^n)}$
be rearrangement-invariant functionals. Assume that inequality
\eqref{reduction1} holds. Then
\begin{equation}\label{hardy1}
\| \psi ^*\|_{ Y (0, \infty)} \leq c \|c (\psi^{**} - \psi^*) \|_{
X(0, \infty)}
\end{equation}
for every nonnegative measurable function $\psi$ in
 $(0, \infty)$ such that $\psi < \infty$ a.e. and $\lim _{s \to \infty}\psi ^*(s) =0$. Here,
 $c$ denotes the constant appearing in \eqref{reduction1}.
 \end{proposition}
\par\noindent
\begin{proof} Given any function $\psi$ as in the statement, there exists a
nonnegative Radon measure $\nu$ on $(0, \infty)$ such that
\begin{equation}\label{hardy2}
\psi ^* (s) = \int _s^\infty d\nu (r) \quad \hbox{for $s > 0$.}
\end{equation}
Thus, by Fubini's Theorem,
\begin{equation}\label{hardy3}
\psi^{**}(s) - \psi^*(s) = \frac 1s \int_0^s r d \nu (r) \quad
\hbox{for $s > 0$.}
\end{equation}
By \eqref{hardy2} and \eqref{hardy3}, inequality \eqref{hardy1} will
follow if we show that
\begin{equation}\label{hardy4}
\bigg\|  \int _s^\infty d\nu (r)\bigg\|_{ Y (0, \infty)} \leq c
\bigg\| \frac cs \int_0^s r d \nu (r) \bigg\|_{ X(0, \infty)}
\end{equation}
for every nonnegative Radon measure $\nu$ in $(0, \infty)$ such that
$ \int _s^\infty d\nu (r) < \infty$ for $s >0$, or, equivalently,
that
\begin{equation}\label{hardy5}
\bigg\|  \int _s^\infty \frac{d\mu (r)}{r}\bigg\|_{ Y (0, \infty)}
\leq c \bigg\| \frac cs \int_0^s  d \mu (r) \bigg\|_{ X(0, \infty)}
\end{equation}
for every nonnegative Radon measure $\mu$ on $(0, \infty)$ such that
\begin{equation}\label{hardy6}
\int _s^\infty \frac{d\mu (r)}{r} <\infty \quad \hbox{for $s >0$.}
\end{equation}
In order to prove \eqref{hardy5}, set
$$\mu (0, s) = \int _0^sd\mu(r) \quad \hbox{for $s >0$,}$$
and observe that integration by parts yields
\begin{equation}\label{hardy7}
\int _s^t \frac{d\mu (r)}{r}= \frac{\mu (0,t)}{t} - \frac{\mu
(0,s)}{s} + \int _s^t \frac{\mu (0, r)}{r^2}dr \quad \hbox{if
$0<s<t$.}
\end{equation}
 By \eqref{hardy6} and \eqref{hardy7}, $\lim _{t\to \infty} \frac{\mu
(0,t)}{t}$ exists,  is finite, and
\begin{equation}\label{hardy8}
\int _s^\infty \frac{\mu (0, r)}{r^2}dr < \infty.
\end{equation}
From equations \eqref{hardy6}--\eqref{hardy8} one can deduce that
$\lim _{t\to \infty} \frac{\mu (0,t)}{t}=0$.
 Thus, passing to the limit as $t \to \infty$ in  \eqref{hardy7}
 tells us that
\begin{equation}\label{hardy9}
\int _s^\infty \frac{d\mu (r)}{r} \leq \int _s^\infty \frac{\mu (0,
r)}{r^2}dr \quad \hbox{for $s >0$.}
\end{equation}
Hence, inequality \eqref{hardy5} is a consequence of the fact  that
\begin{equation}\label{hardy10}
\bigg\|  \int _s^\infty \frac{g(r)}{r^2}dr\bigg\|_{ Y (0, \infty)}
\leq c \bigg\|c \frac {g(s)}s \bigg\|_{ X(0, \infty)}
\end{equation}
for every nonnegative measurable function $g$ in $(0, \infty)$,
which, in turn, follows from  \eqref{reduction1}.
\end{proof}

Upper and lower estimates, in rearrangement form, for sharp maximal
function operators, also play a key role  in the proof of Theorem
\ref{reduction}.
 Given $q
\geq 1$ and any function $\bff : \Rn \to \mathbb R^m$ such that
$\bff \in L^q _{\rm loc} (\Rn)$, the classical maximal function $M^q
\bff : \Rn \to [0, \infty]$ is defined as
$$M^q \bff(x) =  \sup _{B \ni x} \bigg(\dashint _B |\bff|^q \,
dy\bigg)^{\frac 1q} \quad \hbox{for $x \in \Rn$.}$$ Clearly,
\begin{equation}\label{janA}
M^{\sharp , q} \bff(x) \leq 2 M^q \bff (x) \quad \hbox{for $x \in
\Rn$}
\end{equation}
for any such function $\bff$. Furthermore, Riesz' inequality
\cite[Theorem 3.8, Chapter 3]{BS} tells us that there exists a
constant $C=C(n)$ such that
\begin{equation}\label{janB}
(M^q \bff)^* (s) \leq C (|\bff|^q)^{**}(s)^{\frac 1q} \quad
\hbox{for $s > 0$,}
\end{equation}
for $\bff \in L^q _{\rm loc} (\Rn)$. Coupling inequalities
\eqref{janA} and \eqref{janB} tells us that
\begin{equation}\label{janC}
(M^{\sharp , q} \bff)^* (s) \leq 2C  (|\bff|^q)^{**}(s)^{\frac 1q}
\quad \hbox{for $s > 0$,}
\end{equation}
for $\bff \in L^q _{\rm loc} (\Rn )$.
\\ On the other hand, as a consequence of \cite[Theorem 7.3, Chapter
5]{BS}, one can show that
\begin{equation}\label{rearrsharprn}
f^{**}(s)-f^*(s) \leq C (M^\sharp f)^*(s) \quad \hbox{for $s
>0$,}
\end{equation}
for every locally integrable function $f : \Rn \to \mathbb R$.

\smallskip
\par
We are now ready to prove Theorem \ref{reduction}.

\smallskip
\par\noindent
\begin{proof}[Proof of Theorem \ref{reduction}]
 By inequality \eqref{janC},
\begin{equation}\label{red5}
(M^{\sharp , p'}(\bfF))^*(s) \leq C
\Big[\big(|\bfF|^{p'})^{**}(s)\Big]^{\frac 1{p'}} \quad \hbox{for $s
> 0$}
\end{equation}
for some constant $C=C(n)$.
 Next, set $\bfu = (u^1, \cdots , u^N)$. Owing to inequality \eqref{rearrsharprn},
applied with $f = |\nabla \bfu|^{p-2}u_{x_i}^j$, for $i=1, \dots ,
n$, $j=1, \dots , N$, one has that
\begin{align}\label{red9}
\big(|\nabla \bfu|^{p-2}u_{x_i}^j\big)^{**}(s) - \big(|\nabla
\bfu|^{p-2}u_{x_i}^j\big)^{*}(s) & \leq C
\big(M^\sharp (|\nabla \bfu|^{p-2}u_{x_i}^j)\big)^*(s) \\
\nonumber & \leq C \big(M^\sharp(|\nabla \bfu|^{p-2}\nabla
\bfu)\big)^*(s)
 \quad \hbox{for $s > 0$.}
\end{align}
Inequality \eqref{main1} implies that
\begin{equation}\label{janF}
M^\sharp(|\nabla \bfu|^{p-2}\nabla \bfu)^*(s) \leq C M^{\sharp ,
p'}(\bfF)^*(s)  \quad \hbox{for $s
> 0$.}
\end{equation}
Combining inequalities \eqref{red5}--\eqref{janF} yields
\begin{equation}\label{janG}
\big(|\nabla \bfu|^{p-2}u_{x_i}^j\big)^{**}(s) - \big(|\nabla
\bfu|^{p-2}u_{x_i}^j\big)^{*}(s) \leq C
\Big[\big(|\bfF|^{p'})^{**}(s)\Big]^{\frac 1{p'}} \quad \hbox{for $s
> 0$}
\end{equation}
for some constant $C=C(n, N, p)$. Hence, if $\|\cdot \|_{X(0,
\infty)}$ is any  rearrangement-invariant functional,
\begin{equation}\label{janH}
\|\big(|\nabla \bfu|^{p-2}u_{x_i}^j\big)^{**}  - \big(|\nabla
\bfu|^{p-2}u_{x_i}^j\big)^{*}\|_{X(0, \infty)} \leq C
\left\|\Big[\big(|\bfF|^{p'})^{**}\Big]^{\frac 1{p'}}\right\| _{X(0,
\infty)}.
\end{equation}
Inequality \eqref{reduction1} implies, via Proposition \ref{hardy},
that
\begin{equation}\label{red11}
\||\nabla \bfu|^{p-2}u_{x_i}^j\|_{Y(\Rn)} = \|\big(|\nabla
\bfu|^{p-2}u_{x_i}^j\big)^{*}\|_{Y(0, \infty)} \leq C
\|C(\big(|\nabla \bfu|^{p-2}u_{x_i}^j\big)^{**}  - \big(|\nabla
\bfu|^{p-2}u_{x_i}^j\big)^{*})\|_{X(0, \infty)}
\end{equation}
for some constant $C$. Moreover, owing to inequality
\eqref{reduction0},
\begin{align}\label{red6}
\left\|C\Big[\big(|\bfF|^{p'})^{**}\Big]^{\frac 1{p'}}\right\|_{X(0,
\infty)} \leq c\big\|c^{\frac 1{p'}}C|\bfF|^{*}\big\|_{X(0, \infty)}
= c\|cC\bfF\|_{X(\mathbb R^n)}.
\end{align}
Inequality \eqref{reduction2} follows from
\eqref{janH}--\eqref{red6}.
\end{proof}

\smallskip
\par
As a first application of Theorem \ref{reduction}, one can easily
recover the by now standard gradient estimates in Lebesgue spaces.
The fact that the relevant Lebesgue spaces satisfy assumptions
\eqref{reduction0} and \eqref{reduction1} is a consequence of
classical Hardy inequalities.

\begin{proposition}\label{lebesgue} {\bf [Lebesgue
spaces]} Let $q \in (p' , \infty)$. Then
\begin{equation}\label{lebesgue1}
\|\nabla \bfu\|_{L^{q(p-1)}(\mathbb R^n)} \leq C'
\|\bfF\|_{L^{q}(\mathbb R^n)}^{\frac 1{p-1}}
\end{equation}
for every local weak solution  $\bfu\in V^{1,p}(\setR^n)$ to  system
\eqref{eq:sysA} with $\Omega = \Rn$.
\end{proposition}

The Lorentz spaces provide a refinement
 of the Lebesgue
spaces. If either $q \in (1, \infty)$ and $r \in [1, \infty]$, or
$q=r=1$, or $q=r=\infty$, the rearrangement-invariant  functional
$\|\cdot\|_{L\sp{q,r}(0,\infty)}$, defined as
\begin{equation}\label{lorentznorm}
\|\varphi \|_{L\sp{q,r}(0,\infty)}=
\left\|s\sp{\frac{1}{q}-\frac{1}{r}}\varphi ^*(s)\right\|_{L\sp
r(0,\infty)}
\end{equation}
is a
rearrangement-invariant norm for any measurable function $\varphi$ in $(0, \infty)$. The corresponding
rearrangement-invariant space $L\sp{q,r}(\Omega)$ is called a
Lorentz space. The Lebesgue spaces are special instances of Lorentz
spaces, inasmuch as
$$L^{q,q}(\Omega) = L^q(\Omega)$$
for every $q \in [1, \infty]$.

\begin{proposition}\label{lorentz}  {\bf [Lorentz
spaces]} Let $q \in (p' , \infty)$ and $r\in [1, \infty]$. Then
\begin{equation}\label{lorenz1}
\|\nabla \bfu\|_{L^{q(p-1), r(p-1)}(\mathbb R^n)} \leq C'
\|\bfF\|_{L^{q,r}(\mathbb R^n)}^{\frac 1{p-1}}
\end{equation}
for every local weak solution $\bfu\in V^{1,p}(\setR^n)$ to system
\eqref{eq:sysA} with $\Omega = \Rn$.
\end{proposition}
\par\noindent
\begin{proof} Let $X(\Rn)=Y(\Rn)= L^{q,r}(\Rn)$ in Theorem
\ref{reduction}. Then, $X^{\frac 1{p'}}(\Rn)= L^{\frac q{p'}, \frac
r{p'}}(\Rn)$, and condition \eqref{reduction0} reads
\begin{equation}\label{july100}
\bigg(\int _0^\infty \bigg[\bigg(\frac 1{(\cdot)}\int _0^{(\cdot
)}\varphi (\tau )d\tau\bigg)^* (s)\bigg]^{\frac r{p'}}s^{\frac rq -
1}\, ds\bigg)^{\frac {p'}r} \leq C \bigg(\int _0^\infty \varphi
^*(s)s^{\frac rq - 1}\, ds\bigg)^{\frac {p'}r}
\end{equation}
for every measurable function $\varphi : (0, \infty) \to [0,
\infty)$. The Hardy-Littlewood inequality for rearrangements
\cite[Theorem 2.2, Chapter 2]{BS} ensures that $ \frac 1s\int
_0^{s}\varphi (\tau )d\tau \leq \frac 1s\int _0^{s}\varphi ^*(\tau
)d\tau = \varphi ^{**}(s)$ for $s
>0$. Hence,
$$\bigg(\frac 1{(\cdot)}\int _0^{(\cdot
)}\varphi (\tau )d\tau\bigg)^* (s) \leq \varphi ^{**}(s) \hbox{ for
$s>0$.}$$
 Thus,
inequality \eqref{july100} follows from the inequality
\begin{equation}\label{july101b}
\bigg(\int _0^\infty \varphi ^{**}(s)^{\frac r{p'}}s^{\frac rq -
1}\, ds\bigg)^{\frac {p'}r} \leq C \bigg(\int _0^\infty \varphi
^*(s)s^{\frac rq - 1}\, ds\bigg)^{\frac {p'}r}
\end{equation}
for every measurable function $\varphi : (0, \infty) \to [0,
\infty)$, which holds, for instance, as a special case of
\cite[Theorem 4.1]{CPSS}.
\\ As for condition \eqref{reduction1}, we have that
\begin{align*}
\sup_{\varphi \geq 0}& \frac{\|\int _s^\infty \varphi
(\tau)\frac{d\tau}{\tau}\|_{L^{q,r}(0,
\infty)}}{\|\varphi\|_{L^{q,r}(0, \infty)}}  \approx \sup _{\varphi
\geq 0} \sup_{\psi \geq 0} \frac{\int _0^\infty \psi (s) \int
_s^\infty \varphi (\tau)\frac{d\tau}{\tau}\,
ds}{\|\psi\|_{L^{q',r'}(0, \infty)}\|\varphi\|_{L^{q,r}(0, \infty)}}
\\ & =
 \sup_{\psi \geq 0} \sup _{\varphi \geq 0} \frac{\int _0^\infty \frac{\varphi (\tau)}{\tau} \int _0^\tau \psi
(s)\, ds\, d\tau}{\|\psi\|_{L^{q',r'}(0,
\infty)}\|\varphi\|_{L^{q,r}(0, \infty)}} \approx \sup_{\psi \geq 0}
\frac{\| \frac{1}{\tau} \int _0^\tau \psi (s)\, ds\|_{L^{q',r'}(0,
\infty)}}{\|\psi\|_{L^{q',r'}(0, \infty)}} < \infty\,,
\end{align*}
where the equivalences hold due to H\"older type inequalities in
Lorentz spaces, the equality to Fubini's theorem, and the last
inequality to weighted Hardy type inequalities \cite[Theorem
1.3.2/1]{Maz11}. \\ The
conclusion now follows from Theorem \ref{reduction}.
\end{proof}

%

\medskip
\par
Let now focus on gradient estimates in Orlicz spaces, an extension
of the Lebesgue spaces in a different direction. The notion of
Orlicz space involves that of Young functions.  A Young function is
a (non-trivial) convex function $\Phi : [0, \infty ) \to [0,
\infty]$,
 vanishing at $0$. Its Young conjugate $\Phi
^\thicksim$ is again a Young function, and is defined as
$$\Phi ^\thicksim (t) = \sup\{ ts - \Phi (s): s \geq 0\} \quad
\hbox{for $t \geq 0$.}$$
 The Orlicz space $L^{\Phi}(\Omega)$ is the rearrangement-invariant
 space associated with the Luxemburg rearrangement-invariant norm
 $\|\cdot\|_{L^\Phi (\Omega)}$ given
 by
\begin{equation}\label{lux}
\|\varphi\|_{L^\Phi(0,\infty)}= \inf \left\{ \lambda >0 :  \int_0\sp
\infty \Phi \left( \frac{|\varphi(s)|}{\lambda} \right) ds \leq 1
\right\}
\end{equation}
for every measurable function $\varphi$ in $(0, \infty)$.

\begin{proposition}\label{orlicz} {\bf [Orlicz
spaces]} Let $\Phi$ be a Young function such that
\begin{equation}\label{orlicz1}
\essinf_{t\in (0, \infty)}\frac
{t\Phi ' (t)}{\Phi (t)} >p',
\end{equation}
and
\begin{equation}\label{orlicz0}
\int _0  \frac{{\Phi}^{\thicksim} (r) }{r^2}dr < \infty.
\end{equation}
Let $\Psi$ be the Young function defined by
\begin{equation}\label{orlicz2}
\Psi (t^{\frac 1{p-1}}) = \bigg(t\int _0^t \frac{{\Phi}^{\thicksim}
(r) }{r^2}dr \bigg)^{\thicksim} \quad \hbox{for $t \geq 0$.}
\end{equation}
Then there exists a constant $C=C(n,N,p, \Phi)$ such that
\begin{equation}\label{orlicz3}
\|\nabla \bfu\|_{L^{\Psi }(\mathbb R^n)} \leq C
\|\bfF\|_{L^{\Phi}(\mathbb R^n)}^{\frac 1{p-1}},
%
\end{equation}
and
\begin{equation}\label{orlicz4}
\int _{\mathbb R^n}\Psi\big(|\nabla \bfu|\big) \, dx \leq \int
_{\mathbb R^n} \Phi (C^{p-1}|\bfF|) \, dx.
\end{equation}
for every local weak  solution $\bfu\in V^{1,p}(\setR^n)$ to system
\eqref{eq:sysA}, with $\Omega = \Rn$.
\end{proposition}
\par\noindent
\begin{proof} Assumption \eqref{orlicz1} ensures that there exists
$\varepsilon
>0$ such that the function $t \mapsto \frac{\Phi (t^{\frac
1{p'}})}{t^{1+\varepsilon}}$ is increasing.
Thus,
\begin{equation}\label{orlicz5}
\int _0^t \frac{\Phi (s^{\frac 1{p'}})}{s^2}\, ds = \int _0^t
\frac{\Phi (s^{\frac 1{p'}})}{s^{1+\varepsilon}}t^{\varepsilon -1}\,
ds \leq \frac{\Phi (t^{\frac 1{p'}})}{t^{1+\varepsilon}}\frac
{t^\varepsilon}{\varepsilon} = \frac{\Phi (t^{\frac
1{p'}})}{\varepsilon \,t}.
\end{equation}
Hence, by \cite{Cianchistrong, Kita, GIM}, there exists a constant
$C$, depending only on the left-hand side of \eqref{orlicz1}, such
that
\begin{equation}\label{orlicz7}
\bigg\|\bigg(\frac 1s\int _0^s \varphi (r)\, dr\bigg)^{\frac 1{p'}}
\bigg\|_{L^\Phi(0, \infty )} \leq C \big\|\varphi (s)^{\frac
1{p'}}\|_{L^\Phi(0, \infty )}
\end{equation}
for every measurable function $\varphi : (0, \infty) \to [0,
\infty)$.
\par\noindent
On the other hand, under assumption \eqref{orlicz0}, \cite[Lemma 1,
Part (ii)]{Cianchistrong} ensures that there exists an absolute
constant $C$ such that
\begin{equation}\label{orlicz9}
\bigg\|\int _s^\infty \varphi (r)\, \frac{dr}{r}
\bigg\|_{L^\Theta(0, \infty )} \leq C \big\|\varphi (s)\|_{L^\Phi(0,
\infty )}
\end{equation}
for every measurable function $\varphi: (0, \infty) \to [0, \infty)$,
where $\Theta$ is the Young function given by
$$\Theta (t) = \bigg(t\int _0^t \frac{{\Phi}^{\thicksim}
(r) }{r^2}dr \bigg)^{\thicksim} \quad \hbox{for $t \geq 0$.}$$
\par\noindent
Owing to \eqref{orlicz7} and \eqref{orlicz9}, Theorem
\ref{reduction} implies that
\begin{equation}\label{july101}
\||\nabla \bfu|^{p-1}\|_{L^\Theta (\Rn)} \leq C \|\bfF \|_{L^\Phi
(\Rn)}.
\end{equation}
Hence, inequality \eqref{orlicz3} follows.
\\ Inequality \eqref{orlicz4} can be derived on applying
\eqref{july101} with $\Phi$ replaced
 with $\tfrac
 {\Phi }{\int _{\Rn} \Phi (C|\bfF|)\,
 dx}$. Such a derivation makes use of the definition of the Luxemburg
 norm, and of the fact that the constant in \eqref{july101} depends
 on $\Phi$ only through the infimum on the left-hand side of \eqref{orlicz1}, and
 such infimum is invariant under replacements of $\Phi$ with $k\Phi$
 for every positive constant $k$.
 \end{proof}

\bigskip
\par\noindent

The content of the following result  is a special case of
Proposition \ref{orlicz}. In the statement, ${\rm exp}_qL^\gamma
(\Rn)$ denotes the Orlicz space associated with a Young function
which is equivalent to $e^{t^\gamma}$ for large $t$, and   to $t^q$
for small $t$.

%
%
%
%
%

\bigskip
\par\noindent

\begin{proposition}\label{exp}  {\bf [Exponential
spaces]} Let $\gamma >0$, and $q>p'$. Then
\begin{equation}\label{exp1}
\|\nabla \bfu\|_{{\rm exp}_{q(p-1)}L^{\frac{\gamma (p-1)}{\gamma
+1}}(\mathbb R^n)} \leq C \|\bfF\|_{{\rm exp}_qL^{\gamma}(\mathbb
R^n)}^{\frac 1{p-1}}
\end{equation}
for every  local weak solution $\bfu\in V^{1,p}(\setR^n)$ to system
\eqref{eq:sysA} with $\Omega= \Rn$.
\end{proposition}

The last result of this section deals with a borderline case of
Theorem \ref{reduction}. In the statement, $L_q^\infty (\Rn)$
denotes the Orlicz space built upon a Young function which equals
$\infty$ for large $t$, and is equivalent to $t^q$ for small $t$.

\begin{proposition}\label{linf}  {\bf [Borderline exponential
spaces]} Let $q>p'$. Then
\begin{equation}\label{linf1}
\|\nabla \bfu\|_{{\rm exp}_{q(p-1)}L^{p-1}(\mathbb R^n)} \leq C
\|\bfF\|_{L_q^{\infty}(\mathbb R^n)}^{\frac 1{p-1}}
\end{equation}
 for every  local weak solution $\bfu\in V^{1,p}(\setR^n)$ to
system \eqref{eq:sysA} with $\Omega = \Rn$.
\end{proposition}

\section{Estimates in norms depending on oscillations of functions}\label{oscill}

Let $\omega : [0, \infty) \to [0,
\infty)$ be a function,  let $q \geq 1$, and let $m \in \mathbb N$.
The Campanato type space $\mathcal L ^{\omega , q}(\Omega)$  on an
open set $\Omega \subset \Rn$, is the space of those $\mathbb
R^m$-valued functions $\bff \in L^q_{\rm loc}(\Omega )$ for which
the semi-norm
\begin{align}\label{campnorm}
  \norm{\bff}_{\mathcal L ^{\omega , q}(\Omega)}= \sup _{B_r \subset \Omega}
  \frac 1{\omega (r)} \bigg(\dashint_{B_r} |\bff - \mean{\bff}_{B_r}|^q\,
  dx\bigg)^{\frac 1q}
%
\end{align}
is finite.  The space $\mathcal L ^{\omega , q} _{\rm loc}(\Omega)$
is defined accordingly, as the set of all functions $\bff$ such that
$ \norm{f}_{\mathcal L ^{\omega , q}(\Omega ')} < \infty$ for every
open set $\Omega ' \subset \subset \Omega$. When $q=1$, we shall
simply denote  $\mathcal L ^{\omega , q}  (\Omega)$ by  $\mathcal L
^{\omega } (\Omega)$, and similarly
for local spaces. \\
In the special case when $\omega (r) =1$, one
has that
\begin{align}\label{bmonorm}\mathcal L ^{1, q}(\Omega) = \setBMO(\Omega)
\end{align}
for every $q \geq 1$~\cite{JohNir61}. Another customary instance
corresponds to the choice $\omega (r) = r^\beta$, for some $\beta
\in (0, 1]$. Indeed, Campanato's representation theorem tells us
that, if $\Omega$ is regular enough, say a bounded Lipschitz domain,
then
\begin{align}\label{holdernorm}\mathcal L ^{r^\beta , q}(\Omega)=
C^\beta(\Omega)
\end{align} for every $q \geq 1$~\cite{Cam63}.
 It is easily verified that, if $\omega$ is non-decreasing, then
\begin{equation}\label{august106}
C^\omega(\Omega) \subset  \mathcal L ^{\omega , q}(\Omega)
\end{equation}
 for $q \geq 1$. Here,
$C^\omega(\Omega)$ denotes the space of those functions $\bff :
\Omega \to \mathbb R^m$ such that
$$
\sup_{
\begin{tiny}
 \begin{array}{c}{
    x, y \in \Omega} \\
x\neq y
 \end{array}
  \end{tiny}
}\frac{|\bff (x) - \bff (y)|}{\omega (|x-y|)} < \infty\,,$$ a space
of uniformly continuous functions, with modulus of continuity not
exceeding $\omega$, if $\lim _{r \to 0^+} \omega (r) =0$.
 The reverse inclusion in \eqref{august106} need not hold for an arbitrary
$\omega$, as shown, for example, by equation \eqref{bmonorm}.
Results from \cite{Spa65} ensure that, if $\omega $ is
non-decreasing and decays at $0$ so fast that
\begin{equation}\label{dini}
\int _0 \frac{\omega (r)}r\, dr < \infty,
\end{equation}
then
\begin{equation}\label{spanne1}
\mathcal L ^\omega(\Omega) \subset  C^\varpi(\Omega)
\end{equation}
for every regular domain $\Omega$, where the function $\varpi : [0,
\infty) \to [0, \infty )$ is defined as
\begin{equation}\label{spanne2}
\varpi (r) = \int _0^r \frac{\omega (\rho)}\rho\, d\rho \qquad
\hbox{for $r \geq 0$.}
\end{equation}
On the other hand, if \eqref{dini} fails, and the function $r
\mapsto \tfrac{\omega (r)}r$ is non-increasing near $0$, then the
space $\mathcal L ^\omega(\Omega)$ is not even contained in
$L^\infty _{\rm loc} (\Omega)$.

\smallskip

 As an immediate consequence of
Theorem~\ref{thm:main2}, we have  the following regularity result
for the gradient of solutions to \eqref{eq:sysA} in Campanato
spaces.

\begin{theorem}
\label{thm:camp} {\bf [Campanato spaces]} Let $n \geq 2$, $N \geq 1$
and $p \in (1, \infty)$. Let $\Omega $ be an open set in $\Rn$, and
let $R>0$. Then there exist constants $c_1=c_1(n, N, p)>0$,
$c_2=c_2(n, N, p, R)>0$ such that, if the function $\omega (r)
r^{-\beta} $ is (almost) decreasing for some $\beta \in \big(0,
 \min\set{1, \tfrac {2\alpha}{p'}}\big)$, $\bfF \in \mathcal
L^{\omega, p'}_{\rm loc}(\Omega)$, and $B_{2R}\subset \subset
{\Omega}$, then
\begin{equation}\label{campanato1}
 \norm{|\nabla \bfu|^{p-2}\nabla \bfu}_{\mathcal L^\omega(B_R)}\leq c_1\norm{\bfF}_{\mathcal
 L^{\omega , p'}(B_{2R})}+ c_2 \|\nabla \bfu\|_{L^{p}(B_{2R})}^{p-1}
\end{equation}
 for every local weak solution  $\bfu \in W^{1,p}_{\rm loc}(\Omega)$  to
 system
 \eqref{eq:sysA}. Here,
$\alpha =\alpha (n,N,p)$ denotes the  exponent appearing in  Theorem
\ref{thm:decay}.
%
%
\end{theorem}

Bounds for H\"older norms follow from  Theorem \ref{thm:camp} and
equation \eqref{holdernorm}.

\begin{corollary}\label{holderestimate} {\bf [H\"older spaces]}
Let $n \geq 2$, $N \geq 1$ and $p \in (1, \infty)$. Let $\Omega $ be
an open set in $\Rn$, and let $B_{2R}\subset \subset {\Omega}$. Then
there exist constants $c_1=c_1(n, N, p)>0$, $c_2= c_2(n, N, p, R)>0$
such that, if $\beta <  \min\set{1,\frac{2\alpha}{p'}}$,  and $\bfF
\in C^{\beta}_{\rm loc}(\Omega)$, then
%
%
\begin{equation}\label{holderestimate1}
\norm{|\nabla \bfu|^{p-2}\nabla \bfu}_{C^\beta(B_R)}\leq
  c_1\norm{\bfF}_{C^\beta(B_{2R})} + c_2 \|\nabla \bfu\|_{L^{p}(B_{2R})}^{p-1}
\end{equation}
 for every local weak solution  $\bfu \in W^{1,p}_{\rm loc}(\Omega)$  to
 system
 \eqref{eq:sysA}.  Here,
$\alpha =\alpha (n,N,p)$ denotes the  exponent appearing in  Theorem
\ref{thm:decay}.
\end{corollary}

Corollary \ref{holderestimate} is a special case of the next result,
which can be deduced from Theorem \ref{thm:camp} and embedding
\eqref{spanne1}, and deals with estimates for more general moduli of
continuity.

\begin{corollary}\label{continuity} {\bf [Spaces of uniformly continuous functions]}
Let $n \geq 2$, $N \geq 1$ and $p \in (1, \infty)$. Let $\Omega $ be
an open set in $\Rn$, and let $R>0$. Then there exist constants
$c_1=c_1(n, N, p)$ and $c_2=c_2(n, N, p, R)$ such that, if the
function $\omega (r) r^{-\beta} $ is (almost) decreasing for some
$\beta \in \big(0, \min\set{1, \tfrac {2\alpha}{p'}}\big)$, fulfills
condition
 \eqref{dini},    $\bfF \in C^{\omega}_{\rm loc}(\Omega)$,
and $B_{2R}\subset \subset {\Omega}$, then
\begin{equation}\label{contest}
 \norm{|\nabla \bfu|^{p-2}\nabla \bfu}_{\mathcal
C^{\varpi}(B_R)}\leq c_1\norm{\bfF}_{ C^{\omega}(B_{2R})}+ c_2
\|\nabla \bfu\|_{L^{p}(B_{2R})}^{p-1}
\end{equation}
 for every local weak solution  $\bfu \in W^{1,p}_{\rm loc}(\Omega)$  to
 system
 \eqref{eq:sysA}. Here,
$\alpha =\alpha (n,N,p)$ denotes the  exponent appearing in  Theorem
\ref{thm:decay}.
\end{corollary}

\begin{remark}\label{weaker}
{\rm Note that  the norm $\norm{\bfF}_{ C^{\omega}(B_{2R})}$ on the
right-hand side of \eqref{contest} can be replaced with the
(possibly slightly weaker) norm $\norm{\bfF}_{\mathcal L^{\omega ,
p'}(B_{2R})}$.}
\end{remark}

Assumption \eqref{dini} is minimal for $\nabla \bfu$ to be
continuous, or even merely locally bounded, for every local solution
$\bfu$ and every $\bfF \in C^{\omega}_{\rm loc}(\Omega)$. This is
demonstrated by the following example, involving just the scalar
Laplace operator, which is somehow inspired by a counterexample to
  Korn's inequality in Orlicz spaces exhibited in
\cite[remark 1.4]{BrD}.

%

\begin{example}\label{sharpness}
{\rm Given any function $\omega$ as above, which violates
\eqref{dini}, namely such that
\begin{equation}\label{notdini}
\int _0 \frac{\omega (r)}r\, dr =\infty,
\end{equation}
we produce a solution $u$ to the  equation
\begin{equation}\label{poisson}
- \Delta u = - \divergence \bfF \quad \hbox{in $B_1(0)$,}
\end{equation}
with $B_1(0) \subset \mathbb R^2$, and $\bfF \in C^\omega (B_1(0))$,
but $|\nabla u | \notin L^\infty _{\rm loc}(B_1(0))$. To this
purpose, define the function $\xi : (0, 1) \to [0, \infty )$ as
$$\xi (r) = - \int _r^1 \frac{\omega (\rho)}\rho\, d\rho \quad
\hbox{for $r \in (0,1)$,}$$ and consider the function $u : B_1(0)
\to \mathbb R$ given by
$$u(x) = x_2 \xi (|x|) \quad \hbox{for $x \in B_1(0)$.}$$
One can verify that $u$ fulfils \eqref{poisson}, with
$$\bfF (x)= \Big(\frac {2x_1 x_2}{|x|^2}\omega (|x|), \frac{x_2^2 - x_1^2}{|x|^2}\omega (|x|)
\Big) \quad \hbox{for $x \neq 0$,} $$ so that
 $\bfF \in C^\omega (B_1(0))$,
but
$$\nabla u (x)= \Big( \frac {x_1 x_2}{|x|^2} \omega (|x|), \xi (|x|) +
\frac {x_2^2}{|x|^2} \omega (|x|)\Big) \quad \hbox{for $x \neq 0$,}
$$ }
and hence $ \nabla u \notin L^\infty _{\rm loc}(B_1(0))$.
\end{example}

\medskip
\par

As shown by Example \ref{sharpness}, continuity, and  mere  local
boundedness, of the gradient of solution to system \eqref{eq:sysA}
is not guaranteed when $\bfF \in \mathcal L^{\omega , p'} _{\rm
loc}(\Omega )$, or even $\bfF \in C^{\omega} _{\rm loc}(\Omega )$,
if $\omega$ does not decay at $0$ sufficiently fast to $0$ for
\eqref{dini} to be satisfied. Still, it turns out that the degree of
integrability of $|\nabla \bfu|$ is higher than that ensured from
membership of $\bfF $ just to $L^\infty _{\rm loc}(\Omega)$ (see
Proposition \ref{linf}). This assertion can be precisely formulated
in terms of Marcinkiewicz
 quasi-norms, also called weak-type quasi-norms. Recall that, given a non-decreasing
 function $\eta : [0, \infty ) \to [0, \infty )$,  the Marcinkiewicz
 functional $\|\cdot \|_{M^\eta (0, \infty )}$ is defined as
 $$\|f \|_{M^\eta (0, \infty )} = \sup _{s > 0} \eta (s) f^*(s)$$
 for a measurable function $f : (0, \infty) \to \mathbb R$.  Clearly, $\|\cdot \|_{M^\eta
 (0, \infty )}$ is a rearrangement-invariant functional.
 \\
 Owing to \cite[Theorem 1]{Spa65}, if $\Omega$ is a bounded Lipschitz domain, $R_0 > {\rm diam}(\Omega)$, and $\zeta$ is given by
 \begin{equation}\label{zeta}\zeta (r) = \bigg(\int _{r^{\frac 1n}}^{R_0^{\frac 1n}} \frac{\omega
 (\rho)}{\rho} \, d\rho\bigg)^{-1} \quad \hbox{for $r \in (0,
 R_0)$,}
 \end{equation}
 then
 \begin{equation}\label{embmarc}
 \mathcal L^\omega (\Omega ) \to M^{\zeta}(\Omega).
 \end{equation}
 The following result is a straightforward consequence of  Theorem~\ref{thm:camp} and embedding
 \eqref{embmarc}. In what follows, $\zeta _p$ denotes the function
 given by
 $$\zeta _p (r) = \zeta ^{\frac 1{p-1}}(r) \quad \hbox{for $r \in (0,
 R_0)$.}
$$

\begin{corollary}\label{marc} {\bf [Borderline Marcinkiewicz  spaces]} Let $n \geq 2$, $N \geq 1$ and $p
\in (1, \infty)$. Let $\Omega $ be an open set in $\Rn$,  let $R>0$,
and let $R_0 > R$. Then there exist constants $c_1=c_1(n, N, p, R,
R_0)>0$ and $c_2=c_2(n, N, p, R, R_0)>0$ such that, if the function
$\omega (r) r^{-\beta} $ is (almost) decreasing for some $\beta \in
\big(0,  \min\set{1,\frac{2\alpha}{p'}}\big)$, $\bfF \in \mathcal
L^{\omega , p'}_{\rm loc}(\Omega)$, and $B_{2R}\subset \subset
{\Omega}$, then
\begin{equation}\label{marc1}
\norm{\nabla \bfu}_{\mathcal M^{\zeta
_p}(B_R)}\leq c_1\norm{\bfF}_{ \mathcal L^{\omega ,
p'}(B_{2R})}^{\frac 1{p-1}}+ c_2 \|\nabla
\bfu\|_{L^{p'}(B_{2R})}^{\frac 1{p-1}}
\end{equation}
 for every local weak solution  $\bfu \in W^{1,p}_{\rm loc}(\Omega)$  to
 system
 \eqref{eq:sysA}. Here,
$\alpha =\alpha (n,N,p)$ denotes the  exponent appearing in  Theorem
\ref{thm:decay}.
%
%
\end{corollary}


A local version of Proposition \ref{linf} tells us that, if $\bfF
\in L^\infty _{\rm loc}(\Omega)$, then $|\nabla \bfu| \in \exp
L^{p-1}_{\rm loc}(\Omega)$. Since  $L^\infty _{\rm loc}(\Omega)$ is
the smallest (local) rearrangement-invariant space, this is the
strongest integrability information on $|\nabla \bfu|$ which follows
from membership of $\bfF$ to a rearrangement-invariant space. As
observed above, Corollary \ref{marc} complements this results, and
ensures that, if $\bfF$ belongs to a smaller space than $L^\infty
_{\rm loc}(\Omega)$, namely a space of (locally) uniformly
continuous functions $C^\omega _{\rm loc}(\Omega)$, or  just to the
Campanato type space $\mathcal L^{\omega , p'}_{\rm loc}(\Omega)$,
then, even if $\omega$ does not fulfil \eqref{dini}, yet $|\nabla
\bfu|$ belongs to the better (smaller) rearrangement-invariant space
$ M^{\zeta _p} _{\rm loc}(\Omega)$ than $\exp L^{p-1}_{\rm
loc}(\Omega)$, as soon as \begin{equation}\label{limomega} \lim _{r
\to 0} \omega (r) =0.
\end{equation}
 To verify this fact, note that the latter limit ensures that
\begin{equation}\label{limit}
\lim _{r \to 0} \frac {\zeta _p (r)}{\log ^{\frac 1{p-1}}(\frac 1r)}
=0.
\end{equation}
On the other hand, it is well known (and easily verified) that
$$\exp L^{p-1}_{\rm loc}(\Omega) = M ^{\log ^{\frac 1{p-1}}(\frac
1r)}_{\rm loc} (\Omega).$$
 Equation \eqref{limit} thus implies that
 \begin{equation}\label{inclusion}
 M^{\zeta _p} _{\rm loc}(\Omega) \subsetneqq  \exp L^{p-1}_{\rm loc}(\Omega).
 \end{equation}

\begin{example}\label{logalpha}
{\rm Assume that $\bfF \in C ^{\log ^{-\sigma }(\frac 1r)}_{\rm
loc}(\Omega)$ for some $ \sigma >0$.   If $\sigma >1$, then
Corollary \ref{continuity}  entails that
$$|\nabla \bfu|^{p-2} \nabla \bfu \in C^{\log ^{1-\sigma }(\frac 1r)}_{\rm loc}(\Omega).$$
\\ If $0 < \sigma < 1$ we instead have, via Corollary \ref{marc},
$$|\nabla \bfu| \in \exp L^{\frac {p-1}{1-\sigma}}_{\rm loc}(\Omega).$$
\\ In the borderline case when $\sigma =1$,  Corollary
\ref{marc} again tells us that
$$|\nabla \bfu| \in \exp \exp L^{p-1}_{\rm loc}(\Omega).$$
By  Remark \ref{weaker}, the same conclusions hold even under the
slightly weaker assumption that $\bfF \in \mathcal L ^{\log
^{-\sigma }(\frac 1r)}_{\rm loc}(\Omega)$. }
\end{example}

Assume now that \eqref{limomega} fails, namely that $\lim _{r \to 0}
\omega (r)
>0$.  Then $\mathcal L^{ \omega , p'}_{\rm loc}(\Omega) = \setBMO _{\rm
loc}(\Omega)$, and Corollary \ref{marc} just yields $|\nabla \bfu|
\in \exp L^{p-1}_{\rm loc}(\Omega)$. In other words, no difference
seems to be reflected in the integrability of $|\nabla \bfu|$,
depending on whether
  $\bfF \in \setBMO _{\rm loc}(\Omega)$ or  $\bfF \in
L^\infty _{\rm loc}(\Omega)$. However, a difference appears in the
scale of oscillation spaces, since, under the former assumption, the
$\setBMO$ estimate  \eqref{BMOrn}, even in local form, can be
recovered as a special case of Theorem \ref{thm:camp},
owing to  \eqref{bmonorm}.

\begin{corollary}\label{BMOest} {\bf [$\setBMO$]}
Let $n \geq 2$, $N \geq 1$ and $p \in (1, \infty)$. Let $\Omega $ be
an open set in $\Rn$, and let $R>0$. Then there exist constants
$c_1=c_1(n, N, p)>0$, $c_2=c_2(n, N, p, R)>0$ such that, if  $\bfF
\in \setBMO_{\rm loc}(\Omega)$, and $B_{2R}\subset \subset
{\Omega}$, then
\begin{equation}\label{BMOest1}
\norm{|\nabla \bfu|^{p-2}\nabla \bfu}_{\setBMO
(B_R)}\leq c_1\norm{\bfF}_{ \setBMO (B_{2R})}+ c_2 \|\nabla
\bfu\|_{L^{p'}(B_{2R})}
\end{equation}
 for every local weak solution  $\bfu \in W^{1,p}_{\rm loc}(\Omega)$  to
 system
 \eqref{eq:sysA}.
\end{corollary}

Theorem \ref{thm:camp} can  be used to show that \emph{vanishing
mean oscillation} ($\setVMO$) regularity of the datum $\bfF$ is also
reflected into the same regularity of $|\nabla \bfu|^{p-2}\nabla
\bfu$. Recall that a locally integrable function $\bff : \Omega \to
\mathbb R^m$ is said to belong to $\setVMO (\Omega)$ if
\begin{equation}\label{VMO} \lim
_{\varrho \to 0^+} \bigg(\sup _{\begin{tiny}
 \begin{array}{c}{
    B_r \subset \Omega} \\
                        r\leq \varrho
                             \end{array}
                              \end{tiny}
} \dashint _{B_r} |\bff - \mean{\bff}_{B_r}|\,
  dx\bigg) =0.
  \end{equation}
Clearly, $\setVMO (\Omega) \subset \setBMO(\Omega)$.
  We shall write $\bff \in  \setVMO _{\rm loc}
(\Omega)$ to denote that equation holds with $\Omega$ replaced by
any open set $\Omega ' \subset \subset \Omega$.

\begin{corollary} {\bf [$\setVMO$]} Let $n \geq 2$, $N \geq 1$ and $p \in (1, \infty)$. Let $\Omega $
be an open set in $\Rn$, and let $R>0$. If $\bfF\in \setVMO _{\rm
loc} (\Omega)$, then $|\nabla \bfu|^{p-2}\nabla \bfu \in \setVMO
_{\rm loc} (\Omega)$ for every local weak solution $\bfu \in
W^{1,p}_{\rm loc}(\Omega)$  to
 system
 \eqref{eq:sysA}.
\end{corollary}
\begin{proof}
Let $\Omega '$ be an open set such
$\Omega ' \subset \subset \Omega $. Fix any $\gamma \in (0, 1)$. By
H\"older's inequality
\begin{align}\label{august101}
\sup _{\begin{tiny}
 \begin{array}{c}{
    B_r \subset \Omega'} \\
                        r\leq \varrho
                             \end{array}
                              \end{tiny}
}&  \bigg(\dashint _{B_r} |\bfF - \mean{\bfF}_{B_r}|^{p'}\,
  dx\bigg)^{\frac 1{p'}}
\\ \nonumber  & \leq
\sup _{\begin{tiny}
 \begin{array}{c}{
    B_r \subset \Omega'}
    \end{array}
                              \end{tiny}
} \bigg(\dashint _{B_r} |\bfF - \mean{\bfF}_{B_r}|^{\frac{p' -
\gamma}{1-\gamma}}\,
  dx\bigg)^{\frac {1-\gamma} {p'}} \sup _{\begin{tiny}
 \begin{array}{c}{
    B_r \subset \Omega'} \\
                        r\leq \varrho
                             \end{array}
                              \end{tiny}
} \bigg(\dashint _{B_r} |\bfF - \mean{\bfF}_{B_r}|\,
  dx\bigg)^{\frac \gamma {p'}}
\\ \nonumber & =
\| \bfF\|_{\mathcal L^{\frac{p' - \gamma}{1-\gamma}}(\Omega
')}^{\frac {p'-\gamma}{p'}}\sup _{\begin{tiny}
 \begin{array}{c}{
    B_r \subset \Omega'} \\
                        r\leq \varrho
                             \end{array}
                              \end{tiny}
} \bigg(\dashint _{B_r} |\bfF - \mean{\bfF}_{B_r}|\,
  dx\bigg)^{\frac \gamma {p'}}
%
%
%
%
%
%
\quad \hbox {for $\varrho >0$.}
  \end{align}
Denote by  $\sigma : (0, \infty ) \to (0, \infty]$ the function
defined for $\varrho \in (0, \infty )$ by the leftmost  side of
equation \eqref{august101}. Owing to the assumption that $\bfF\in
\setVMO _{\rm loc} (\Omega)$, to  equations \eqref{bmonorm} and
\eqref{august101}, one has that $\lim _{\varrho \to 0^+} \sigma
(\varrho)=0$. Next, given any exponent $\beta$ as in Theorem
\ref{thm:camp}, define the function $\omega : (0, \infty ) \to (0,
\infty)$ as
\begin{equation}\label{august102}
\omega(r) =r^\beta\sup_{\varrho\geq
r}\frac{\set{\sigma (\varrho)}}{\varrho^\beta} \quad \hbox{for $r
>0$.}
\end{equation}
Obviously, the function $r \mapsto \frac {\omega(r)}{r^\beta}$ is
non-increasing, and one can verify that  $\lim _{r \to 0^+} \omega
(\varrho)=0$. Since $\omega \geq \sigma$, equation \eqref{august101}
ensures that $\bfF \in \mathcal L^{\omega , p'}_{\rm loc}(\Omega)$.
An application of Theorem \ref{thm:camp} tells us that $|\nabla
\bfu|^{p-2}\nabla \bfu \in \mathcal L^{\omega , p'}_{\rm
loc}(\Omega)$ as well. In particular, this entails that $|\nabla
\bfu|^{p-2}\nabla \bfu \in \setVMO _{\rm loc} (\Omega)$.
\end{proof}


\begin{thebibliography}{10}

\bibitem{AceF89}
E.~Acerbi and N.~Fusco,
\newblock Regularity for minimizers of nonquadratic functionals: the case
  {$1<p<2$,}
\newblock {\em J. Math. Anal. Appl.} 140  (1989), 115--135.

\bibitem{AdiPhuc}
K.~Adimurthi and N.~Phuc,
\newblock Global {L}orentz and {L}orentz-{M}orrey estimates below the natural
  exponent for quasilinear equations.
\newblock {\em Calc. Var. and PDEs}, to appear,
\newblock DOI 10.1007/s00526-015-0895-1.


\bibitem{AFT}
A.Alvino, V.Ferone and G.Trombetti,
\newblock {Estimates for the gradient of solutions of nonlinear elliptic
equations with $L^1$ data,}
\newblock {\em Ann. Mat. Pura Appl.} 178  (2000), 129--142.

\bibitem{ACMM}
A.Alvino, A.Cianchi, V.Maz'ya and A.Mercaldo
\newblock { Well-posed elliptic {N}eumann problems involving irregular data and domains,}
\newblock {\em Ann. Inst. H. Poincar\'e Anal. Non Lin\'eaire} 27  (2010), 1017–1054.



\bibitem{BanLew14}
A.~Banerjee and J.~Lewis,
\newblock Gradient bounds for p-harmonic systems with vanishing {N}eumann
  ({D}irichlet) data in a convex domain,
\newblock {\em Nonlinear Anal.} 100 (2014), 78--85.

\bibitem{BS}
C.~Bennett and R.~Sharpley,
\newblock {\em Interpolation of operators},
\newblock Academic Press, 1988.

\bibitem{Du11}
V.~B\"ogelein, F.~Duzaar, J.~Habermann, and C.~Scheven,
\newblock Partial {H}\"older continuity for discontinuous elliptic problems
  with {VMO}-coefficients,
\newblock {\em Proc. London Math. Soc.} 103 (2011), 1215--1240.

\bibitem{BrD}
D.~Breit and L.~Diening,
\newblock Sharp conditions for {K}orn inequalities in {O}rlicz
spaces,
\newblock {\em J. Math. Fluid Mech.} 14 (2012), 565--573.




\bibitem{BreStrVer11}
D.~Breit, B.~Stroffolini, and A.~Verde,
\newblock A general regularity theorem for functionals with
$\varphi$-growth,
\newblock {\em J. Math. Anal. Appl.} 383 (2011), 226--233.

\bibitem{Cam63}
S.~Campanato,
\newblock Propriet\`a di h\"olderianit\`a di alcune classi di
funzioni,
\newblock {\em Ann. Scuola Norm. Sup. Pisa} 17 (1963), 175--188.

\bibitem{CPSS}
M.~Carro, L.~Pick, J.~Soria, and V.~D. Stepanov,
\newblock On embeddings between classical {L}orentz spaces,
\newblock {\em Math. Inequal. Appl.} 4 (2001), 397--428.

\bibitem{ChenDiBe}
Y.~Z. Chen and E.~DiBenedetto,
\newblock Boundary estimates for solutions of nonlinear degenerate parabolic
  systems,
\newblock {\em J. Reine Angew. Math.} 395 (1989), 102--131.

\bibitem{Cianchistrong}
A.~Cianchi,
\newblock Strong and weak type inequalities for some classical operators in
  {O}rlicz spaces,
\newblock {\em J. London Math. Soc.}, 60 (1999), 187--202.

\bibitem{CiaFus}
A.~Cianchi and N.Fusco,
\newblock Gradient regularity for minimizers under general growth conditions,
\newblock {\em J. Reine Angew. Math.}  507 (1999), 15--36.



\bibitem{CiaMaz11}
A.~Cianchi and V.~Maz'ya,
\newblock Global {L}ipschitz regularity for a class of quasilinear elliptic
  equations,
\newblock {\em Comm. Part. Diff. Equat.}  36 (2011), 100--133.

\bibitem{CiaMaz14syst}
A.~Cianchi and V.~Maz'ya,
\newblock Global boundedness of the gradient for a class of nonlinear elliptic
  systems,
\newblock {\em Arch. Ration. Mech. Anal.}, 212 (2014), 129--177.

\bibitem{CiaMaz14rearr}
A.~Cianchi and V.~Maz'ya,
\newblock Gradient regularity via rearrangements for p-{L}aplacian type
  elliptic boundary value problems,
\newblock {\em J. Eur. Math. Soc. (JEMS)} 16 (2014), 571--595.

\bibitem{DiBMan93}
E.~DiBenedetto and J.~Manfredi,
\newblock On the higher integrability of the gradient of weak solutions of
  certain degenerate elliptic systems,
\newblock {\em Amer. J. Math.} 115 (1993), 1107--1134.

\bibitem{DieE08}
L.~Diening and F.~Ettwein,
\newblock Fractional estimates for non-differentiable elliptic systems with
  general growth,
\newblock {\em Forum Mathematicum}, 20 (2008), 523--556.

\bibitem{DieKapSch11}
L.~Diening, P.~Kaplick{\'y}, and S.~Schwarzacher,
\newblock B{MO} estimates for the {$p$}-{L}aplacian.
\newblock {\em Nonlinear Anal.} 75 (2012), 637--650.

\bibitem{DieK08}
L.~Diening and C.~Kreuzer,
\newblock Linear convergence of an adaptive finite element method for the
  $p$-{L}aplacian equation,
\newblock {\em SIAM J. Numer. Anal.} 46 (2008), 614--638.

\bibitem{DieSV09}
L.~Diening, B.~Stroffolini, and A.~Verde,
\newblock Everywhere regularity of functionals with {$\phi$}-growth,
\newblock {\em Manuscripta Math.} 129 (2009), 449--481.

\bibitem{DolM95}
G.~Dolzmann and S.~M{\"u}ller,
\newblock Estimates for {G}reen's matrices of elliptic systems by {$L^p$}
  theory,
\newblock {\em Manuscripta Math.} 88 (1995), 261--273.

\bibitem{Du02}
F.~Duzaar and A.~Gastel,
\newblock Nonlinear elliptic systems with {D}ini continuous
coefficients,
\newblock {\em Arch. Math.} 78 (2002), 58--73.

\bibitem{Du04}
F.~Duzaar, A.~Gastel, and G.~Mingione,
\newblock Elliptic systems, singular sets and dini continuity,
\newblock {\em Comm. Part. Diff. Equat.} 29 (2004), 371--404.

\bibitem{DuzMin10a}
F.~Duzaar and G.~Mingione,
\newblock Gradient continuity estimates,
\newblock {\em Calc. Var. and PDEs} 39 (2010), 379--418.

\bibitem{DuzMing2011}
F.~Duzaar and G.~Mingione,
\newblock Gradient estimates via non-linear potentials,
\newblock {\em Amer. J. Math.} 133 (2011), 1093--1149.

\bibitem{Fu11}
M.~Fuchs,
\newblock Local lipschitz regularity of vector valued local minimizers of
  variational integrals with densities depending on the modulus of the
  gradiend,
\newblock {\em Math. Nachr.} 284 (2011), 266--272.

\bibitem{GIM}
L.~Greco, T.~Iwaniec, and G.~Moscariello,
\newblock Gioconda limits of the improved integrability of the volume
forms,
\newblock {\em Indiana Univ. Math. J.} 44 (1995), 305--339.

\bibitem{Nec96}
W.~Hao, S.~Leonardi, and J.~Ne\v{c}as,
\newblock An example of irregular solution to a nonlinear {E}uler-{L}agrange
  elliptic system with real analytic coefficients,
\newblock {\em Ann. Sc. Norm. Super. Pisa }   23 (1996), 57--67.

\bibitem{Iwa83}
T.~Iwaniec.
\newblock Projections onto gradient fields and {$L^{p}$}-estimates for
  degenerated elliptic operators
\newblock {\em Studia Math.} 75 (1983), 293--312.

\bibitem{IwaMan89}
T.~Iwaniec and J.~J. Manfredi,
\newblock Regularity of {$p$}-harmonic functions on the plane,
\newblock {\em Rev. Mat. Iberoamericana}  5 (1989), 1--19.

\bibitem{JohNir61}
F.~John and L.~Nirenberg,
\newblock On functions of bounded mean oscillation,
\newblock {\em Comm. Pure. Appl. Math.} 14 (1961), 415--426.

\bibitem{KilpMaly}
T.~Kilpel\"ainen and J.~Maly,
\newblock The {W}iener test and potential estimates for quasilinear elliptic
  equations,
\newblock {\em Acta Math.} 172 (1994), 137--161.

\bibitem{KinZho99}
J.~Kinnunen and S.~Zhou,
\newblock A local estimate for nonlinear equations with discontinuous
  coefficients,
\newblock {\em Comm. Part. Diff. Equat.} 24 (1999), 2043--2068.

\bibitem{KinZho01}
J.~Kinnunen and S.~Zhou,
\newblock A boundary estimate for nonlinear equations with discontinuous
  coefficients,
\newblock {\em Diff. Int. Equat.} 14(2001), 475--492.

\bibitem{Kita}
H.~Kita,
\newblock On maximal functions in {O}rlicz spaces,
\newblock {\em Proc. Amer. Math. Soc.} 124 (1996), 3019--3025.

\bibitem{KuuMin12univ}
T.~Kuusi and G.~Mingione,
\newblock Universal potential estimates,
\newblock {\em J. Funct. Anal.} 262 (2012), 4205--4269.

\bibitem{KuuMin13}
T.~Kuusi and G.~Mingione,
\newblock Linear potentials in nonlinear potential theory,
\newblock {\em Arch. Ration. Mech. Anal.} 207 (2013), 215--246.

\bibitem{KuuMin14}
T.~Kuusi and G.~Mingione,
\newblock A nonlinear {S}tein theorem,
\newblock {\em Calc. Var. and PDE's} 51 (2014), 45--86.

\bibitem{Lew08}
J.~K. Lewis and K.~Nystr\"om,
\newblock Boundary behaviour of p-harmonic functions in domains beyond
  {L}ipschitz domains.
\newblock {\em Adv. Calc. Var.} 1 (2008), 133--170.

\bibitem{Lew12}
J.~K. Lewis and K.~Nystr\"om,
\newblock Regularity and free boundary regularity for the p-{L}aplace operator in
  {R}eifenberg flat and {A}hlfors regular domains,
\newblock {\em J. Amer. Math. Soc.} 25 (2012), 827--862.

\bibitem{Li94}
G.~M. Lieberman,
\newblock Gradient estimates for a new class of degenerate elliptic and
  parabolic equations,
\newblock {\em Ann. Scuola Norm. Sup. Pisa} 21 (1994), 497--522.

\bibitem{Li05}
G.~M. Lieberman,
\newblock Gradient estimates for anisotropic elliptic equations,
\newblock {\em Adv. Diff. Equat.} 10 (2005), 767--812.

\bibitem{Li14}
G.~M. Lieberman,
\newblock Gradient estimates for singular fully nonlinear elliptic
equations,
\newblock {\em Nonlinear Anal.} 119 (2014), 382--397.

\bibitem{MP}
L.~Maligranda and L.~Persson,
\newblock Generalized duality of some {B}anach function spaces,
\newblock {\em Akad. Wetensch. Indag. Math.} 51 (1989), 323--338.

\bibitem{MarPap06}
P.~Marcellini and G.~Papi,
\newblock Nonlinear elliptic systems with general growth,
\newblock {\em J. Diff. Equat.} 221 (2006), 412--443.

\bibitem{Maz11}
V.~Maz'ya,
\newblock {\em Sobolev spaces with applications to elliptic partial
  differential equations},
\newblock Springer, Heidelberg, augmented edition, 2011.

\bibitem{Mingione2011}
G.~Mingione,
\newblock Gradient potential estimates,
\newblock {\em J. Eur. Math. Soc.} 13(2011), 459--486.

\bibitem{Nec75}
J.~Ne\v{c}as,
\newblock {\em Example of an irregular solution to a nonlinear elliptic system
  with analytic coefficients and conditions for regularity},
\newblock Seminario di Matematica della Scuola Normale Superiore di Pisa,
  1960-61. In: Theor. Nonlin. Oper., Constr. Aspects., Proc. 4th Int. Summer
  School. Akademie-Verlag, Berlin, 1975.

\bibitem{Phuc14}
N.~Phuc,
\newblock Morrey global bounds and quasilinear {R}iccati type equations below
  the natural exponent,
\newblock {\em J. Math. Pures Appl.} 102 (2014), 99--123.

\bibitem{Sch13}
S.~Schwarzacher,
\newblock {H}\"{o}lder-{Z}ygmund estimates for degenerate parabolic
systems,
\newblock {\em J. Diff. Equat.} 256 (2014), 2423--2448.

\bibitem{Spa65}
S.~Spanne,
\newblock Some function spaces defined using the mean oscillation over
cubes,
\newblock {\em Ann. Scuola Norm. Sup. Pisa} 19 (1963), 593--608.

\bibitem{Uhl77}
K.~Uhlenbeck,
\newblock Regularity for a class of non-linear elliptic systems,
\newblock {\em Acta Math.} 138 (1977), 219--240.

\bibitem{Ura68}
N.~N. Ural'ceva,
\newblock Degenerate quasilinear elliptic systems,
\newblock {\em Zap. Nau\v cn. Sem. Leningrad. Otdel. Mat. Inst. Steklov.
  (LOMI)} 7 (1968), 184--222.

\bibitem{SVYa00}
V.~\v{S}ver\'ak and X.~Yan.
\newblock A singular minimizer of a smooth strongly convex functional in three
  dimensions,
\newblock {\em Calc. Var. and PDEs} 10 (2000), 213--221.

\bibitem{SVYa02}
V.~\v{S}ver\'ak and X.~Yan,
\newblock Non-{L}ipschitz minimizers of smooth uniformly convex variational
  integrals,
\newblock {\em Proc. Natl. Acad. Sci. USA} 99 (2002), 15269--15276.

\end{thebibliography}
\end{document}